\documentclass[11pt,oneside,reqno]{amsart}

\usepackage{amsmath,ifthen, amsfonts, amssymb,
srcltx, amsopn, color, enumerate} 
\usepackage[linktocpage=true]{hyperref}
\usepackage[cmtip,arrow]{xy}
\usepackage{pb-diagram, pb-xy}
\usepackage{overpic}
\usepackage{microtype}
\usepackage{mathabx}
\usepackage{tikz}
\usetikzlibrary{decorations.markings,cd}
\usetikzlibrary{decorations.pathreplacing}
\usetikzlibrary{arrows}
\usetikzlibrary{shapes.misc}
\usepackage{pifont}
\usepackage[normalem]{ulem}

\usepackage{tablefootnote}

\newtheorem{thm}{Theorem}[section]
\newtheorem{lem}[thm]{Lemma}

\newtheorem{cor}[thm]{Corollary}

\newtheorem{prop}[thm]{Proposition}

\newtheorem{thmi}{Theorem}
\newtheorem{cori}[thmi]{Corollary}

\newtheorem{ques}{Question} 
\newtheorem{propappend}{Proposition}

\theoremstyle{definition}
\newtheorem{defnappend}[propappend]{Definition}
\newtheorem{defn}[thm]{Definition}
\newtheorem{rem}[thm]{Remark}

\newtheorem{exmp}[thm]{Example}

\newtheorem{notation}[thm]{Notation}

\newtheorem{claim}{Claim}

\newtheorem{claim*}{Claim}

\DeclareMathOperator{\diam}{diam}

\DeclareMathOperator{\hull}{hull}

\newcommand{\field}[1]{\mathbb{#1}}

\makeatletter

\newcommand{\Rmnum}[1]{\mathbf{{\expandafter\@slowromancap\romannumeral #1@}}}

\makeatother

\newcommand{\tsh}[1]{\left\{\kern-.7ex\left\{#1\right\}\kern-.7ex\right\}}
\newcommand{\Tsh}[2]{\tsh{#2}_{#1}}
\newcommand{\ignore}[2]{\Tsh{#2}{#1}}

\newcommand{\co}{\colon}

\usepackage{calc}
\usepackage{array}

\newcommand*\mc[1]{\mathcal{#1}}
\newcommand*\mf[1]{\mathfrak{#1}}
\newcommand\nest{\sqsubseteq}
\newcommand\trans{\pitchfork}
\newcommand{\relevant}{\mathbf{Rel}}

\newcommand{\dist}{d}
\newcommand{\propnest}{\sqsubsetneq}

\newcommand{\Path}{\mathcal{P}}
\usepackage{multicol}
\usepackage{ulem}
\usepackage{float}
\usepackage{subcaption}

\newcommand*{\sqc}{strongly quasiconvex}
\newcommand*{\sqcity}{strong quasiconvexity}
\newcommand*{\SQC}{Strongly quasiconvex}
\newcommand*{\SQCity}{Strong quasiconvexity}

\def\X{\mathcal{X}}

\def\PF{\mathbf{F}}
\def\PE{\mathbf{E}}
\def\P{\mathbf{P}}
\def\gate{\mathfrak{g}}
\def\myparagraph{\textbf}
\def\cpproj{\mathfrak{p}}
\def\proj{\tau}

\newcommand*\notinfinal[1]{}

\makeindex

\title[Convexity in HHS]{ 
Convexity in Hierarchically Hyperbolic Spaces}
\author{Jacob Russell}
\author{Davide Spriano}
\author{Hung Cong Tran}
\date{\today}
\begin{document}
\begin{abstract}

Hierarchically hyperbolic spaces (HHSs) are a large class of spaces that provide a unified framework for studying the mapping class group, right-angled Artin and Coxeter groups, and many 3--manifold groups. We investigate {{\sqc}} subsets in this class and characterize them in terms of their contracting properties, relative divergence, the coarse median structure, and the hierarchical structure itself. Along the way, we obtain new tools to study HHSs, including two new equivalent definitions of hierarchical quasiconvexity and a version of the bounded geodesic image property for {\sqc} subsets. Utilizing our characterization, we prove that the hyperbolically embedded subgroups of hierarchically hyperbolic groups are precisely those that are almost malnormal and {\sqc}, producing a new result in the case of the mapping class group.  We also apply our characterization to study {\sqc} subsets in several specific examples of HHSs. We show that while many commonly studied HHSs have the property that that every {\sqc} subset is either hyperbolic or coarsely covers the entire space, right-angled Coxeter groups exhibit a wide variety of {\sqc} subsets.

\end{abstract}

\maketitle

\tableofcontents

\section{Introduction}

From Gromov's original work on hyperbolic groups to the resolution of the virtual Haken conjecture, quasiconvex subsets have played a central role in the study of hyperbolic metric spaces  and groups \cite{Gromov1,Gromov2,Riches_to_Raags, Agol_virt_haken}. A subset $Y$ is \emph{quasiconvex} if every geodesic based on $Y$ in contained in a fixed neighborhood of $Y$. A central feature of quasiconvex subsets of hyperbolic spaces is their quasi-isometry invariance, i.e., the image of a quasiconvex subset of a hyperbolic space under a quasi-isometry is quasiconvex.

Outside of hyperbolic spaces, quasiconvexity  fails to be a quasi-isometry invariant. However, a strengthening of this definition to require ``quasiconvexity with respect to quasi-geodesics" and not just geodesics is sufficient to ensure quasi-isometry invariance. A subset $Y$ of a quasi-geodesic metric space $X$ is \emph{{\sqc}} if every quasi-geodesic based on $Y$ is contained in a bounded neighborhood of $Y$, where the radius of the neighborhood is determined by the quasi-geodesic constants. {\SQCity} provides a ``coarse-ification" of the classical definition of a convex subset that ensures that the image of a {\sqc} subset under a quasi-isometry will be {\sqc}, regardless of the geometry of the ambient space. {\SQC} subsets are therefore an avenue to study the geometry of any space up to quasi-isometry.

The study of {\sqc} geodesics in non-hyperbolic spaces (often called \emph{Morse geodesics}) has been a vibrant and fruitful area of research over the last decade (for example \cite{CS_Boundary,ACGH,OOS_Lacunary,DMS_divergence}).  Recently, considerable interest has arisen  in understanding general {\sqc} subsets in non-hyperbolic spaces.

The third author studied {\sqc} subsets and subgroups in \cite{Tran2017} and showed that many important properties of quasiconvex subsets in hyperbolic spaces persist for {\sqc} subsets of any geodesic metric space. These result have found applications in understanding the cell stabilizers of groups acting on CAT(0) cube complexes \cite{Manning_Groves_Non-proper} and the splittings of groups over co-dimension 1 subgroups \cite{Petrosyan_decomposing_groups}.  Using the name Morse instead of {\sqc}, Genevois studied  {\sqc} subsets of CAT(0) cube complexes in \cite{Genevois_Hyp_in_CAT(0)} and Kim studied {\sqc}   subgroups of the mapping class groups in \cite{Kim2017}. {\SQC} subgroup that are also hyperbolic were introduced by Durham-Taylor as \emph{stable} subgroups \cite{DurhamTaylorStability} and have received considerable study; for a sampling see \cite{KMT,MR3742433,AMST,B,ABD}.

In this paper, we are primarily interested in understanding the {\sqc} subsets of hierarchically hyperbolic spaces (HHSs). Introduced by Behrstock-Hagen-Sisto in \cite{BHS} and refined in \cite{BHS17b}, examples of hierarchically hyperbolic spaces include hyperbolic spaces, the mapping class group of a surface, Teichm\"uller space with either the Weil-Petersson or Teichm\"uller metrics, many cocompactly cubulated groups, and the fundamental groups of $3$--manifolds without Nil or Sol components. Important consequences of hierarchical hyperbolicity include a Masur-Minsky style distance formula \cite{BHS17b}, a quadratic isoperimetric inequality \cite{BHS17b}, restrictions on morphisms from higher rank lattices \cite{Haettel_higher_rank_lattices}, a largest acylindrical action on a hyperbolic space \cite{ABD}, rank-rigidity and Tits alternative theorems \cite{DHS_HHS_Boundary},   control over the top-dimensional quasi-flats \cite{BHS_HHS_Quasiflats}, and bounds on the asymptotic dimension \cite{BHSAsyDim}. The definition and much of the theory of  hierarchically hyperbolic  spaces is inspired by the Masur-Minsky subsurface projection machinery for the mapping class group. Our investigation is therefore a natural extension of the problem purposed by Farb in \cite[Problem 2.3.8]{FarbProblems} to study convexity in the mapping class group.

Heuristically, a hierarchically hyperbolic space consists of a metric space ${X}$ with an associated collection of hyperbolic spaces $\mf{S}$, such that for each space $Z$ in $\mf{S}$, there is a projection map $ X \rightarrow Z$.  The philosophy of hierarchically hyperbolic spaces is that one can study the coarse geometry of $X$ by studying the projection of $X$ to each of the spaces in $\mf{S}$.  In this paper, we shall consider hierarchically hyperbolic spaces satisfying the \emph{bounded domain dichotomy}: a minor regularity condition requiring every space in $\mf{S}$ to have either infinite or uniformly bounded diameter. The bounded domain dichotomy simplifies the statements and proofs of our results while being satisfied by all of the examples of hierarchically hyperbolic spaces given above and more broadly by all hierarchically hyperbolic groups.\\

\noindent\textbf{Equivalent conditions to being {\sqc}.} The main goal of this paper is to provide several equivalent conditions for a subset of a hierarchically hyperbolic space to be {\sqc}.  A major theme is that several different notions of convexity that coincide with being quasiconvex in a hyperbolic space, coincide with being {\sqc} in a hierarchically hyperbolic spaces. One such notion of convexity is that of contracting subsets. A subset $Y\subseteq  X$ of a quasi-geodesic space is \emph{contracting}   if there exists a coarsely Lipschitz  retraction  $r\colon  X \rightarrow Y$ under which large balls far from $Y$ have images with uniformly bounded diameter. 
Being contracting generalizes the strong contracting behavior of the closest point projection onto a convex subset of the hyperbolic plane. In general, {\sqc} subsets are not contracting (see Example \ref{ex:quasiconvex_not_contracting}), however these two notions of convexity tend to agree in the presence of non-positive curvature. Indeed, it is a classical fact that a subset of a hyperbolic space is {\sqc} if and only if it is contracting; the same is true for subsets of a CAT(0) cube complex  \cite{Genevois_Hyp_in_CAT(0),CS_Boundary}. The first of our equivalent condition is to extend these results to hierarchically hyperbolic spaces.

\begin{thmi}[{\SQC} and contracting are equivalent]\label{intro:contracting_iff_quasiconvex}
Let ${X}$ be a hierarchically hyperbolic space with the bounded domain dichotomy.  A subset $Y \subseteq {X}$ is {\sqc} if and only if $Y$ is contracting.
\end{thmi}

In \cite{ACGH}, a different notion of contracting subset is considered, and it is shown that a subset of a geodesic metric space is {\sqc} if and only if the subset is \textit{sublinearly contracting}. Example \ref{ex:quasiconvex_not_contracting} demonstrates that our definition of contracting (Definition \ref{defn:contracting}) is strictly stronger than sublinear contracting, but the two notions agree in the setting of hierarchically hyperbolic spaces. Another key difference between our definition of contracting and that in \cite{ACGH} is that we do not require the contracting map $r \colon {X} \rightarrow Y$ to be the closest point projection, but allow for any coarsely Lipschitz retraction that has the contracting property. This has the advantage of turning contracting into a quasi-isometry invariant directly from the definition and is crucial in allowing us to utilize a naturally occurring retraction map in hierarchically hyperbolic spaces that is far more tractable than the closest point projection.

The third notion of convexity considered is \emph{hierarchical quasiconvexity}, which is specific to hierarchically hyperbolic spaces. Introduced in \cite{BHS17b} by Behrstock-Hagen-Sisto, hierarchically quasiconvex subsets have played a central role in the study of hierarchically hyperbolic space \cite{BHS17b,BHSAsyDim,BHS_HHS_Quasiflats}. Notably, a hierarchical quasiconvex subset of an HHS is itself an HHS. 
While hierarchically quasiconvex subsets are not always {\sqc}, we classify precisely when the two concepts agree. {\SQC} subsets are exactly the hierarchically quasiconvex subsets that satisfy the \emph{orthogonal projection dichotomy} (Definition \ref{defn: orthogonal projection dichotomy}), which describes how the projections of a {\sqc} subset to each of the associated hyperbolic spaces must look.

\begin{thmi}[{\SQC} subsets are hierarchically hyperbolic]\label{intro:quasiconvex and HQC}
Let ${X}$ be a hierarchically hyperbolic space with the bounded domain dichotomy.  A subset $Y \subseteq {X}$ is {\sqc} if and only if $Y$ is hierarchically quasiconvex and has the orthogonal projection dichotomy. In particular, if $Y\subseteq X$ is {\sqc}, then $Y$ is hierarchically hyperbolic.
\end{thmi}

Theorem \ref{intro:quasiconvex and HQC} is truly the central result of this paper as it explains how the {\sqc} subsets  interact with the projections defining the hierarchically hyperbolic structure of the ambient space. Further, this characterization is complete as the theorem fails whenever any of the hypotheses are weakened (see Remark \ref{rem: we need all the hypotheses}).

In \cite{ABD}, Abbott-Behrstock-Durham give several equivalent conditions for quasi-geodesics in a hierarchically hyperbolic space to be {\sqc} and for a map from a quasi-geodesic space $Y$ into a hierarchically hyperbolic space to be a stable embedding (see Definition \ref{defn:stable}). Theorems \ref{intro:contracting_iff_quasiconvex} and \ref{intro:quasiconvex and HQC}  generalize these results to general {\sqc} subsets and do not require the hypothesis of unbounded products utilized by Abbott-Behrstock-Durham. This generalization to all {\sqc} subsets is essential to our applications in Section \ref{sec: Examples} and Section \ref{sec: Hyperbolically embedded subgroups of HHGs}.

Part of the proof of Theorem \ref{intro:quasiconvex and HQC} involves studying hierarchically quasiconvex hulls in hierarchically hyperbolic spaces. The \emph{hierarchically quasiconvex hull} of a subset $Y$ is (coarsely) the smallest hierarchically quasiconvex set containing $Y$. We show that the hull of any subset of a hierarchically hyperbolic space can be constructed using special quasi-geodesics called hierarchy paths (see Theorem \ref{thm: hierarchy paths to hulls} for the precise statement).

\begin{thmi}[Constructing hulls with hierarchy paths]\label{intro:hull_from_hierarchy_paths}
If $Y$ is a subset of a hierarchically hyperbolic space $X$, then the hierarchically quasiconvex hull of $Y$ can be constructed in a uniformly finite number of steps by iteratively connecting points by hierarchy paths.
\end{thmi}

This construction is reminiscent of the construction of quasiconvex hulls in hyperbolic spaces by connecting pairs of points by geodesics and is similar to the join construction of hulls in coarse median spaces presented by Bowditch in \cite{BowditchConvexity}. The main purpose of Theorem \ref{intro:hull_from_hierarchy_paths} in our paper is to establish that hierarchically quasiconvex subsets are exactly the subsets that are ``quasiconvex with respect to hierarchy paths."  However, we expect this construction to have further applications in the study of hierarchically hyperbolic spaces. Indeed, Hagen-Petyt have used this construction to build quasi-isometries from some hierarchically hyperbolic groups to cube complexes \cite{Hagen_Petyt_HHG_QI_cube_complex}, and in Section \ref{subsec:hulls_and_coarse_medians} we apply Theorem \ref{intro:hull_from_hierarchy_paths} to provide a characterization of hierarchical quasiconvexity in terms of the coarse median structure on a hierarchically hyperbolic space.  This later result allows us to conclude that, in the setting of hierarchically hyperbolic spaces, the coarse median hull constructed in \cite{BowditchConvexity} is coarsely equal to the hierarchically quasiconvex hull; extending \cite[Lemma 7.3]{BowditchConvexity} from finite to arbitrary subsets.

Charney-Sultan proved that {\sqc} geodesics in a CAT(0) space are characterized by having at least quadratic \emph{lower divergence} \cite{CS_Boundary}. The third author introduced a generalization of lower divergence to all subsets \cite{Tran2015} and studied its relationship with {\sqcity} \cite{Tran2017}. If $Y$ is a subset of the quasi-geodesic space $X$, the \emph{lower relative divergence of $X$ with respect to $Y$} (or the divergence of $Y$ in $X$) is a family of functions that measures how efficiently one can travel in $X$ while avoiding $Y$.  Building on the work in \cite{Tran2017}, we establish the following.

\begin{thmi}[Contracting subsets have at least quadratic divergence]\label{intro:divergence and quasiconvex}
Let $X$ be a quasi-geodesic metric space. If $Y \subseteq X$ is contracting, then the lower relative divergence of $X$ with respect to $Y$ is at least quadratic. Further, if $X$ is a hierarchically hyperbolic space with the bounded domain dichotomy, then the lower relative divergence of $X$ with respect to $Y$ is at least quadratic if and only if $Y$ is {\sqc} (equivalently if and only if $Y$ is contracting).
\end{thmi}

Since the lower relative divergence of $X$ with respect to $Y$ agrees with Charney-Sultan's lower divergence when $Y$ is a geodesic in $X$, Theorem \ref{intro:divergence and quasiconvex} proves that {\sqc} geodesics (aka Morse geodesics) in hierarchically hyperbolic spaces with the bounded domain dichotomy are also characterized by having at least quadratic lower divergence.

After proving Theorems \ref{intro:contracting_iff_quasiconvex} through \ref{intro:divergence and quasiconvex}, we establish several HHS analogues of the ``bounded geodesic image property" of quasiconvex subsets of hyperbolic spaces. One of these analogues is the following.

\begin{thmi}\label{intro:contracting_features}
Let $Y$ be a {\sqc} subset of a hierarchically hyperbolic space $X$ with the bounded domain dichotomy. There is a contracting map  $\gate_Y \colon {X} \rightarrow Y$ so that for each $\lambda \geq 1$ there exists a constant $r_\lambda>0$ such that for all $x,y\in\mc{X}$, if $d(\gate_Y(x),\gate_Y(y))>r_\lambda$, then any $\lambda$--hierarchy path from $x$ to $y$ must intersect the $r_\lambda$--neighborhood of $Y$.\\
\end{thmi}

\noindent \textbf{{\SQC} subsets in specific examples.} 
After characterizing the {\sqc} subsets of hierarchically hyperbolic spaces, we apply our results to study the {\sqc} subsets of some of the most common examples of hierarchically hyperbolic spaces: the mapping class group, Teichm\"uller space, right-angled Artin and Coxeter groups, and the fundamental groups of graph manifolds.

It has been shown that {\sqc} subgroups of the mapping class group \cite{Kim2017}, right-angled Artin groups with connected defining graph \cite{Tran2017,Genevois_Hyp_in_CAT(0)}, and certain $\mathcal{CFS}$ right-angled Coxeter groups \cite{HH2017} are either hyperbolic or finite index. We give sufficient conditions for a hierarchically hyperbolic space to have the property that all its {\sqc} subsets are either hyperbolic or coarsely cover the entire space (see Proposition \ref{prop:one_orthogonal_component}). Applying this criteria to specific examples yields a new, unified proof of the work of Kim, Tran, Genevois, and Nguyen-Tran as well as the following new results for Teichm\"uller space, graph manifolds, and a class of right-angled Coxeter groups that we call strongly $\mc{CFS}$.

\begin{cori}\label{intro:Hyp_vs_coarsely_covering}
The following HHSs have the property that every {\sqc} subset is either hyperbolic or coarsely covers the entire space:
\begin{enumerate}[(a)]

	\item The Teichm\"uller space of a finite type surface with the Teichm\"uller metric
	\item The Teichm\"uller space of a finite type surface of complexity at least $6$ with the Weil-Petersson metric
	\item The mapping class group of a oriented, connected, finite type surface
	\item A right-angled Artin group with connected defining graph
	\item A right-angled Coxeter group with strongly $\mathcal{CFS}$ defining graph
	\item The fundamental group of a non-geometric graph manifold
\end{enumerate}

\noindent In particular, if  $H$ is a {\sqc} subgroup in any of the groups (c)-(f),  then $H$ is either stable or finite index.
\end{cori}

Stable subgroups of the mapping class group and right-angled Artin groups have been studied extensively and have several interesting equivalent characterizations including convex cocompactness in the mapping class group and purely loxodromic in right-angled Artin groups \cite{DurhamTaylorStability,KMT}.

We also use HHS theory and Theorem \ref{intro:quasiconvex and HQC} to give a new proof of \cite[Theorem 1.11]{Tran2017} and \cite[Proposition 4.9]{Genevois_Hyp_in_CAT(0)} characterizing when a special subgroup of a right-angled Coxeter group is {\sqc}. We then utilize this characterization, along with a construction of Behrstock, to demonstrate the large variety of different {\sqc} subsets that can be found in the class of $\mc{CFS}$ right-angled Coxeter groups.

\begin{thmi}
Every right-angled Coxeter group is an infinite index {\sqc} subgroup of some $\mc{CFS}$ right-angled Coxeter group.
\end{thmi}

\noindent \textbf{Hyperbolically Embedded Subgroups.}
As a final application of our characterization of {\sqc} subsets, we study the hyperbolically embedded subgroups of hierarchically hyperbolic groups. Hyperbolically embedded subgroups are generalizations of peripheral subgroups in relatively hyperbolic groups (see \cite{DGO}) and are a key component of studying acylindrically hyperbolic groups, a large class of groups exhibiting hyperbolic-like behavior (see \cite{Osin_acyl_hyp}). Work of Dahmani-Guirardel-Osin \cite{DGO} and Sisto \cite{Sisto_quasiconveity_of_hyperbolically_embedded} showed that if a finite collection of subgroups $\{H_i\}$  is hyperbolically embedded in a finitely generated group $G$, then $\{H_i\}$ is an almost malnormal collection and each $H_i$ is {\sqc}. While the converse of this statement is false in general (see the beginning of Section \ref{sec: Hyperbolically embedded subgroups of HHGs} for a counterexample), the converse does hold in the case of hyperbolic groups \cite[Theorem 7.11]{Bowditch_Rel_Hyp_Groups} and cocompactly cubulated groups \cite[Theorem 6.31]{Genevois_Hyp_in_CAT(0)}. We prove the converse in the setting of hierarchically hyperbolic groups.

\begin{thmi}[Characterization of hyperbolically embedded subgroups]\label{intro:hyperbolically_embedded}
Let $G$ be a hierarchically hyperbolic group. A finite collection of subgroups $\{H_i\}$ is hyperbolically embedded in $G$ if and only if $\{H_i\}$ is an almost malnormal collection and each $H_i$ is {\sqc}.
\end{thmi}

By \cite[Theorem A]{Kim2017}, an infinite index subgroup of the mapping class group of a surface is {\sqc} if and only if it is convex cocompact (this fact can also be deduced from Corollary \ref{intro:Hyp_vs_coarsely_covering}). Thus, as a specific case of Theorem \ref{intro:hyperbolically_embedded}, we have the following new result for the mapping class group.

\begin{cori}\label{intro:hyp_embedded_MCG}
If $S$ is oriented, connected, finite type surface of complexity at least 2 and $\{H_i\}$ is a finite collection of subgroups of the mapping class group of $S$ then the following are equivalent:
\begin{itemize}
	\item $\{H_i\}$ is hyperbolically embedded.
	\item $\{H_i\}$ is an almost malnormal collection and each $H_i$ is {\sqc}.
	\item $\{H_i\}$ is an almost malnormal collection and each $H_i$ is convex cocompact.
\end{itemize}

\end{cori}

\subsection{Open questions}

We believe that {\sqc} subgroups are a rich area of study with many interesting open questions both in the setting of hierarchically hyperbolic groups and beyond. In light of Theorem \ref{intro:contracting_iff_quasiconvex}, it is natural to wonder which results for {\sqc} subgroups of hyperbolic groups can be extended to {\sqc} subgroups of hierarchically hyperbolic groups (or even finitely generated groups). As a starting point, one may  aim to extend work of Gromov, Arzhantseva, and Gitik on combination theorems for {\sqc} subgroups of hyperbolic groups \cite{Gromov1,MR1866849,MR1700476}.

\begin{ques} Prove combination theorems for {\sqc} subgroups of hierarchically hyperbolic groups (or even finitely generated groups). In particular, investigate conditions guaranteeing that the subgroup generated by two  {\sqc} subgroups, $Q_1$ and $Q_2$, is {\sqc} and isomorphic to $Q_1*_{ Q_1\cap Q_2}Q_2$.
\end{ques}

As {\sqc} subsets are  invariant under quasi-isometry, they have the potential to play an important role in the quasi-isometric classification of hierarchically hyperbolic spaces. The following would be an interesting first step in this direction. 

\begin{ques}
Provide necessary conditions for an HHS to have the property that all its {\sqc} subsets are either hyperbolic or coarsely cover the entire space. Using defining graphs, characterize all right-angled Coxeter groups whose {\sqc} subsets are hyperbolic or coarsely cover the entire group\footnote{The case of right-angled Coxeter groups has been resolved by Genevois \cite{Genevois_RACG_SQC_subgroups}.}. 
\end{ques}

Looking beyond hierarchically hyperbolic spaces, we wonder about the possibilities of understanding {\sqc} subsets in other spaces with a notion of non-positive curvature. Specifically we ask the following.

\begin{ques}\label{ques:quasiconvex iff contracting}
For what other spaces are {\sqc} subsets contracting (in the sense of Definition \ref{defn:contracting})?
\end{ques}

Some of the first spaces one could consider are CAT(0) spaces, coarse median spaces, and the outer automorphism groups of free groups. In \cite{Sultan2014} it is shown that {\sqc} geodesics in  CAT(0) spaces are always contracting. We conjecture the same holds for all {\sqc} subsets of a CAT(0) space\footnote{This conjecture has been confirmed by Cashen \cite{Cashen_CAT(0)}.}. A possible starting point for coarse median spaces could be the recently posted paper \cite{BowditchConvexity}, in which Bowditch constructs  hulls for subsets of coarse median spaces and produces a number of  results similar to our work in Section \ref{sec:constructing_hulls}.

Our proof of Theorem \ref{intro:hyperbolically_embedded} rests strongly upon the equivalence between {\sqc} and contracting subsets. One may then presume that any group that is an answer  to Question \ref{ques:quasiconvex iff contracting} is also an answer for the following question.

\begin{ques}\label{ques:hyperbolically_embedded}
For what other finitely generated groups are almost malnormal, {\sqc} subgroups hyperbolically embedded? 
\end{ques}
A long standing open question in the study of quasiconvex subgroups of hyperbolic group is whether or not finitely generated, almost malnormal subgroups of hyperbolic groups must be quasiconvex.  Accordingly, we ask the analogous question for the larger class of hierarchically hyperbolic groups.

\begin{ques}
Are finitely generated, almost malnormal subgroups of hierarchically hyperbolic groups {\sqc}?
\end{ques}

\noindent \textbf{Acknowledgments.}
We gratefully acknowledge Mark F. Hagen for suggesting the strategy to attack Theorem \ref{thm: hierarchy paths to hulls};  Mark V. Sapir for an example of an almost malnormal {\sqc} subgroup that is not hyperbolically embedded; and Johanna Mangahas for suggesting the relation between hierarchical quasiconvexity and coarse median quasiconvexity in Section \ref{subsec:hulls_and_coarse_medians}. We thank Brian Bowditch, Tai-Danae Bradley, Heejoung Kim, Chris Hruska, and Dan Berylne for their comments on early versions of this paper. We are also grateful to Kevin Schreve for pointing out an error in the first version of this paper. The first two authors thank the organizers of YGGT 2018 and GAGTA 2018 where some of the work on this paper was completed. They  also give special thanks to their respective advisors,  Jason Behrstock and Alessandro Sisto, for their ongoing support and their many helpful comments on early drafts of this paper. Finally, we would like to thank the anonymous referees for a number of comments that improved this paper.

\subsection{Outline of the paper}

In Section \ref{sec:coarse_geometry}, we begin with the basic definitions and properties of {\sqc} subsets and the related notions of stability and contracting subsets of general quasi-geodesic spaces. In Section \ref{sec: contracting and relative divergence}, we define lower relative divergence and study the relationship between contracting subsets, {\sqc} subsets, and lower relative divergence in any quasi-geodesic space. We move on to hierarchically hyperbolic spaces in Section \ref{sec:HHS_Background}, where we give the definition of an HHS and detail the relevant tools and constructions we will need from the theory. In Section \ref{sec:constructing_hulls}, we explain how to construct hierarchically quasiconvex hulls using hierarchy paths.  As applications of this construction, we give a characterization of hierarchically quasiconvex sets in terms of the coarse median structure on the HHS and prove that {\sqc} subsets are also hierarchically quasiconvex.  In Section \ref{sec:characterizing_quasiconvex_subsets}, we state and prove our equivalent characterizations of {\sqc} subsets, finishing the proofs of Theorems \ref{intro:contracting_iff_quasiconvex}, \ref{intro:quasiconvex and HQC}, and \ref{intro:divergence and quasiconvex}. The remaining sections are devoted to applications of this characterization. We give a generalization of the bounded geodesic image property for {\sqc} subsets in Section \ref{subsec:bounded_geodesic_image}, study {\sqc} subsets in specific examples in Section \ref{sec: Examples}, and characterize hyperbolically embedded subgroups of HHGs in Section \ref{sec: Hyperbolically embedded subgroups of HHGs}.

\section{Coarse geometry}\label{sec:coarse_geometry}

\subsection{Quasi-geodesic spaces, conventions, and notations}\label{subsec:quasigeodesic_spaces}

This paper focuses on understanding the geometry of metric spaces up to quasi-isometry. While many of the metric spaces we are interested in applying our results to are geodesic metric spaces, many of the subspaces we will be studying will be quasi-geodesic, but not geodesic metric spaces. Thus, we will almost always assume our metric spaces are \emph{quasi-geodesic metric spaces}.

\begin{defn}
A metric space $X$ is a \emph{$(K,L)$--quasi-geodesic metric space} if for all $x,y \in X$ there exists a $(K,L)$--quasi-geodesic $\gamma \colon [a,b] \rightarrow X$ with $\gamma(a) = x$ and $\gamma(b) = y$.
\end{defn}

Given a $(K,L)$--quasi-geodesic metric space $X$, we can construct a geodesic metric space quasi-isometric to $X$ as follows: fix an  $\epsilon$--separated net  $N \subseteq X$ and connect a pair of points $x,y \in N$ by an edge of length $d(x,y)$ if $d(x,y) < 2\epsilon$. The resulting metric graph will be quasi-isometric to $X$. Since $\epsilon$ can be chosen to depend only on $K$ and $L$, this graph can be constructed such that the quasi-isometry constants will also depend only on $K$ and $L$. When convenient, we will exploit this fact to reduce proofs to the geodesic case.

A particularly important collection of metric spaces in geometric group theory is the class of $\delta$--hyperbolic metric spaces, introduced by Gromov in \cite{Gromov1, Gromov2}. While \(\delta\)--hyperbolic spaces are usually required to be geodesic, the following  is a direct extension of the definition to the setting of quasi-geodesic metric spaces. 

\begin{defn}\label{def:hyperbolic}
A $(K,L)$--quasi-geodesic metric space is \emph{$\delta$--hyperbolic} if for every $(K,L)$--quasi-geodesic triangle the  $\delta$--neighborhood of the union of any two of the edges contains the third. 
\end{defn}

Gromov's four-point condition can also be used to define a hyperbolic quasi-geodesic metric space, however as shown in \cite[Example 11.36]{DrutuKapovichBook}, this definition fails to be a quasi-isometry invariant if the spaces are not geodesic. In contrast, Definition \ref{def:hyperbolic} is a quasi-isometry invariant among quasi-geodesic spaces.  In particular, using the ``guessing geodesic" criterion, from \cite[Theorem 3.15]{MasurSchleimer} or \cite[Theorem 3.1]{Bowditch14}, you can show a quasi-geodesic space is hyperbolic in the sense of Definition \ref{def:hyperbolic} if and only if it is quasi-isometric to a geodesic metric space that is hyperbolic in the usual sense.

When referring to a property defined by a parameter (e.g. $\delta$--hyperbolic), we will often suppress that parameter when its specific value is not needed. To reduce the proliferation of additive and multiplicative constants throughout this paper, we will adopt the following notations.
\begin{notation}
Let \(A, B, K, L\) be real numbers. We write \[A \stackrel{K,L}{\preceq} B \text{ if } A \leq KB +L.\] If \(A \stackrel{K,L}{\preceq} B\) and \(B \stackrel{K,L}{\preceq} A\), we write \(A \stackrel{K,L}{\asymp} B\). 

We say two subsets of a metric space \emph{\(K\)--coarsely coincide} if their Hausdorff distance is at most \(K\). 
\end{notation}
\subsection{{\SQCity}, contracting, and stability}\label{subsec:Defn_of_quasiconvex}

The primary notion of convexity we will consider is the following notion of {\sqcity}.

\begin{defn}[{\SQC} subset]\label{defn:quasiconvex}

A subset $Y$ of a quasi-geodesic metric space $X$ is \emph{{\sqc}} if there is a function $Q \colon [1,\infty) \times [0,\infty) \rightarrow [0,\infty)$ such that for every $(K,L)$--quasi-geodesic $\gamma$ with endpoints on $Y$, we have $\gamma \subseteq N_{Q(K,L)}(Y)$. We call the function $Q$ the \emph{convexity gauge} for $Y$.
\end{defn}

It follows directly from the definition that {\sqcity} is a quasi-isometry invariant in the following sense.

\begin{lem}
\label{lem:quasiconvex_qi_invariant} 
Let $X$ and $Z$ be a quasi-geodesic metric spaces and $f\colon X \rightarrow Z$ be a $(K,L)$--quasi-isometry. If $Y$ is a $Q$--{\sqc} subset of $X$, then $f(Y)$ is a $Q'$--{\sqc} subset of $Z$, with $Q'$ depending only on $Q$, $K$ and $L$.
\end{lem}

In the setting of hyperbolic spaces, {\sqcity} is  equivalent to the  weaker  condition of quasiconvexity.

\begin{defn}
A subset $Y$ of a geodesic metric space $X$ is \emph{quasiconvex}  if there exists $D\geq 0$ such that for any geodesic $\gamma$ with endpoints on $Y$, we have $\gamma \subseteq N_{D}(Y)$. We call the constant $D$ the \emph{convexity constant} for $Y$.
\end{defn}

If $Y$ is a $Q$--{\sqc} subset of the $(K,L)$--quasi-geodesic space $X$, then any two points in $Y$ can be joined by a $(K,L)$--quasi-geodesic in $X$ that lies uniformly close to $Y$.  Thus $Y$ equipped with the metric inherited from $X$ will be a $(K',L')$--quasi-geodesic metric space where $K'$ and $L'$ depend only on $K$, $L$, and $Q$. For the rest of the paper, when discussing geometric properties (such as hyperbolicity) of a {\sqc} subset, we shall implicitly do so with respect to the metric inherited from the ambient space. In particular, if $f \colon X \rightarrow Z$ is a quasi-isometry between quasi-geodesic spaces and $Y$ is a {\sqc} subset of $X$, then $Y$ is quasi-isometric to $f(Y)$.

In \cite{DurhamTaylorStability}, Durham and Taylor introduced the following related notion of convexity.

\begin{defn}
A quasi-isometric embedding  $\Phi$ from a quasi-geodesic metric space \(Y\) into a quasi-geodesic metric space \(X\) is a \emph{stable embedding} if there is a function $R \colon [1,\infty) \times [0,\infty) \rightarrow [0,\infty)$  such that  if $\alpha$ and $\beta$ are two $(K,L)$--quasi-geodesics of \(X\) with the same endpoints in $\Phi(Y)$, then $d_{Haus}(\alpha,\beta)\leq R(K,L)$. 
\end{defn}

While the images of stable embeddings maintain many of the features of quasiconvex subsets of hyperbolic spaces, the definition is highly restrictive.  In particular, as the next proposition records, stable embeddings must always be onto hyperbolic subsets.

\begin{prop}
\label{defn:stable}
Let $\Phi \colon Y\to X$ be a quasi-isometric embedding from a quasi-geodesic metric space $Y$ to a quasi-geodesic metric space $X$. Then $\Phi$ is a stable embedding if and only if $Y$ is hyperbolic and $\Phi(Y)$ is {\sqc}.  In particular, if $Y$ is a {\sqc} subset of $X$, then the inclusion $i \colon Y \hookrightarrow X$ is a stable embedding if and only if $Y$ is hyperbolic with respect to the metric inherited from $X$.
\end{prop}

In \cite[Proposition 4.3]{Tran2017}, the third author proves the above proposition for the case of geodesic spaces. The more general statement above follows immediately from the fact that a quasi-geodesic space is always quasi-isometric to a geodesic space plus the fact that {\sqcity}, stability, and hyperbolicity are all quasi-isometry invariants.

One class of metric spaces we are particularly interested in are finitely generated groups equipped with a word metric.  In this setting we are particularly interested in understanding the {\sqc} and stable subgroups.

\begin{defn}
Let $G$ be a finitely generated group equipped with a word metric from some finite generating set. A subgroup $H<G$ is a \emph{{\sqc} subgroup} of $G$ if $H$ is a {\sqc} subset of $G$ with respect to the word metric on $G$. A subgroup $H<G$ is a \emph{stable subgroup} if $H$ is a {\sqc} subgroup and $H$ is a hyperbolic group.
\end{defn}

The above definition of stable subgroup is different than the one originally given in \cite{DurhamTaylorStability}, but it is equivalent by Proposition \ref{defn:stable}.

If $H$ is a {\sqc} subgroup of $G$, then $H$ is also finitely generated and undistorted in $G$.  Further, since {\sqc} is a quasi-isometry invariant, being a {\sqc} or a stable subgroup is independent of the choice of finite generating set for $G$.

It is common in the literature to study various ``contracting" properties of {\sqc} subsets.  In this paper, we compare {\sqc} subsets with the following notion of a contracting subset.

\begin{defn}
\label{defn:contracting}
Let $X$ be a quasi-geodesic metric space and $Y\subseteq X$. A map $g \colon X\to Y$ is said to be \emph{$(A,D)$--contracting} for some $A\in (0,1]$ and $D\geq 1$ if the following hold:
\begin{enumerate}
	\item $g$ is $(D,D)$--coarsely Lipschitz.
	\item For any $y\in Y$, $d\bigl(y,g(y)\bigr)\leq D$.
	\item For all $x\in X$, if we set $R=A d(x,Y)$, then $\diam\bigl(g\bigl(B_R(x)\bigr)\bigr)\leq D$.
\end{enumerate}

A subset $Y$ is said to be \emph{$(A,D)$--contracting} if there is an $(A,D)$--contracting map from $X$ to $Y$. 
\end{defn}

The above definition is motivated by \cite[Definition 2.2]{MMI} and generalizes the usual definition of contracting in hyperbolic and CAT(0) spaces to include maps that are not the closest point projection. This is critical to our study of hierarchically hyperbolic spaces in Section \ref{sec:characterizing_quasiconvex_subsets} and allows quasi-isometry invariance to be established directly from the definition.
\begin{lem}
\label{lem:contracting_qi_invariant} 
Let $X$ and $Z$ be quasi-geodesic metric spaces and $f\colon X \rightarrow Z$ be a $(K,L)$--quasi-isometry. If $Y$ is an $(A,D)$--contracting subset of $X$, then $f(Y)$ is an $(A',D')$--contracting subset of $Z$, where $A'$ and $D'$ depend only on $A$, $D$, $K$, and $L$.
\end{lem}

In the setting of hyperbolic spaces, {\sqc} subsets are contracting.  The contracting map will be the following  coarse closest point projection: if $X$ is a $\delta$--hyperbolic metric space and $Y \subseteq X$ is $Q$--{\sqc}, then there exist $K$ depending on $\delta$ and $Q$ and a $(1,K)$--coarsely Lipschitz map $\mf{p}_Y \colon X \rightarrow Y$ such that for all $x \in X$, $d(x,\mf{p}_Y(x)) \leq d(x,Y)+1$. By an abuse of language, we will refer to $\mf{p}_Y$ as the \emph{closest point projection} of $X$ onto $Y$. For any $Q$--{\sqc} subset $Y$ of a $\delta$--hyperbolic space, the map $\mf{p}_Y$ is $(1,D)$--contracting where $D$ depends only on $Q$ and $\delta$.

\section{Divergence of contracting subsets}\label{sec: contracting and relative divergence}

In this section we show that contracting subsets are always {\sqc}.  Without some negative curvature hypotheses, such as being hierarchically hyperbolic, the converse is not always true as we show in Example \ref{ex:quasiconvex_not_contracting}. Both of these statements are proved using lower relative divergence which was originally introduced by the third author in \cite{Tran2015}. The lower relative divergence is a family of functions that measures how efficiently one can travel in $X$ while avoiding a subset $Y$ (see Figure \ref{fig:Relative_Divergence}).

\begin{defn}[Lower relative divergence]
Let \(X\) be a geodesic space  and \(Y \subseteq X\). For \(r>0\) we adopt the following notations:
\begin{enumerate}
	\item \(\partial N_r(Y) = \{x \in X \mid d(x,Y)= r\}\)
	\item  \(d_r\) is the induced path metric on \(X - N_r(Y)\). 
\end{enumerate}
The \emph{lower relative divergence of $X$ with respect to $Y$} (or the \emph{divergence of $Y$ in $X$}), denoted $div(X,Y)$ is the set of functions $\{\sigma_\rho^n\}$ defined as follows:
For each $\rho \in (0,1]$, integer $n\geq 2$ and $r \in (0,\infty)$, if there is no pair of $x_1, x_2 \in \partial N_r(Y)$ such that $d_r(x_1, x_2)<\infty$ and $d(x_1,x_2)\geq nr$, we define $\sigma^n_{\rho}(r)=\infty$. Otherwise, we define $\sigma^n_{\rho}(r)=\inf d_{\rho r}(x_1,x_2)$ where the infimum is taken over all $x_1, x_2 \in \partial N_r(Y)$ such that $d_r(x_1, x_2)<\infty$ and $d(x_1,x_2)\geq nr$.
\end{defn}

\begin{figure}[H]
\centering
\def\svgscale{.7}
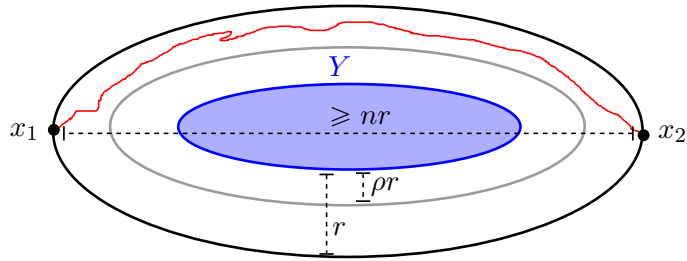
\caption{\label{fig:Relative_Divergence} A sketch of a step in the construction of the function $\sigma_\rho^n$.  The points $x_1,x_2\in \partial N_r(Y)$ are at least $nr$ far apart, so we measure the distance between $x_1$ and $x_2$ in the complement of the $\rho r$--neighborhood of $Y$. We then take the infimum of these distances over all such pairs of points to obtain $\sigma_\rho^n(r)$. }
\end{figure}

The lower relative divergence is often characterized by how the asymptotics of the functions $\{\sigma_\rho^n\}$ compare to linear, polynomial and exponential functions. Such descriptions are described in detail in \cite{Tran2015}.  For this paper we will restrict our attention to the following two properties of $div(X,Y)$.

\begin{defn}
\label{superlinear-quadratic}
Let $X$ be a geodesic metric space and $Y \subseteq X$. 

The lower relative divergence of $X$ with respect to $Y$ is \emph{completely superlinear} if there exists $n_0 \geq 3$ such that for every $\rho \in (0,1]$ and $C>0$ the set $\{r \in [0,\infty) :\sigma_\rho^{n_0}(r) \leq Cr\}$ is bounded. 

The lower relative divergence of $X$ with respect to $Y$ is \emph{at least quadratic} if there exists a positive integer $M$ such that for every $\rho \in (0,1]$ and $n\geq 2$ there exist $C>0$ and $r_0>0$ such that $ \sigma_\rho^{Mn}(r) > Cr^2$ for all $r>r_0$.
\end{defn}

The properties of being completely superlinear and at least quadratic are preserved under quasi-isometry in the following sense.

\begin{lem}[Consequence of {\cite[Proposition 4.9]{Tran2015}}]\label{lem:invariance_of_superlinear_and_quadratic}  
Let $f\colon X \rightarrow Z$ be a quasi-isometry between geodesic spaces. If $Y \subseteq X$ and $W \subseteq Z$ with $d_{Haus}(f(Y),W)< \infty$, then $div(X,Y)$ is completely superlinear (resp. at least quadratic) if and only if $div(Z,W)$ is completely superlinear (resp. at least quadratic).
\end{lem}

In \cite{Tran2015}, the lower relative divergence was defined only for geodesic ambient spaces, however the definition can be extended to include quasi-geodesic metric spaces as follows.

\begin{defn}[Lower relative divergence in quasi-geodesic spaces]\label{defn:lower relative div in quasi-geodesic}
Let $X$ be a quasi-geodesic space and $Y \subseteq X$. Let $Z$ be a geodesic space and $f\colon X\to Z$ be a quasi-isometry. Then \emph{the lower relative divergence of $X$ with respect to $Y$} (or the \emph{divergence of $Y$ in $X$}), denoted $div(X,Y)$, is the lower relative divergence of $Z$ with respect to $f(Y)$. 

We say $div(X,Y)$ is \emph{completely superlinear} (resp. \emph{at least quadratic}) if $div(Z,f(Y))$ is completely superlinear (resp. at least quadratic).

\end{defn}

While the definition of $div(X,Y)$ in a quasi-geodesic space depends on a choice of $Z$ and $f$, $div(X,Y)$ being completely superlinear (resp. at least quadratic) is independent of this choice by Lemma \ref{lem:invariance_of_superlinear_and_quadratic}. In fact, while it will not be relevant for the content of this paper, $div(X,Y)$ is independent of the choice of $Z$ and $f$ in a much stronger sense. In \cite{Tran2015} the third author defined an equivalence relation $\sim$ between the collections of functions used to define the lower relative divergence. If $f_1:X\to Z_1$ and $f_2:X\to Z_2$ are two quasi-isometries with $Z_1$ and $Z_2$ geodesic spaces, then by  \cite[Proposition 4.9]{Tran2015}, $div(Z_1,f_1(Y)) \sim div(Z_2,f_2(Y))$. Thus  $div(X,Y)$ is well defined up to this notion of equivalence.

The following proposition shows that contracting subsets always have at least quadratic divergence.

\begin{prop}
\label{prop:contracting_implies_quad}
If $X$ is a quasi-geodesic space and  $Y$ is a contracting subset of $X$, then the lower relative divergence of $X$ with respect to $Y$ is at least quadratic.
\end{prop}

\begin{proof}
Since every quasi-geodesic space is quasi-isometric to a geodesic metric space, Lemma \ref{lem:contracting_qi_invariant} allows us to assume $X$ is geodesic. Assume that $Y$ is $(A,D)$--contracting and let $g:X\to Y$ be an $(A,D)$--contracting map.  We first show that for all $x \in X$, \[d(x,g(x)) \leq 2D d(x,Y) + 4D.\]

Let $y \in Y$ such that $d(x,y) \leq d(x,Y)+1$. Then from the definition of $(A,D)$--contracting we have
\begin{align*}
	d(x,g(x)) \leq& d(x,y) + d\bigl(y,g(y)\bigr) + d\bigl(g(y),g(x)\bigl)\\
	\leq& d(x,Y)+1 + D + D  d(x,y)+ D\\
	\leq& (D+1)d(x,Y) +3D +1\\
	\leq& 2D d(x,Y) + 4D.
\end{align*}

Now, let $\{\sigma^n_{\rho}\}$ be the lower relative divergence of $X$ with respect to $Y$. We claim that for each $n\geq  4D+2 $ and $\rho\in(0,1]$ \[\sigma^n_{\rho}(r)\geq \biggl(\frac{A\rho}{4D}\biggr)r^2\text{ for each  }r>8D. \]
Let $r>8D$ , $n$ be an integer greater than ${ 4D+2 }$, and $\rho \in (0,1]$. If $\sigma^n_{\rho}(r)=\infty$, then the above inequality is true. Otherwise, let \(x_1, x_2\in \partial N_r(Y)\) be such that \(d(x_1, x_2) \geq nr\) and $d_r(x_1,x_2) \leq \infty$. The distances $d\bigl(x_1,g(x_1)\bigr)$ and $d\bigl(x_2,g(x_2)\bigr)$ are bounded above by  $ 2Dr+4D $. Therefore, 

\begin{align*}
	d\bigl(g(x_1),g(x_2)\bigr) &\geq d(x_1,x_2)-d\bigl(x_1,g(x_1)\bigr)-d\bigl(x_2,g(x_2)\bigr)\\
	&\geq nr - 4Dr-8D \\
	&\geq r
\end{align*}

Let $\gamma$ be a rectifiable path in $N_{\rho r}(Y)$ connecting $x_1$ and $x_2$ and $R=A\rho r/2$. There exist $t_0<t_1<t_2<\cdots<t_{m-1}<t_m$ such that $\gamma(t_0) =x_1$, $\gamma(t_m)=x_2$ and \[\frac{R}{2}\leq \ell(\gamma_{|[t_{i-1},t_i]})\leq R\] where $\ell(\cdot)$ denotes the length of a path. 
This implies \[\ell(\gamma)=\sum\limits_{i=1}^m \ell(\gamma_{|[t_{i-1},t_i]})\geq \frac{mR}{2}.\label{eq:contracting1}\tag{$1$} \]
Since $g$ is an $(A,D)$--contracting map and $d\bigl(\gamma(t_{i-1}), \gamma(t_i)\bigr)<Ad\bigl(\gamma(t_{i-1}), Y\bigr)$, we have $d\bigl(g(\gamma(t_{i-1})),g(\gamma(t_i))\bigr)\leq D$ for each $1\leq i \leq m$. Thus \[
d\bigl(g(x_1),g(x_2)\bigr)\leq \sum\limits_{i=1}^m d\bigl(g(\gamma(t_{i-1})),g(\gamma(t_i))\bigr)\leq mD. \label{eq:contracting2} \tag{$2$}
\]

Since $d\bigl(g(x_1),g(x_2)\bigr) \geq r$, Inequality (\ref{eq:contracting2}) implies $m\geq r/D$. Combining this with Inequality (\ref{eq:contracting1}), we have \[\ell(\gamma)\geq \frac{mR}{2}\geq \biggl(\frac{A\rho}{4D}\biggr)r^2.\]
Therefore, \[\sigma^n_{\rho}(r)\geq \biggl(\frac{A\rho}{4D}\biggr)r^2\] for $n\geq 4D+2$, $\rho\in (0,1]$, and $r>8D$. This implies that the lower relative divergence of $X$ with respect to $Y$ is at least quadratic.  
\end{proof}

In \cite{Tran2015}, the third author classified {\sqc} subsets in terms of their lower relative divergence. This result continues to hold in the slightly more general setting of quasi-geodesic spaces.

\begin{thm}[{\cite[Theorem 3.1]{Tran2017}}]
\label{thm:quasiconvex_iff_superlinear}
Let $X$ be a quasi-geodesic space and $Y \subseteq X$. Then $Y$ is {\sqc} if and only if the lower relative divergence of $X$ with respect to $Y$ is completely superlinear.
\end{thm}

\begin{proof}
Since every quasi-geodesic metric space is quasi-isometric to a geodesic metric space,  the result follows immediately from \cite[Theorem 1.5]{Tran2017} when $Y$ is infinite diameter. If $\diam(A) = r_0 < \infty$, then for all $r>r_0$, $\partial N_{r_0}(Y) = \emptyset$ and thus $\sigma_\rho^n(r) = \infty$. Hence $div(X,Y)$ is completely superlinear and $Y$ is {\sqc}. 
\end{proof}

Proposition \ref{prop:contracting_implies_quad} and Theorem \ref{thm:quasiconvex_iff_superlinear} combine to say that if a subset $Y \subseteq  X$ is $(A,D)$--contracting, then $Y$ is {\sqc}. A direct proof of this result was shown by Sultan for the case of quasi-geodesics, but their proof goes through for any subset without modification \cite[Lemma 3.3]{Sultan2014}. For completeness, we include a proof using the bound on the lower relative divergence of $Y$ from Proposition \ref{prop:contracting_implies_quad}.


\begin{cor}\label{cor:contracting implies quasiconvex}
Let $X$ be a $(K,L)$--quasi-geodesic space and $Y \subseteq X$. If $Y$ is $(A,D)$--contracting, then $Y$ is $Q$--{\sqc} where $Q$ is determined by $A$, $D$, $K$, and $L$.
\end{cor}

\begin{proof}

Let $Y$ be a $(A,D)$--contracting subset of $X$. We first assume that $X$ is a geodesic metric space. Let $\{\sigma_\rho^n\}$ be the lower relative divergence of $X$ with respect to $Y$. The proof of Proposition \ref{prop:contracting_implies_quad} shows that for each $n\geq 4D+2$ and $\rho\in(0,1]$ we have \[\sigma_{\rho}^{n}(r) \geq \biggl(\frac{A\rho}{4D}\biggr)r^2   \text{ for all } r>8D.\] Therefore, by fixing $n=n_0=4D+3$ and $\rho=1$ we have \[\sigma_{1}^{n_0}(r) \geq \biggl(\frac{A}{4D}\biggr)r^2    \text{ for all } r>8D.\] If $\gamma$ is a $(\lambda,\epsilon)$--quasi-geodesic with endpoints on $Y$, let $m = \inf \{B \in \mathbb{R} : \gamma \subseteq N_B(Y)\}$. The proof of \cite[Proposition 3.1]{Tran2017} establishes that if $m$ is larger than a fixed constant depending on $\lambda$ and $\epsilon$, then there exist constants $C_0$ and $C_1$ depending only on $\lambda$, $\epsilon$ and $n_0$, such that $ \sigma_1^{n_0}(C_0 m) \leq C_1m$. Thus, we have
\[ \biggl(\frac{A}{4D}\biggr)(C_0m)^2 \leq \sigma_{1}^{n_0}(C_0 m) \leq C_1 m,\]
and hence $m$ is bounded by some constant depending only on $\lambda$, $\epsilon$, $A$, $D$. Thus, there exists a function $Q$ depending only on $A$ and $D$ such that $Y$ is $Q$--{\sqc}. 

When $X$ is a $(K,L)$--quasi-geodesic space, there exist a geodesic metric space $Z$ and a quasi-isometry $f\colon X \rightarrow Z$ with constants determined by $K$ and $L$. The result follows from the geodesic case by Lemmas~\ref{lem:quasiconvex_qi_invariant} and \ref{lem:contracting_qi_invariant}.
\end{proof}

We finish this section by adapting \cite[Example 3.4]{ACGH} to give a counterexample to the converse of Corollary \ref{cor:contracting implies quasiconvex}.

\begin{exmp}[{\SQC} subsets need not be contracting]\label{ex:quasiconvex_not_contracting}

Let $Y$ be a ray with initial point $x_0$ and let $(x_n)$ be the sequence of points along $Y$ such that for each $n\geq 1$ the distance between $x_{n-1}$ and $x_n$ is equal to $n$. We connect each pair $(x_{n-1},x_n)$ by an additional segment $J_n$ of length $n^{3/2}$ as shown below. Let $X$ be the resulting geodesic space. 

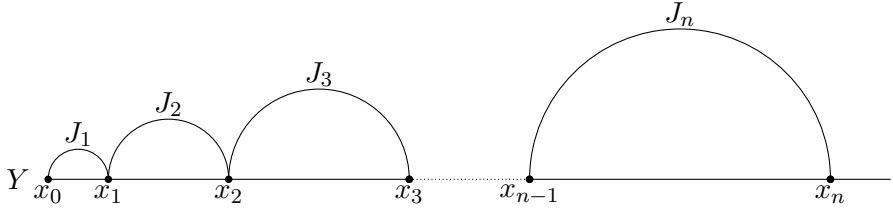
\begin{figure}[H]
	\begin{tikzpicture}[scale=.8]
		
		\draw (1,0) node[circle,fill,inner sep=1pt, color=black](1){} -- (2,0) node[circle,fill,inner sep=1pt, color=black](1){}-- (4,0) node[circle,fill,inner sep=1pt, color=black](1){}-- (7,0) node[circle,fill,inner sep=1pt, color=black](1){}; 
		
		\draw[densely dotted] (7,0) node[circle,fill,inner sep=1pt, color=black](1){} -- (9,0) node[circle,fill,inner sep=1pt, color=black](1){};
		
		\draw (9,0) node[circle,fill,inner sep=1pt, color=black](1){} -- (14,0) node[circle,fill,inner sep=1pt, color=black](1){};
		
		\draw (14,0) node[circle,fill,inner sep=1pt, color=black](1){} -- (15,0);

		\draw (2,0) arc (0:180:0.5);
		
		\draw (4,0) arc (0:180:1);
		
		\draw (7,0) arc (0:180:1.5);
		
		\draw (14,0) arc (0:180:2.5);

		\node at (1,-0.25) {$x_0$};
		
		\node at (2,-0.25) {$x_1$};
		
		\node at (4,-0.25) {$x_2$};
		
		\node at (7,-0.25) {$x_3$};
		
		\node at (9,-0.25) {$x_{n-1}$};
		
		\node at (14,-0.25) {$x_n$};

		\node at (1.5,0.75) {$J_1$};
		
		\node at (3,1.25) {$J_2$};
		
		\node at (5.5,1.75) {$J_3$};
		
		\node at (11.5,2.75) {$J_n$};

		\node at (0.5,0) {$Y$};

	\end{tikzpicture}
	\caption{The space $X$}\label{fig:quasiconvex not contracting}
	
\end{figure}

By Proposition \ref{prop:lower_relative_div_example}  the lower relative divergence of $X$ with respect to $Y$ is completely superlinear, but not at least quadratic (heuristically, $div(X,Y)$ behaves like $r^{3/2}$). So $Y$ is {\sqc}, but not contracting by Proposition~\ref{prop:contracting_implies_quad} and Theorem~\ref{thm:quasiconvex_iff_superlinear}.
\end{exmp}

\section{Hierarchically hyperbolic spaces}\label{sec:HHS_Background}
We now recall the main definitions of hierarchically hyperbolic groups and spaces.  The main references, where not specified, are \cite{BHS, BHS17b}. While we give the entire definition of an HHS for completeness, we advise the reader that we shall only directly utilize Axioms (\ref{item:dfs_curve_complexes}), (\ref{item:dfs_nesting}), (\ref{item:dfs_orthogonal}), (\ref{item:dfs_transversal}), (\ref{item:dfs:bounded_geodesic_image}), and (\ref{item:dfs_uniqueness}) of Definition \ref{defn:HHS} in the remainder of the paper. 

\begin{defn}[Hierarchically hyperbolic space]\label{defn:HHS}
Let $\mc{X}$ be a quasi-geodesic space. A \emph{hierarchically hyperbolic space (HHS) structure} on $\mc{X}$ consists of constants $E\geq \kappa_0>0$, an index set $\mathfrak S$, and a set $\{ C W : W\in\mathfrak S\}$ of geodesic $\delta$--hyperbolic spaces $( C W,\dist_W)$,  such that the following conditions are satisfied: \begin{enumerate}
	\item\textbf{(Projections.)}\label{item:dfs_curve_complexes} For each $W \in \mf{S}$, there exists a \emph{projection} $\pi_W \colon \mc{X} \rightarrow 2^{CW}$ such that for all $x \in \mc{X}$, $\pi_W(x) \neq \emptyset$, and $\diam(\pi_W(x))<E$. Moreover, there exists $K$ so that each $\pi_W$ is $(K,K)$--coarsely
	Lipschitz and $\pi_W(\mc{X})$ is $K$--quasiconvex in $CW$. 
	
	\item \textbf{(Nesting.)} \label{item:dfs_nesting} $\mathfrak S$ is
	equipped with a  partial order $\nest$, and either $\mathfrak
	S=\emptyset$ or $\mathfrak S$ contains a unique $\nest$--maximal
	element; when $V\nest W$, we say $V$ is \emph{nested} in $W$.  For each
	$W\in\mathfrak S$, we denote by $\mathfrak S_W$ the set of
	$V\in\mathfrak S$ such that $V\nest W$.  Moreover, for all $V,W\in\mathfrak S$
	with $V \propnest W$ there is a specified non-empty subset
	$\rho^V_W\subseteq C W$ with $\diam_{ C W}(\rho^V_W)\leq E$.
	There is also a \emph{projection} $\rho^W_V\colon  C
	W\rightarrow 2^{ C V}$. 
	
	\item \textbf{(Orthogonality.)} 
	\label{item:dfs_orthogonal} $\mathfrak S$ has a symmetric and
	anti-reflexive relation called \emph{orthogonality}; we write $V\perp
	W$ when $V,W$ are orthogonal. Whenever $V\nest W$ and $W\perp
	U$, we require that $V\perp U$. Additionally, if $V\perp W$, then $V,W$ are not
	$\nest$--comparable.
	
	\item \textbf{(Containers.)} 
	\label{item:dfs_containers} For each
	$T\in\mathfrak S$ and each $U\in\mathfrak S_T$ for which
	$\{V\in\mathfrak S_T : V\perp U\}\neq\emptyset$, there exists $W\in
	\mathfrak S_T-\{T\}$, so that whenever $V\perp U$ and $V\nest T$, we
	have $V\nest W$. We say $W$ is a \emph{container} for $U$ in $\mf{S}_T$.  
	
	\item \textbf{(Transversality and consistency.)}
	\label{item:dfs_transversal} If $V,W\in\mathfrak S$ are not
	orthogonal and neither is nested in the other, then we say $V,W$ are
	\emph{transverse}, denoted $V\trans W$.  If $V\trans W$, then there are non-empty
	sets $\rho^V_W\subseteq C W$ and
	$\rho^W_V\subseteq C V$ each of diameter at most $E$ and 
	satisfying: \[\min\left\{\dist_{
		W}(\pi_W(x),\rho^V_W),\dist_{
		V}(\pi_V(x),\rho^W_V)\right\}\leq \kappa_0\] for all $x\in\mc X$.
	
	For $V,W\in\mathfrak S$ satisfying $V\nest W$ and for all
	$x\in\mc X$, we have: \[\min\left\{\dist_{
		W}(\pi_W(x),\rho^V_W),\diam_{ C
		V}(\pi_V(x)\cup\rho^W_V(\pi_W(x)))\right\}\leq\kappa_0.\]
	
	Finally, if $U\nest V$, then $\dist_W(\rho^U_W,\rho^V_W)\leq\kappa_0$ whenever $W\in\mathfrak S$ satisfies either $V \propnest W$ or $V\trans W$ and $W\not\perp U$.
	
	\item \textbf{(Finite complexity.)} \label{item:dfs_complexity} There exists $n\geq0$ so that any set of pairwise--$\nest$--comparable elements has cardinality at most $n$.

	\item \textbf{(Large links.)} \label{item:dfs_large_link_lemma} There
	exists $\zeta\geq1$  such that the following holds.
	Let $W\in\mathfrak S$ and $x,x'\in\mc X$.    There exist $\{U_i\}_{i=1,\dots,m}\subseteq\mathfrak S_W-\{W\}$ such that $m \leq \zeta\dist_{_W}(\pi_W(x),\pi_W(x'))+\zeta$ and for all $V\in\mathfrak
	S_W-\{W\}$, either $V\in\mathfrak S_{U_i}$ for some $i$, or $\dist_{
		V}(\pi_V(x),\pi_V(x'))<E$.  Also, $\dist_{
		W}(\pi_W(x),\rho^{U_i}_W)\leq \zeta\dist_{W}(\pi_W(x),\pi_W(x'))+\zeta$ for each $i$.

	\item \textbf{(Bounded geodesic image.)} \label{item:dfs:bounded_geodesic_image} For all $W\in\mathfrak S$, all $V\in\mathfrak S_W-\{W\}$, and all geodesics $\gamma$ of $ C W$, either $\diam(\rho^W_V(\gamma))\leq E$ or $\gamma\cap N_E(\rho^V_W)\neq\emptyset$. 
	
	\item \textbf{(Partial realization.)} \label{item:dfs_partial_realization} There exists a constant $\alpha$ with the following property. Let $\{V_j\}$ be a family of pairwise orthogonal elements of $\mathfrak S$, and let $p_j\in \pi_{V_j}(\mc X)\subseteq  C V_j$. Then there exists $x\in \mc X$ so that:
	\begin{itemize}
		\item $\dist_{V_j}(x,p_j)\leq \alpha$ for all $j$,
		\item for each $j$ and 
		each $V\in\mathfrak S$ with $V_j\nest V$, we have 
		$\dist_{V}(x,\rho^{V_j}_V)\leq\alpha$,
		\item if $W\trans V_j$ for some $j$, then $\dist_W(x,\rho^{V_j}_W)\leq\alpha$.
	\end{itemize}
	
	\item\textbf{(Uniqueness.)} For each $\kappa\geq 0$, there exists
	$\theta_u=\theta_u(\kappa)$ such that if $x,y\in\mc X$ and
	$\dist(x,y)\geq\theta_u$, then there exists $V\in\mathfrak S$ such
	that $\dist_V(x,y)\geq \kappa$.\label{item:dfs_uniqueness}
\end{enumerate}
\end{defn}

We will refer to the elements of the index set $\mf{S}$ as \emph{domains} and use $\mf{S}$ to denote the entire HHS structure, including all the spaces, constants, projections and relations defined above. A quasi-geodesic space $\mc{X}$ is a \emph{hierarchically hyperbolic space} (HHS) if it admits a hierarchically hyperbolic structure. We will use the pair $(\mc{X},\mf{S})$ to denote $\mc{X}$ equipped with the hierarchically hyperbolic structure $\mf{S}$.\\

If $(\mc{X},\mf{S})$ is a hierarchically hyperbolic space and $f \colon \mc{Y} \rightarrow \mc{X}$ is a quasi-isometry, then $\mf{S}$ is also an HHS structure for $\mc{Y}$ where the projections maps are defined by $\pi_W \circ f$ for each $W \in \mf{S}$.

Many of the key examples of hierarchically hyperbolic spaces are finitely generated groups where the Cayley graph admits an HHS structure. In the case where this structure is preserved by the group action, we will call those groups hierarchically hyperbolic groups.

\begin{defn}[Hierarchically hyperbolic groups]\label{defn:hierarchically hyperbolic groups}

Let $G$ be a finitely generated group.  We say $G$ is a \emph{hierarchically hyperbolic group} (HHG) if

\begin{enumerate}[(1)]
	\item $G$ with the word metric from a finite generating set admits an HHS structure $\mf{S}$.
	\item There is a $\nest$, $\perp$, and $\trans$ preserving action of $G$ on $\mf{S}$ by bijections such that $\mf{S}$ contains finitely many $G$ orbits.
	\item \label{item:HHG conditions} For each $W \in \mf{S}$ and $g\in G$, there exists an isometry $g_W \colon CW \rightarrow C(gW)$ satisfying the following for all $V,W \in \mf{S}$ and $g,h \in G$:
	\begin{itemize}
		\item The map $(gh)_W \colon CW \to C(ghW)$ is equal to the map $g_{hW} \circ h_W \colon CW \to C(ghW)$.
		\item For each $h \in G$, $g_W(\pi_W(h))$ and $\pi_{gW}(gh)$ $E$--coarsely coincide.
		\item If $V \trans W$ or $V \nest W$, then $g_W(\rho_W^V)$  and $\rho_{gW}^{gV}$ $E$--coarsely coincide.
		\item If $V \nest W$ and $p \in CW - N_E\left( \rho_W^V \right)$, then $g_W(\rho_W^V(p))$ and $\rho_{gW}^{gV}(g_W(p))$ $E$--coarsely coincide.
	\end{itemize}
\end{enumerate}

The HHS structure $\mf{S}$ satisfying (1) - (3)  is called a hierarchically hyperbolic group (HHG) structure on $G$ and we use $(G,\mf{S})$ to denote a group $G$ equipped with a specific HHG structure $\mf{S}$.

\end{defn}

Being a hierarchically hyperbolic group is independent of choice of generating set by virtue of being able to pass the HHG structure through a $G$--equivariant quasi-isometry. The reader may find it helpful to note that the conditions in (\ref{item:HHG conditions}) above can be summarized by saying the following two diagrams coarsely commute whenever $V,U \in\mf{S}$ are not orthogonal.

\begin{multicols}{2}
\[
\begin{diagram}
	\node{ 
		G}\arrow[2]{e,t}{g}\arrow{s,l}{\pi_W}\node{}\node{
		G}\arrow{s,r}{\pi_{gW}}\\
	\node{ C W}\arrow[2]{e,t}{g_W}\node{}\node{ C 
		(gW)}
\end{diagram}
\]

\columnbreak

\[
\begin{diagram}
	\node{ 
		CV}\arrow[2]{e,t}{g_V}\arrow{s,l}{\rho_W^V}\node{}\node{
		C(gV)}\arrow{s,r}{\rho_{gW}^{gV}}\\
	\node{ C W}\arrow[2]{e,t}{g_W}\node{}\node{ C 
		(gW)}
\end{diagram}
\]

\end{multicols}

\begin{notation}\label{notation:suppress_pi}
When writing distances in $ C W$ for some $W\in\mathfrak S$, we 
often simplify the notation by suppressing the projection
map $\pi_W$, that is, given $x,y\in\mc X$ and
$p\in C W$  we write
$\dist_W(x,y)$ for $\dist_W(\pi_W(x),\pi_W(y))$ and $\dist_W(x,p)$ for
$\dist_W(\pi_W(x),p)$. Note that when we measure distance between a 
pair of sets (typically both of bounded diameter) we are taking the minimum distance 
between the two sets. 
Given $A\subseteq \mc X$ and $W\in\mathfrak S$ 
we let $\pi_{W}(A)$ denote $\bigcup_{a\in A}\pi_{W}(a)$.
\end{notation}

The guiding philosophy of hierarchically hyperbolic spaces is that one can ``pull back" the hyperbolic geometry of the various $CW$s to obtain features of negative curvature in the original space. The most prominent example of this philosophy is the following distance formula which allows distances in the main space $\mc{X}$ to be approximated by distances in the hyperbolic spaces.

\begin{thm}[The distance formula;  
{\cite[Theorem 4.4]{BHS17b}}]\label{thm:distance_formula}
Let $(\mc X, \mathfrak S)$ be a hierarchically hyperbolic space.  Then
there exists $\sigma_0$ such that for all $\sigma \geq \sigma_0$, there exist $K\geq 1$, $L\geq 0$ so
that for all $x,y\in\mc X$,
$$\dist_\mc{X}(x,y)\stackrel{{K,L}}{\asymp}\sum_{U\in\mathfrak
	S}\ignore{\dist_U(x,y)}{\sigma}$$

\noindent where $ \ignore{N}{\sigma} = N$ if $N \geq \sigma$ and $0$ otherwise.
\end{thm}

The distance formula can  be ``distributed'' over a sum of distances in the hyperbolic spaces as described in the next lemma.

\begin{lem}[{\cite[Lemma 2.26]{Russell_Rel_Hyp}}]\label{lem:distributing_distance_formula}
Let $(\mc{X},\mf{S})$ be an HHS and  $x_0,x_1, \dots, x_n$ be points in $\mc{X}$. If there exists $C\geq 1$ such that  $ \displaystyle \sum\limits_{i=0}^{n-1} d_W(x_i,x_{i+1}) \stackrel{C,C}{\asymp} d_W(x_0,x_n) $ for all $W \in \mf{S}$, then there exist $K$ depending only on  $C$, $n$, and $(\mc{X},\mf{S})$ such that \[\sum\limits_{i=0}^{n-1} d_\mc{X}( x_{i}, x_{i+1}) \stackrel{K,K}{\asymp} d_\mc{X} (x_0,x_n).  \]
\end{lem}

Part of the content of Theorem \ref{thm:distance_formula} is that for any  pair of points in an HHS, there is only a finite number of domains where that pair of points can have a large projection. More precisely, if  $(\mc X,\mathfrak S)$ is a hierarchically hyperbolic space, then a domain $W \in \mf{S}$ is said to be \emph{$\sigma$--relevant for $x,y \in \mc{X}$} if $d_W(x,y) > \sigma$. We denote  set of all $\sigma$--relevant domains for $x,y\in\mc{X}$ by $\relevant_{\sigma}(x,y)$. By Theorem \ref{thm:distance_formula}, for all $\sigma \geq \sigma_0$, $\relevant_\sigma(x,y)$ has finite cardinality. The relevant facts about $\relevant_\sigma(x,y)$ that we will need are summarized in the following proposition.

\begin{prop}[{\cite[Lemma  2.2, Proposition 2.8, Lemma 2.14]{BHS17b}}]\label{prop:partial_ordering_of_domains}
Let $(\mc{X},\mf{S})$ be a hierarchically hyperbolic space and $E\geq 0$ be the maximum of all the constants in the HHS structure for $(\mc{X},\mf{S})$. 
\begin{enumerate}[(1)]
	\item There exists $\chi>0$ such that if $\mf{U} \subseteq \mf{S}$ does not contain a pair of transverse domains, then $|\mf{U}| \leq \chi$.
	\item If $\sigma \geq 100E$ and $x,y \in \mc{X}$, then the set $\relevant_\sigma(x,y)$ can be partially ordered as follows:\[U\leq V \text{ if } U=V \text{ or } U\trans V \text{ and }d_V(\rho_V^U,y)\leq \kappa_0.\]
	\item  If $\sigma \geq 100E$ and $x,y \in\mc{X}$, then there exists $n \leq \chi$ such that $\relevant_\sigma (x,y)$ can be partitioned into $n$ disjoint subsets $\mc{U}_1,\hdots,\mc{U}_n$ where for each $i$, $\mc{U}_i$ is totally ordered with respect to the above ordering on $\relevant_\sigma(x,y)$.
\end{enumerate}

\end{prop}

Hierarchically hyperbolic spaces contain a particularly nice class of quasi-geodesics, called hierarchy paths. Even when considering a geodesic HHS, it is often preferable to work with hierarchy paths over geodesics.

\begin{defn}[Hierarchy path]\label{defn:hierarchy_path}
For $\lambda \geq 1$, a (not necessarily continuous) path $\gamma\co[a,b]\to \mc X$ is a \emph{$\lambda$--hierarchy path} if
\begin{enumerate}
	\item $\gamma$ is a $(\lambda,\lambda)$--quasi-geodesic,
	\item for each $W\in\mathfrak S$, the path $\pi_W\circ\gamma$ is an unparameterized $(\lambda,\lambda)$--quasi-geodesic. 
\end{enumerate}
\end{defn}

Recall that a map \(f \colon [a,b] \to X\) is an unparameterized \((\lambda, \lambda)\)--quasi-geodesic if there exists  an increasing function \(g \colon [0, \ell] \to [a, b]\) such that $g(0)=a$, $g(\ell)=b$, and \(f \circ g \) is a \((\lambda, \lambda)\)--quasi-geodesic of \(X\). 

While not every quasi-geodesic in an HHS is a hierarchy path, every pair of points can be connected by a hierarchy path as the next theorem describes.

\begin{thm}[Existence of hierarchy paths; {\cite[Theorem 5.4]{BHS17b}}]\label{thm:monotone_hierarchy_paths}
Let $(\mc X,\mathfrak S)$ be a hierarchically hyperbolic space. Then there exists $\lambda_0$ so that any $x,y\in\mc X$ are joined by a $\lambda_0$--hierarchy path. 
\end{thm}

\subsection{Hierarchical quasiconvexity and gate maps}\label{subsec:HQC} 
In \cite{BHS17b}, Behrstock-Hagen-Sisto introduced \emph{hierarchical quasiconvexity}, a notion of convexity unique to hierarchically hyperbolic spaces. 

\begin{defn}[Hierarchical quasiconvexity; {\cite[Definition 5.1]{BHS17b}}]\label{defn:hierarchical_quasiconvexity}
Let $(\mc X,\mathfrak S)$ be a hierarchically hyperbolic space and $k\colon[0,\infty)\to[0,\infty)$.  
A subset $ Y\subseteq\mc X$ is \emph{$k$--hierarchically quasiconvex} if the following hold:
\begin{enumerate}
	\item For all $U\in\mathfrak S$, the projection $\pi_U(Y)$ is 
	a $k(0)$--quasiconvex subspace of the $\delta$--hyperbolic space $C U$.
	\item For every $\kappa>0$ and every point 
	$x\in\mc X$ satisfying 
	$\dist_{U}(x, Y)\leq\kappa$ 
	for all $U\in\mathfrak S$, we have that 
	$\dist_\mc{X}(x, Y)\leq k(\kappa)$.
\end{enumerate}
\end{defn}

While hierarchically quasiconvex subsets need not be {\sqc}, they are ``quasiconvex with respect to hierarchy paths." That is, if $Y \subseteq \mc{X}$ is $k$--hierarchically quasiconvex then any $\lambda$--hierarchy path with endpoints on $Y$ must stay uniformly close to $Y$.  The existence of hierarchy paths (Theorem \ref{thm:monotone_hierarchy_paths}) therefore ensures that if $Y$ is equipped with the induced metric from $\mc{X}$, then  $Y$ is also a quasi-geodesic metric space with constants depending only on $(\mc{X},\mf{S})$ and $k$. In Section \ref{sec:constructing_hulls} we will prove that hierarchically quasiconvex subsets are actually characterized by this ``quasiconvexity with respect to hierarchy paths."

One of the key features of hierarchically quasiconvex subsets is that they are hierarchically hyperbolic spaces with the restriction of the HHS structure from the ambient space.
\begin{thm}[{\cite[Proposition 5.6]{BHS17b}}]
Let $(\mc X,\mathfrak S)$ be a hierarchically hyperbolic space and $ Y\subseteq\mc X$ be {$k$--hierarchically quasiconvex}. Then $(Y,\mf{S})$ is a hierarchically hyperbolic space, where \(Y\) is equipped with the induced metric from \(\mc X\).
\end{thm}

The following lemma is a special case of the powerful realization theorem for hierarchically hyperbolic spaces (see \cite[Theorem 3.1]{BHS17b}). It is often useful when verifying that a subset is hierarchically quasiconvex.

\begin{lem}[{\cite[Theorem 3.1, Lemma 5.3]{BHS17b}}]\label{lem:proj_consistent}
For each $R \geq 0$ there exists $\mu\geq 0$ so that the following holds. Let $Y\subseteq\mc X$ be such that $\pi_W(Y)$ is $R$--quasiconvex for each $W\in\mathfrak S$. Let  $x\in\mc X$ and for each $W\in\mathfrak S$, let $p_W \in \pi_W(Y)$ satisfy $\dist_V(x,p_W)\leq d_W(x,Y)+1$. Then there exists $p \in \mc{X}$ such that $d_W(p,p_W) \leq \mu$ for all $W \in \mf{S}$.
\end{lem}

Given a subset $Y \subseteq \mc{X}$, there exists a \textit{hierarchically quasiconvex hull} of $Y$ which can be thought of as the coarsely smallest hierarchically quasiconvex subset of $\mc{X}$ containing $Y$.

\begin{defn}[Hierarchically quasiconvex hull]\label{defn:hull}
For each set $Y\subseteq \mc X$ and $W \in \mf{S}$, let $\hull_{CW}(Y)$ denote the convex hull of $\pi_W(Y)$ in $CW$, i.e., the union of all $CW$--geodesics connecting pairs of points in $\pi_W(Y)$. Given $\theta\geq 0$, let $H_{\theta}(Y)$ be the set of all $p\in\mc X$ so that, for each
$W\in\mathfrak S$, the set $\pi_W(p)$ lies at distance at most
$\theta$ from $\hull_{C W}(Y)$.
Note that $Y\subseteq H_{\theta}(Y)$.
\end{defn}

\begin{lem}[{\cite[Lemma 6.2]{BHS17b}}]\label{lem:hull}
Let $(\mc{X},\mf{S})$ be an HHS.
There exists $\theta_0$ so that for each $\theta\geq \theta_0$ there exists $k:[0,\infty)\to [0,\infty)$ such that for each $Y\subseteq \mc X$, the hull $H_{\theta}(Y)$ is $k$--hierarchically quasiconvex.
\end{lem}

In Section \ref{sec:constructing_hulls} we strengthen the analogy between hierarchically quasiconvex hulls and convex hulls in hyperbolic spaces, by showing that $H_\theta(Y)$ can be constructed by iteratively connecting points in $Y$ by hierarchy paths.

One of the important properties of hierarchically quasiconvex subsets is the existence of a \textit{gate map} which retracts the entire space onto the hierarchically quasiconvex subset.  The gate map is a generalization to hierarchically hyperbolic spaces of the closest point projection, $\mf{p}$, defined at the end of Section \ref{sec:coarse_geometry}.

\begin{lem}[Existence of coarse gates; {\cite[Lemma 5.5]{BHS17b}}]\label{lem:gate}
If $(\mc{X},\mf{S})$ is a hierarchically hyperbolic space and $ Y\subseteq\mc X$ is $k$--hierarchically quasiconvex and non-empty, then there exists a  gate map $\gate_{Y}\colon \mc{X} \rightarrow{Y}$ such that
\begin{enumerate}[(1)]
	\item $\gate_Y$ is $(K,K)$-coarsely Lipschitz;
	\item for all $y \in Y$, $d_\mc{X}\bigl(y,\gate_Y(y)\bigr) \leq K$;
	\item for all $x \in \mc{X}$ and $U \in \mf{S}$, $d_U(\gate_{Y}(x),\cpproj_{\pi_U({Y})}(\pi_U(x))) \leq K$;
\end{enumerate}
where $K$ depends only on $k$ and $\mf{S}$.
\end{lem}

While the gate map need not be the closest point projection, it approximates the closest point projection with a multiplicative and additive error.

\begin{lem}[{\cite[Lemma 1.27]{BHS_HHS_Quasiflats}}]\label{lem: closest_point_proj}
Let ${Y}$ be a $k$-hierarchically quasiconvex subset of the HHS $(\mc{X},\mf{S})$ and $x \in \mc{X}$. If $y \in {Y}$ is a  point such that $d_{\mc{X}}(x,y)\leq d_\mc{X}(x,{Y})+1$, then $d_\mc{X}(x,y) \asymp d_\mc{X}(x,\gate_Y(x))$ where the constants depend only on  $k$ and $\mf{S}$.
\end{lem}

In the case of hierarchically hyperbolic groups, the gate is also coarsely equivariant.
\begin{lem}[Coarse equivariance of gate maps]\label{lem:coarse_equiv_of_gate}
Let \((G,\mf{S})\) be a hierarchically hyperbolic group and let \(Y\) be a \(k\)-hierarchically quasiconvex subspace of \(G\). There exists \(K\) depending on \((G,\mf{S})\) and \(k\) such that for every \(g, x \in G\), we have \[d_G (g \gate_Y(x), \gate_{gY}(gx)) \leq K.\]
\end{lem}

\begin{proof}
Since $G$ acts on the disjoint union of the $CW$s by isometries,  Lemma \ref{lem:gate} and the definition of HHG provide a uniform bound on \(d_W(\pi_W(g\gate_Y(x)),\pi_W(\gate_{gY}(gx))\) for all $W \in \mf{S}$, which depends only on $\mf{S}$, $k$, and the choice of finite generating set for $G$. The result now follows from the distance formula (Theorem \ref{thm:distance_formula}). 
\end{proof}

The following lemma explains the nice behavior of the gates of hierarchically quasiconvex sets onto each other. The lemma is stated in slightly more generality than presented in \cite{BHS_HHS_Quasiflats}, but the more general statement is implicit in the proof of \cite[Lemma 1.20]{BHS_HHS_Quasiflats}. The following notation will simplify the exposition.

\begin{notation}
If $\mf{S}$ is an HHS structure on a metric space $\mc{X}$ and $\mc{H} \subseteq \mf{S}$ we use
$\mc{H}^\perp$ to denote the set $\{W \in \mf{S} : \forall H \in \mc{H}, \ H \perp W\}$. In particular,  $\mf{S}_U^\perp = \{W \in \mf{S} : U\perp W \}$ for any  $U \in \mf{S}$. Note, if $\mc{H} = \emptyset$, then $\mc{H}^\perp = \mf{S}$ as every domain in $\mf{S}$ would vacuously satisfy the condition of the set.

\end{notation}

\begin{thm}[The bridge theorem; {\cite[Lemma 1.20]{BHS_HHS_Quasiflats}} ]\label{thm:parallel_gates}
Let $(\mc{X},\mf{S})$ be a hierarchically hyperbolic space and $\theta_0$ be as in Lemma \ref{lem:hull}. For every $k$ and $\theta \geq \theta_0$, there exist $k' \colon [0,\infty) \to [0,\infty)$ and $K_0\geq 0$ such that for any $k$--hierarchically quasiconvex sets $A,B$, the following hold.
\begin{enumerate}
	\item\label{item:gate_hq} $\gate_{A}(B)$ is
	$k'$--hierarchically quasiconvex.  
	
	\item\label{item:double_gate_surjective}  The composition $\gate_A\circ\gate_B|_{\gate_A(B)}$ is bounded distance from the identity $\gate_A(B)\to\gate_A(B)$.
	
	\item\label{item:bridge}
	For any $a\in \gate_A(B),b=\gate_B(a)$, we have a $(K_0,K_0)$--quasi-isometric
	embedding $f\co\gate_A(B)\times H_\theta(a,b)\to \mc{X}$ with
	image $H_\theta(\gate_A(B)\cup\gate_B(A))$, so that 
	$f(\gate_A(B)\times\{b\})$ $K_0$--coarsely coincides with $\gate_B(A)$.
\end{enumerate}

Let $K\geq K_0$ and $\mathcal H=\{U\in\mathfrak S: \diam\bigl(\pi_U(\gate_A(B))\bigr)> K\}$.
\begin{enumerate}\setcounter{enumi}{3}
	\item \label{item:df_bridge_0}For each $p,q\in \gate_A(B)$ and $t\in H_\theta(a,b)$, we have \[\relevant_K(f(p,t),f(q,t))\subseteq \mathcal H.\]
	\item\label{item:df_bridge}  For each $p\in \gate_A(B)$ and $t_1,t_2\in H_\theta(a,b)$, we have \[\relevant_K(f(p,t_1),f(p,t_2))\subseteq \mathcal H^\perp.\]
	\item\label{item:take_the_bridge} For each $p\in A,q\in B$ we have $$d(p,q)\asymp_{K_0,K_0} d(p,\gate_A(B))+d(q,\gate_B(A))+d(A,B)+d(\gate_{\gate_B(A)}(p),\gate_{\gate_B(A)}(q)).$$
\end{enumerate}
\end{thm}

We name Theorem \ref{thm:parallel_gates}  the bridge theorem as one should think of the set $H_\theta(\gate_A(B)\cup\gate_B(A))$ as a ``bridge" between $A$ and $B$: in order to efficiently travel between $A$ and $B$ one needs to always traverse this bridge. The bridge theorem, along with the construction of the gate map and hulls produces the following fact about the set $H_\theta(\gate_A(B)\cup\gate_B(A))$ which we will need in Section \ref{sec: Hyperbolically embedded subgroups of HHGs}.

\begin{lem}\label{lem:gate_of_bridge}
For every $k$ and $\theta \geq \theta_0$, there exists $K$ such that for any $k$--hierarchically quasiconvex sets $A,B$, the sets $\gate_B(H_\theta(\gate_A(B)\cup\gate_B(A)))$ and $\gate_B(A)$ $K$--coarsely coincide.
\end{lem}

We finish this section by recalling the construction of 
\emph{standard product regions} introduced in 
\cite[Section~13]{BHS} and studied further in 
\cite{BHS17b}. For what follows, fix a hierarchically hyperbolic space $(\mc{X},\mf{S})$.

\begin{defn}[Nested partial tuple $\PF_U$]\label{defn:nested_partial_tuple}
Recall $\mathfrak S_U=\{V\in\mathfrak S \mid V\nest U\}$.  Define $\PF_U$ to be the set of
tuples  in $\prod_{V\in\mathfrak S_U}2^{C
	V}$ satisfying the conditions 
of Definition~\ref{defn:HHS}.(\ref{item:dfs_transversal}) for all $V,W \in \mf{S}_U$ with $V \not\perp W$.
\end{defn}

\begin{defn}[Orthogonal partial tuple $\mathbf E_U$]\label{defn:orthogonal_partial_tuple}
Recall $\mathfrak S_U^\perp=\{V\in\mathfrak S\mid V\perp U\}$.  Define $\mathbf E_U$ to be the set of
tuples in $\prod_{V\in\mathfrak
	S_U^\perp}2^{C V}$ satisfying the conditions 
of Definition~\ref{defn:HHS}.(\ref{item:dfs_transversal}) for all $V,W \in \mf{S}^\perp_U$ with $V \not\perp W$.
\end{defn}

\begin{defn}[Product regions in $\mc X$]\label{const:embedding_product_regions}
Let $U \in \mf{S}$.  There exists $\mu$ depending only on $\mf{S}$ such that for each $(a_V)_{V \in\mf{S}_U} \in \PF_U$ and $(b_V)_{V \in \mf{S}_U^\perp} \in \PE_U$, there exists $x\in \mc{X} $ such that the following hold for each $V \in\mf{S}$:
\begin{itemize}
	\item If $V \nest U $, then $d_V(x,a_V) \leq \mu$.
	\item If $V \perp U$, then $d_V(x,b_V) \leq \mu$.
	\item If $V \trans U$ or $U \nest V$, then $d_V(x,\rho_V^U) \leq \mu$.
\end{itemize}
Thus there is a map $\phi_U \colon \PF_U \times \PE_U \rightarrow \mc{X}$, whose image is $k$--hierarchically quasiconvex where $k$ only depends on $\mf{S}$. We call $\phi_U(\PF_U \times \PE_U)$ the \emph{product region for $U$} and denote it $\P_U$. 
\end{defn}

For any $e \in \PE_U$, $f \in \PF_U$, the sets $\phi_U(\PF_U \times \{e\})$ and $\phi_U(\{f\} \times \PE_U)$ will also be hierarchically quasiconvex, thus $\PE_U$ and $\PF_U$ are quasi-geodesic metric spaces when equipped with the subspace metric from $\phi_U(\PF_U \times \{e\})$ and $\phi_U(\{f\} \times \PF_U)$. While these metrics depend on the choice of $e$ and $f$, the distance formula (Theorem \ref{thm:distance_formula}) ensures that the different choices are all uniformly quasi-isometric.

The definition of the product regions ensure that they are not only uniformly hierarchically quasiconvex, but have easily described gate maps.

\begin{lem}\label{lem:gate_onto_product_region}
Let $(\mc{X},\mf{S})$ be an HHS. The exists $k \colon [0,\infty) \to [0,\infty)$ so that for all $U \in \mf{S}$, the product region $\P_U$ is $k$-hierarchically quasiconvex. Moreover, there exists $K\geq 0$ depending only on $(\mc{X},\mf{S})$ so that for all $x \in \mc{X}$ we have.
\begin{itemize}
	\item $d_V\left(\gate_{\P_U}(x),x\right) \leq K$ if $V \nest U$ or $V \perp U$;
	\item $d_V\left(\gate_{\P_U}(x),\rho_V^U\right) \leq K$ if $V \trans U$ or $U \propnest V$.
\end{itemize}
\end{lem}

A version of our last result appeared as \cite[Proposition 5.17]{BHS17b}. However, that result contains an error in both its statement and its proof\footnote{The error in the proof of  \cite[Proposition 5.17]{BHS17b} is the incorrect claim that $V \nest U \implies \P_V \subseteq \P_U$. The error in the statement  is that \emph{all} hierarchy paths have the stated properties instead of there existing at least one hierarchy path with the stated properties.}. We provide a corrected statement and proof.

\begin{prop}[Active subpaths, {Corrected version of \cite[Proposition 5.17]{BHS17b}}]\label{prop: active subpath}
Let $(\mc{X},\mf{S})$ be an HHS. There exist constants $D, \nu,\lambda \geq 1$ so  that for all $x, y \in \X$, if $d_U(x,y) > D$ for some $U \in \mf{S}$, then there exists a $\lambda$--hierarchy path $\gamma\colon [a,b]\to \X$ joining $x$ and $y$ that has a subpath $\alpha=\gamma_|{[a_1,b_1]}$ such that:
\begin{enumerate}
	\item $\alpha \subseteq N_\nu (\mathbf{P}_U)$.
	\item  The diameters of $\pi_W\bigl(\gamma([a,a_1])\bigr)$ and $\pi_W\bigl(\gamma([b_1,b])\bigr)$ are both bounded by $\nu$, for all $W \in \mf{S}_U \cup \mf{S}_U^\bot$. 
	\item  For any point $p\in \gamma\bigl([a,a_1]\bigr)$ or  $q \in \gamma\bigl([b_1,b]\bigr)$, we have $$d_{\mc{X}}\left(\gate_{\P_U}(x),\gate_{\P_U}(p)\right) \leq \nu  \text{ and } d_{\mc{X}}\left(\gate_{\P_U}(y),\gate_{\P_U}(q)\right) \leq \nu.$$
\end{enumerate}
We call $\alpha$ the \emph{active subpath} of $\gamma$ for $U$.
\end{prop}

\begin{proof}
Let $\delta$, $E$, and $\kappa_0$ be the constants appearing in the HHS structure $\mf{S}$ for $\mc{X}$. Let $x' = \gate_{\P_U}(x)$ and $y' = \gate_{\P_U}(y)$. Let $\lambda_0 \geq 1$ be the constant so that every pair of points in $\mc{X}$ can be joined by a $\lambda_0$--hierarchy path and $\mu$ be the constant from Definition \ref{const:embedding_product_regions}. Both $\mu$ and $\lambda_0$ depend only on $(\mc{X},\mf{S})$.

Let $\gamma_0$, $\gamma_1$, and $\gamma_2$ be  $\lambda_0$--hierarchy paths connecting the pairs $(x, x')$, $(x',y')$, and $(y',y)$ respectively. Let $\gamma \colon [a,b] \to \mc{X}$ be the concatenation $\gamma_0 \ast \gamma_1 \ast \gamma_2$. We first verify that the path $\gamma$ satisfies the requirements of the proposition with $\alpha = \gamma_1$ and then verify that $\gamma$ is in fact a hierarchy path with  constant depending only on the HHS $(\mc{X},\mf{S})$.

For the first item, let $z \in \alpha = \gamma_1$. By Lemma \ref{lem:gate_onto_product_region}, $\pi_W\left(\gate_{\P_U}(z)\right)$ and $\pi_W(z)$ are uniformly close for all $W \in \mf{S}_U \cup \mf{S}_U^\bot$. If $W \not \in \mf{S}_U \cup \mf{S}_U^\bot$, then $\pi_W\left(x'\right)$, $\pi_W(y')$, and $\pi_W \left( \gate_{\P_U}(z) \right)$ are all $\mu$--close to $\rho_W^U$ because $x'$, $y'$, and $\gate_{\P_U}(z)$ are all in $\P_U$.  Since $\pi_W \circ \gamma_1$ is an unparameterized  $\lambda_0$--quasi-geodesic, $\pi_W(z)$ must also be uniformly close to $\rho_W^U$. Therefore, $d_W(\gate_{\P_U}(z), z)$ is uniformly bounded for all $W \not \in \mf{S}_U \cup \mf{S}_U^\perp$. Since  $d_W(\gate_{\P_U}(z), z)$ is uniformly bounded for all $W \in \mf{S}$,  the distance formula (Theorem \ref{thm:distance_formula}) provides  $\nu_1 \geq 0$ so that $\gamma_1 \subseteq N_\nu (\mathbf{P}_U)$.

For the second  item, if $W  \in \mf{S}_U \cup \mf{S}_U^\bot$, then  $d_W(x,x')$ and $d_W(y',y)$ are both uniformly bounded by Lemma \ref{lem:gate_onto_product_region}. Since $\pi_W \circ \gamma_0$ and $\pi_W\circ \gamma_2$ are unparameteized $(\lambda_0,\lambda_0)$--quasi-geodesics, there is a constant $\nu_2 \geq 0$ satisfying the second item.

We prove the third item for $p \in \gamma_0$ as the case $q \in \gamma_2$ is identical. By the second item $d_W(x,p) \leq \nu_2$ for all $W \in \mf{S}_U \cup \mf{S}_U^\perp$. Since $d_W(x,\gate_{\P_U}(x))$ and $d_W(p,\gate_{\P_U}(p))$ are uniformly bounded for all $W \in \mf{S}_U \cup \mf{S}_U^\perp$ as well (Lemma \ref{lem:gate_onto_product_region}), we have that $d_W(\gate_{\P_U}(x), \gate_{\P_U}(p))$ has a bound depending only on $(\mc{X},\mf{S})$ for all $W \in \mf{S}_U \cup \mf{S}_U^\perp$.  If instead $U \propnest W$ or $W \trans U$, then $\pi_W\left(\gate_{\P_U}(x)\right)$ and $\pi_W\left(\gate_{\P_U}(p)\right)$  are both uniformly close of $\rho_W^U$ as they are points in the product region $\P_U$. Hence $d_W(\gate_{\P_U}(x), \gate_{\P_U}(p)) $ is uniformly bounded for all $W \in \mf{S}$. Thus, the distance formula provides  $\nu_3 \geq 0$ depending only on $\mf{S}$ so that $d_{\mc{X}}\left(\gate_{\P_U}(x),\gate_{\P_U}(p)\right) \leq \nu_3$.

Set $\nu = \max\{\nu_1,\nu_2,\nu_3\}$. This depends only on $(\mc{X},\mf{S})$ since each of the $\nu_i$ depend only on $(\mc{X},\mf{S})$. It remains to show that $\gamma$ is a hierarchy path with constant depending only on $(\mc{X}, \mf{S})$. For this we need to assume that $d_U(x,y)>10(E+\kappa_0)$.

We first show that $\pi_W \circ \gamma$ is a uniform unparameterized quasi-geodesic for each $W \in \mf{S}$. 
\begin{itemize}
	\item If $W \in  \mf{S}_U \cup \mf{S}_U^\bot$, then $\diam\left(\pi_W(\gamma_0)\right) \leq \nu$, $\diam\left(\pi_W(\gamma_2)\right) \leq \nu$, and $\pi_W\circ \gamma_1$ is an unparameterized $(\lambda_0,\lambda_0)$--quasi-geodesic. Hence $\pi_W\circ \gamma$ is an unparameterized $(\lambda_0, \lambda_0 +2\nu)$--quasi-geodesic. 
	\item If $U \propnest W$, then by the bounded geodesic image axiom (Axiom \ref{item:dfs:bounded_geodesic_image}) any $CW$--geodesic from $\pi_W(x)$ to $\pi_W(y)$ must intersect the $E$--neighborhood of $\rho_W^U$. Since  all of $\pi_W \circ \gamma_1$ is contained in $N_{\lambda_0(E+\mu)+\lambda_0}(\rho_W^U)$, the hyperbolicity of $CW$ implies that both of the unparameterized quasi-geodesics $\pi_W\circ \gamma_0$ and $\pi_W \circ \gamma_2$ are contained in a regular neighborhood of a $CW$--geodesic from $\pi_W(x)$ to $\pi_W(y)$. Thus $\pi_W \circ \gamma$ will be a unparameterized quasi-geodesic with constants depending on $\lambda_0$, $\mu$, $E$, and $\delta$.
	\item If $W \trans U$, then since $d_U(x,y) > 10(E+\kappa_0)$, the consistency axiom (Axiom \ref{item:dfs_transversal}) ensures that at most one of $d_W(x,\rho_W^U)$ and $d_W(y,\rho_W^U)$ are larger than $\kappa_0$. Without loss of generality, assume $d_W(x,\rho_W^U) \leq \kappa_0$.  Since $\pi_W(x')$ and $\pi_W(y')$  are  $\mu$--close to $\rho_W^U$ and $\gamma_0$ and $\gamma_1$ are both $\lambda_0$--hierarchy paths, the diameter of $\pi_W(\gamma_0) \cup \pi_W(\gamma_1)$ is at most $2\lambda_0(3E + \mu +\kappa_0)+ 2\lambda_0$. This makes $\pi_W \circ \gamma$ an unparameterized $(\lambda_0, 2\lambda_0(3E + \mu +\kappa_0)+ 3\lambda_0)$--quasi-geodesic.
\end{itemize}
The above shows that there exists  $\lambda'\geq 1$ depending only on $(\mc{X},\mf{S})$ so that $\pi_W \circ \gamma$ is an unparameterized $(\lambda',\lambda')$--quasi-geodesic for all $W \in \mf{S}$.

Finally we show that $\gamma \colon [a,b] \to \mc{X}$ is a quasi-geodesic with constants depending only on $(\mc{X},\mf{S})$.  Let $t,s \in [a,b]$ and let $u = \gamma(t)$ and $v = \gamma(s)$. Since $\gamma_0, \gamma_1, \gamma_2$ are all $(\lambda_0,\lambda_0)$--quasi-geodesics, we can assume $u$ and $v$ do not lie in the same $\gamma_i$. Without loss of generality we have two cases.

In the first case, $u \in \gamma_0$ and $v \in \gamma_1$. Since $\pi_W \circ \gamma$ is a uniform unparameterized quasi-geodesic, there exists $C \geq 1$ so that \[ d_W(u, x') + d_W(x',v) \stackrel{C,C}{\asymp} d_W(u,v)\] for all $W \in \mf{S}$. By Lemma \ref{lem:distributing_distance_formula}, there is $K\geq 1$ depending only on $(\mc{X},\mf{S})$ so that  \[ d_{\mc{X}}(u, x') + d_{\mc{X}}(x',v) \stackrel{K,K}{\asymp} d_{\mc{X}}(u,v)\] which implies \[ \frac{1}{\lambda_0 K} | t-s| -\frac{2\lambda_0}{K} - K \leq  d_{\mc{X}}(\gamma(t),\gamma(s)) \leq  \lambda_0 |t-s| + 2\lambda_0  \] because $\gamma_0$ and $\gamma_1$ are $ (\lambda_0,\lambda_0)$--quasi-geodesics.

The second case is when $u \in \gamma_0$ and $v\in\gamma_2$. The proof is the same as the first case using the fact that \[ d_W(u, x') + d_W(x',y') +d_W(y',v) \asymp d_W(u,v)\] for all $W \in \mf{S}$ instead. Hence $\gamma$ is a quasi-geodesic with constants depending only on $(\mc{X},\mf{S})$ as desired.
\end{proof}

\subsection{Summary of constants}
Before continuing we summarize the constants associated to the hierarchically hyperbolic space $(\mc{X},\mf{S})$ that we will utilize frequently.

\begin{itemize}
\item $\delta$ is the hyperbolicity constant of $CW$ for each $W \in \mf{S}$.
\item $\kappa_0$ is the consistency constant from Axiom (\ref{item:dfs_transversal}).
\item $E$ is the bound on projections in Axioms (\ref{item:dfs_curve_complexes}), (\ref{item:dfs_transversal}), and (\ref{item:dfs:bounded_geodesic_image}). 
\item $\sigma_0$ is the minimal threshold constant from the distance formula (Theorem \ref{thm:distance_formula}).
\item $\lambda_0$ is the constant such that any two points in $\mc{X}$ can be joined by a $\lambda_0$--hierarchy path (Theorem \ref{thm:monotone_hierarchy_paths}).
\item $\chi$ is the constant from  Proposition \ref{prop:partial_ordering_of_domains} which  bounds the cardinality of any subset of $\mf{S}$ that does not contain a pair of transverse domains.
\item $\theta_0$ is the constant such that for all $\theta\geq \theta_0$ and $Y \subset \mc{X}$, $H_\theta(Y)$ is hierarchically quasiconvex (Lemma \ref{lem:hull}). 
\end{itemize}
We can and shall assume that $E\geq \kappa_0$ and $E \geq \delta$. When we say that a quantity depends on $\mf{S}$, we mean that it depends on any of the above constants.

\section{Constructing hulls with hierarchy paths} \label{sec:constructing_hulls}

In this section, we study hierarchically quasiconvex hulls in hierarchically hyperbolic spaces.  The main result is Theorem \ref{thm: hierarchy paths to hulls} below which says that the hierarchically quasiconvex hull can be constructed by iteratively connecting points with hierarchy paths. While our motivation for such a construction is to establish that {\sqc} subsets are hierarchically quasiconvex (Proposition \ref{prop: quasiconvex implies HQC}) we believe it will have many other applications. At the end of the section, we give an example of such an application by characterizing hierarchical quasiconvexity in terms of the coarse median structure on a hierarchically hyperbolic space.

\begin{defn}[Hierarchy path hull]\label{defn:hierarchy_path_hull}
Let $Y$ be a subset of the hierarchically hyperbolic space $(\mc{X},\mf{S})$.  Define $\Path_\lambda^1(Y)$ to be the union of all $\lambda$--hierarchy paths between points in $Y$. Inductively define $\Path_\lambda^n(Y) = \Path_\lambda^1(\Path_\lambda^{n-1}(Y))$ for all integers $n \geq 2$. For all $\lambda \geq \lambda_0$ and $n\geq 1$, $\Path_\lambda^n(Y) \neq \emptyset$.
\end{defn}

\begin{thm}[Constructing hulls using hierarchy paths]\label{thm: hierarchy paths to hulls}
Let $(\mc{X},\mf{S})$ be a hierarchically hyperbolic space and $N = 2\chi$, where \(\chi\) is as in Proposition \ref{prop:partial_ordering_of_domains}. There exist $\overline{\theta}\geq \theta_0$ and $ \overline{\lambda}\geq\lambda_0$ depending only on $\mf{S}$ such that for all $\theta\geq \overline{\theta}$, $\lambda\geq\overline{\lambda}$ and $Y \subseteq \mc{X}$ \[d_{Haus}(\Path_\lambda^N(Y),H_\theta(Y)) < D\] where $D$ depends only on $\theta$, $\lambda$, and $\mf{S}$.
\end{thm}

In a recent paper, Bowditch independently constructs hulls in coarse medians spaces in a similar manner to the construction in Definition \ref{defn:hierarchy_path_hull} \cite{BowditchConvexity}. Hierarchically hyperbolic spaces are one of the primary examples of coarse median spaces and \cite[Lemma 7.3]{BowditchConvexity} establishes a version of Theorem \ref{thm: hierarchy paths to hulls} for finite subsets of hierarchically hyperbolic spaces. At the end of this section we show that Bowditch's coarse median hull is coarsely equal to the hierarchical quasiconvex hull for any subset of an HHS. This is achieved by using Theorem \ref{thm: hierarchy paths to hulls} to give a new characterization of the hierarchical quasiconvexity in terms of the coarse median structure on a hierarchically hyperbolic space.

The number of iterations of connecting pairs of points by hierarchy paths required by Theorem \ref{thm: hierarchy paths to hulls} is unlikely to be optimal. However, a simple example illustrates that the number of iteration required must increase with the maximal number of pairwise orthogonal domains. Consider the group $\mathbb{Z}^n$ with the standard HHG structure. Let $Y$ be the union of the positive halves of each of the coordinate axes. The hull $H_\theta(Y)$ then coarsely coincides with the positive orthant of $\mathbb{Z}^n$, but $\Path_\lambda^m(Y)$ coarsely coincides with the set of points in the positive orthant where at most $2^m$ coordinates are non-zero. Thus, the number of iterations of $\Path_\lambda^{1}(\cdot)$ required to achieve $H_\theta(Y)$ will be approximately $\log(n)$.

For the remainder of this section, let $(\mc{X},\mf{S})$ be a hierarchically hyperbolic space and $Y \subseteq \mc{X}$. Recall, there exist $\theta_0$ and $\lambda_0$ such that for all $\theta \geq \theta_0$, $H_\theta(Y)$ is hierarchically quasiconvex (Lemma \ref{lem:hull}) and any two points in $\mc{X}$ can be joined by a $\lambda_0$--hierarchy path (Theorem \ref{thm:monotone_hierarchy_paths}). 

The following lemma can be found in \cite[Proposition 6.4.4]{BHS17b} and says for sufficiently large $\theta$, all hierarchically quasiconvex hulls coarsely coincide. We record the proof for completeness.

\begin{lem}[{\cite[Proposition 6.4.4]{BHS17b}}]\label{lem:unique_hulls}
There exists $\overline{\theta}\geq \theta_0$ depending only on $\mf{S}$, such that for all $\theta_1,\theta_2 \geq \overline{\theta}$ \[d_{Haus}(H_{\theta_1}(Y),H_{\theta_2}(Y)) \leq D\] where $D$ depends on $\theta_1$ and $\theta_2$.
\end{lem}

\begin{proof}
Without loss of generality, assume $\theta_0<\overline{\theta}\leq\theta_1 < \theta_2$ with $\overline{\theta}$ to be determined below. By definition $H_{\theta_1}(Y) \subseteq H_{\theta_2}(Y)$. Let $x \in H_{\theta_2}(Y)$. For each $U\in \mf{S}$, $\pi_U(H_{\theta_0}(Y))$ is $K$--{quasiconvex}, where $K$ depends on $\theta_0$ and $\delta$. Let $y_U$ be the closest point projection of $\pi_U(x)$ onto $\pi_U(H_{\theta_0}(Y))$. By Lemma \ref{lem:proj_consistent}, there exist $y \in \mc{X}$ and $\theta'$ depending on $\theta_0$ and $\mf{S}$ such that $d_U(\pi_U(y),y_U) \leq \theta'$. In particular, setting $\overline{\theta}=\theta_0 + \theta'$, we have $y \in H_{\overline{\theta}}(Y) \subseteq H_{\theta_1}(Y) $. To bound $d_\mc{X}(x,y)$, we will uniformly bound $d_U(x,y_U)$ in terms of $\theta_2$ for every $U\in \mf{S}$; the bound on $d_\mc{X}(x,y)$ will then follow from the distance formula (Theorem \ref{thm:distance_formula}). By the definition of \(y_U\) we have \(d_{U}(x,y_U) \leq d_{U}(x,\pi_U(H_{\theta_0}(Y))) +1\). Since $\pi_U(H_{\theta_0}(Y))$ is quasiconvex, contains \(Y\), and is contained in the \(\theta_0\)--neighborhood of \(\hull_{CU}(Y)\), there exists a $D'$ depending only on $\mf{S}$ such that $\hull_{CU}(Y) \subseteq N_{D'}(\pi_U(H_{\theta_0}(Y)))$.
Since \(d_{U}(x, \hull_{CU}(Y)) \leq \theta_2\), we have  \[d_U(x,y_U) \leq d_{U}(x, \pi_U(H_{\theta_0}(Y))) +1\leq \theta_2 + D' +1\] providing the result.
\end{proof}

For the remainder of this section, $\overline{\theta}$ will denote the constant from Lemma \ref{lem:unique_hulls}.

To prove Theorem \ref{thm: hierarchy paths to hulls} we will show for sufficiently large $\theta$ and $\lambda$, we can find $\theta' > \theta$ and $\lambda' > \lambda$ such that
\[\Path^N_\lambda(Y) \subseteq H_{\theta'}(Y) \hspace{1cm} \text{and} \hspace{1cm}H_\theta(Y) \subseteq \Path^N_{\lambda'}(Y) . \]

Theorem \ref{thm: hierarchy paths to hulls} will then follow by applying Lemma \ref{lem:unique_hulls}. The inclusion  $\Path^N_\lambda(Y) \subseteq H_{\theta'}(Y)$ is the following direct consequence of hierarchical quasiconvexity.

\begin{lem}\label{lem:Path_in_hulls}
For each $\lambda,n \geq 1$, there exists $\theta \geq \overline{\theta}$, such that for any $Y\subseteq\mc{X}$ \[\Path_\lambda^n(Y) \subseteq H_\theta(Y).\]
\end{lem}

\begin{proof}
The $n=1$ case follows directly from the definition of $H_\theta(Y)$ and hierarchy paths. We can proceed by induction on $n$ and assume there exists $\theta' \geq \overline{\theta}$  such that $\Path^{n-1}_\lambda(Y) \subseteq H_{\theta'}(Y)$. Let  $x\in \Path_\lambda^n(Y)$. There exist $y_1, y_2 \in \Path^{n-1}_\lambda(Y)$ such that $x$ is on a $\lambda$--hierarchy path from $y_1$ to $y_2$. For each $U \in \mf{S}$, $\pi_U(y_i)$ is within $\theta'$ of $\hull_{CU}(Y)$.  Therefore, quasiconvexity of $\hull_{CU}(Y)$ in $CU$ guarantees there exists a $\theta$ depending only on $\lambda$ and $\theta'$ (which in turn depends on $n$) such that $\pi_U(x)$ is within $\theta$ of $\hull_{CU}(Y)$ and thus $x \in H_\theta(Y)$.
\end{proof}

The other inclusion, $H_\theta(Y) \subseteq \Path^N_{\lambda'}(Y)$, requires two main steps.  First we prove that if $x \in H_\theta(Y)$, then there exists at most $2\chi+1$ points, $x_1,\hdots, x_n$, in $Y$ such that $x \in H_{\theta'}(x_1,\dots, x_n)$  where $\theta'$ depends only on $\theta$ (Lemma \ref{lem:finite_catch}). We then  show that for any finite collection of points $x_1,\hdots,x_n \in \mc{X}$, $H_{\theta'}(x_1,\hdots,x_n) \subseteq \Path_\lambda^{n-1}(x_1,\hdots,x_n)$  where $\lambda$ ultimately depends only on $n$ and $\theta$ (Proposition \ref{prop:hulls_finite_case}). Together, these imply  $H_\theta(Y) \subseteq \Path_\lambda^{2\chi+1}(Y)$.

We start with the first step, which can be thought of  a version of Carath\'eodory's Theorem for HHSs.

\begin{lem} \label{lem:finite_catch}
Let $Y \subseteq \mc{X}$, $\theta \geq \overline{\theta}$, and $\chi$ be as in Proposition \ref{prop:partial_ordering_of_domains}. For each $x \in H_\theta(Y)$, there exist  $x_1,\hdots,x_{\ell+1} \in Y$,  where $1\leq \ell\leq 2\chi$, and $\theta'$ depending only on $\theta$ such that $x \in H_{\theta'}(x_1,\hdots,x_{\ell +1})$.
\end{lem}

\begin{proof}
Let $K = 100(E+2\kappa_0+\theta)$ and $x \in H_\theta(Y)$. If for all $y \in Y,$ $\relevant_{K}(x,y) = \emptyset$, then $x \in H_{K}(y)$ for each $y \in Y$. Thus we can assume there is \(y \in Y\) such that  $\relevant_{K}(x,y) \neq \emptyset$.  

As in Proposition \ref{prop:partial_ordering_of_domains}, we can partition $\relevant_{K}(x,y)$ in subsets  $\mc{U}_1,\hdots,\mc{U}_n$ where $n\leq \chi$. Further, for each $i$, all the elements of $\mc{U}_i$ are pairwise transverse and are totally ordered with respect to the order: $U\leq V$ if  $d_U(\rho^V_U,y) \leq \kappa_0$. Let $U_{i,1}<\cdots< U_{i,k_i}$ be the distinct domains in $\mc{U}_i$. Now for each $i$, there exist $a_i,b_i \in Y$ such that 
$\pi_{U_{i,1}}(x)$ is within $\theta$ of the $CU_{i,1}$ geodesic between $a_i$ and $b_i$. If $a_i$ and $b_i$ project $2\kappa_0+E$ close to $y$ in $CU_{i,1}$, then $d_{U_{i,1}}(x,y) \leq \theta + 4\kappa_0+3E$ which contradicts $U_{i,1} \in \relevant_K(x,y)$.  Thus without loss of generality, $d_{U_{i,1}}(a_i,y) > 
2\kappa_0+E$ and in particular $d_{U_{i,1}}(a_i,\rho_{U_{i,1}}^{U_{i,j}}) >\kappa_0$ for all $j>1$. The total order on $\mc{U}_i$ and the consistency axioms (Axiom \ref{item:dfs_transversal}) therefore ensures that $d_{U_{i,j}}(x,a_i) \leq 2\kappa_0 +E$ for all $1<j\leq k_i$. Thus for each $U_{i,j}$, $x$ projects $\theta+2\kappa_0 +E$ close to the $CU_{i,j}$ 
geodesic between $a_i$ and $b_i$ and $x \in H_{K}(y,a_1,\hdots,a_n,b_1,\hdots,b_n)$.
\end{proof}

Armed with Lemma \ref{lem:finite_catch}, the next step is to prove that  for a finite set of points, the hierarchical hull is contained in the path hull.

\begin{prop}\label{prop:hulls_finite_case}
For each $\theta \geq \overline{\theta}$ and $n\geq2$, there exists $\lambda \geq 1$ such that \[H_\theta(x_1,\hdots,x_n) \subseteq \Path^{n-1}_\lambda(x_1,\hdots x_n)\] for any $n$ distinct points $x_1,\hdots, x_n \in \mc{X}$.
\end{prop}

\begin{proof}
We shall proceed by induction on $n$. First we will show the base case of $n=2$

\begin{claim}[Base Case]\label{claim:hulls_of_pairs_of_points}
	For each $\theta\geq \overline{\theta}$ there exists $\lambda \geq 1$ such that \[H_{\theta}(x,y) \subseteq \Path^1_\lambda(x,y)\] for each $x,y\in\mc{X}$.
\end{claim}

\begin{proof}[Proof of Claim \ref{claim:hulls_of_pairs_of_points}]
	Let $z \in H_\theta(x,y)$, $\gamma_0:[a,b]\to \mc{X}$ be a $\lambda_0$--hierarchy path from $x$ to $z$ and $\gamma_1:[b,c]\to \mc{X}$ is a $\lambda_0$--hierarchy path from $z$ to $y$. We will show that $\gamma=\gamma_0*\gamma_1 \colon [a,c]\to \mc{X}$ is a $\lambda$--hierarchy path from $x$ to $y$, where $\lambda$ depends only on $\theta$. By the definition of $H_\theta(x,y)$ and hyperbolicity of the $CU$'s  we have that $\pi_U(\gamma)$ is an unparameterized $(\lambda_1,\lambda_1)$--quasi-geodesic for each $U\in \mf{S}$, where $\lambda_1$ depends only on $\theta$. Therefore, it suffices to show that $\gamma$ is a $(\lambda,\lambda)$--quasi-geodesic in $\mc{X}$, where $\lambda$ depends only on $\theta$. That is, we need to prove for each $t,s \in [a,c]$ we have \[|t-s| \stackrel{\lambda,\lambda}{\asymp} d_{\mc{X}}(\gamma(t),\gamma(s)). \]
	
	Since $\gamma_0$ and $\gamma_1$ are both $(\lambda_0,\lambda_0)$--quasi-geodesics, we can restrict ourselves to the case where $t \in [a,b)$ and $s\in (b,c]$.  Let $u = \gamma(t)$ and $v = \gamma(s)$. Since $\pi_U(\gamma)$ is a uniform unparameterized quasi-geodesic for each $U \in \mf{S}$, we have that \[d_U(u,z) + d_U(z,v)  \stackrel{C,C}{\asymp} d_U(u,v)\] where $C\geq 1$ depends only on $\theta$.  Hence,  Lemma \ref{lem:distributing_distance_formula}  provides a constant $K \geq 1$ depending only on $\theta$  so that \[  d_{\mc{X}}(u,z) + d_{\mc{X}}(z,v)  \stackrel{K,K}{\asymp} d_{\mc{X}}(u,v).\] Since  $\gamma_0$ and $\gamma_1$ are both $(\lambda_0,\lambda_0)$--quasi-geodesics, we have \[ \frac{1}{\lambda_0 K} | t-s| -\frac{2\lambda_0}{K} - k \leq  d_{\mc{X}}(\gamma(t),\gamma(s)) \leq  \lambda_0 |t-s| + 2\lambda_0  \] as desired.
\end{proof}

We now show the key fact for the inductive step, that the hull of \(n\) points can be obtained by taking the hull on \(n-1\) points, and then considering all the hierarchy paths between this smaller hull and the remaining point. 

\begin{claim} \label{claim:inductive_step}
	Let $x_1,\hdots, x_n \in \mc{X}$, for \(n \geq 2\). If $x \in H_{\theta}(x_1,\hdots,x_n)$ where $\theta\geq \overline{\theta}$, then there exist $\theta'$ and $\lambda$ depending only on $\theta$ and  $y \in H_{\theta'}(x_1,\hdots,x_{n-1})$ such that $x$ is on a $\lambda$--hierarchy path from $x_n$ to $y$.
\end{claim} 

\begin{proof}[Proof of Claim \ref{claim:inductive_step}]
	For $1 \leq i \leq n$, let $A_i = \{ x_1,\hdots,x_i\}$.  For each $U \in \mf{S}$, $\pi_U(H_\theta(A_{n-1}))$ is $R$--quasiconvex where $R$ depends only on $\theta$. Let $y_U$ be the closest point projection of $\pi_U(x)$ to 
	$\pi_U(H_\theta(A_{n-1}))$, $z_U$ be a point in $\hull_{CU}(A_n)$ within $\theta$ of $\pi_U(x)$, and $z'_U$ be the closest point projection of $z_U$ to $\pi_U(H_\theta(A_{n-1}))$. 
	By Lemma \ref{lem:proj_consistent}, there exist $y \in \mc{X}$ and a constant $\theta'$ depending on $\theta$ and $\delta$ such that $d_U(\pi_U(y),y_U) \leq \theta'$.  Further, we can assume $\theta'$ is large enough so that the following hold:
	\begin{enumerate}[(1)]
		\item $\theta'> \theta +\delta+R+1$ 
		\item $y \in H_{\theta'}(A_{n-1})$
		\item \label{contracting} For all $v,w \in CU$, if $d_U(v,w) < d_U(v, H_{\theta}(A_{n-1}))$, then the closest point projection of $v$ and $w$ to $\pi_U(H_{\theta}(A_{n-1}))$ are no more than $\theta'$ apart.
	\end{enumerate}
	
	For each $U\in\mf{S}$, let $\gamma_U$ be a $CU$ geodesic from $\pi_U(x_n)$ to $\pi_U(y)$. We will show that $d_U(x_n,\gamma_U)$ is uniformly bounded for each $U \in\mf{S}$. If $d_U(y_U,z_U) \leq 5\theta' $, then $d_U(x,y_U) \leq 6\theta'$ which implies   $d_U(x,\gamma_U) \leq 7\theta'$. Otherwise 
	$d_U(y_U,z_U) > 5\theta'$ implies that $d_U(x,H_{\theta}(A_{n-1})) > d_U(x,z_U)$ and thus $d_U(y_U,z'_U) \leq \theta'$ by (\ref{contracting}). This implies that  $d_U(z_U,H_\theta(A_{n-1})) > 3\theta'$. Since $z_U \in \hull_{CU}(A_n)$ and $z_U \not\in H_\theta(A_{n-1})$, there exist $D\geq0$ depending only on $\theta$ and  $x_U \in \pi_U(A_{n-1})$ such that $z_U$ is within $D$ of any $CU$ geodesic from $\pi_U(x_n) $ to $x_U$. Further, by increasing $\theta'$, we can assume $D<\theta'$. Take a geodesic triangle with endpoints $\pi_U(x_n)$, $y_U$
	and $x_U$. Since $d_U(z_U,H_\theta(A_{n-1})) > 3\theta'$, it must be the case that $z_U$ is within $2\theta'$ of any $CU$ geodesic from $\pi_U(x_n)$ to $y_U$.
	
	Thus there exists $\theta''$ depending ultimately only on $\theta$, such that  $d_U(x,\gamma_U) \leq \theta''$ for all $U \in \mf{S}$.  Therefore $x \in H_{\theta''}(x_n,y)$ and the statement in Claim \ref{claim:inductive_step} follows from Claim \ref{claim:hulls_of_pairs_of_points}.
\end{proof}

We now finish the proof of Proposition \ref{prop:hulls_finite_case}. 
Let $x \in H_\theta(x_1,\hdots, x_n)$. Claim \ref{claim:inductive_step} shows that there exist a $\lambda'\geq 1$ and $\theta' \geq \overline{\theta}$  such that $x$ is on a $\lambda'$--hierarchy path from $x_n$ to a point in $H_{\theta'}(x_1,\hdots,x_{n-1})$. By induction, there exists $\lambda\geq \lambda'$ such that $H_{\theta'}(x_1,\hdots,x_{n-1}) \subseteq \Path_\lambda^{n-2}(x_1,\cdots, x_{n-1})$ and therefore $x \in \Path^{n-1}_\lambda(x_1,\hdots, x_n)$.
\end{proof}

We can now finish the proof of Theorem \ref{thm: hierarchy paths to hulls}.

\begin{proof}[Proof of Theorem \ref{thm: hierarchy paths to hulls}]
Recall, we need to show that for all sufficiently large $\theta$ and $\lambda$, $H_\theta(Y)$ coarsely coincides with $\Path_\lambda^N(Y)$ where $N=2\chi$. First we will show that for all $\theta \geq \overline{\theta}$, there exists $\lambda\geq 1$ such that $H_\theta(Y) \subseteq \Path^N_\lambda(Y)$.

Let $x \in H_\theta(Y)$ and let $x_1,\hdots,x_{\ell+1}$ be the finite number of points in $Y$ provided by Lemma \ref{lem:finite_catch}. By Proposition \ref{prop:hulls_finite_case}, there exists $\lambda$ depending on $\theta$ such that $x \in \Path^{\ell}_\lambda(x_1,\hdots, x_{\ell+1}) \subseteq \Path^{\ell}_\lambda(Y) \subseteq \Path^N_\lambda(Y)$. Thus $H_\theta(Y) \subseteq \Path^N_\lambda(Y)$.

Now, fix $\overline{\lambda}\geq \lambda_0$ such that $H_{\overline{\theta}}(Y) \subseteq \Path^N_{\overline{\lambda}}(Y)$. If $\theta\geq \overline{\theta}$ and $\lambda \geq \overline{\lambda}$, then by Lemma \ref{lem:Path_in_hulls} there exists $\theta'>\overline{\theta}$  such that \[H_{\overline{\theta}}(Y) \subseteq \Path_\lambda^N(Y) \subseteq H_{\theta'}(Y).\] The conclusion now follows by Lemma \ref{lem:unique_hulls}.
\end{proof}

The primary use of Theorem \ref{thm: hierarchy paths to hulls} in this paper is the following proof that hierarchically quasiconvex subsets are exactly the subsets that are ``quasiconvex with respect to hierarchy paths." From this it immediately follows that all {\sqc} subsets are hierarchically quasiconvex.

\begin{prop}\label{prop: quasiconvex implies HQC}
Let $(\mc{X},\mf{S})$ be a hierarchically hyperbolic space. A subset $Y \subseteq \mc{X}$ is $k$--hierarchically quasiconvex if and only if there exists a function $R \colon [1,\infty) \rightarrow [0,\infty)$ such that if $\gamma$ is a $\lambda$--hierarchy path with endpoints on $Y$, then $\gamma \subseteq N_{R(\lambda)}(Y)$ where $k$ and $R$  each determines the other.  In particular, if $Y$ is $Q$--{\sqc}, then $Y$ is $k$--hierarchically quasiconvex where $k$ is determined by $Q$.
\end{prop}

\begin{proof}
The proof of the forward implication follows directly from the definition of hierarchical quasiconvexity and hierarchy path. Assume there exists a function $R \colon [1,\infty) \rightarrow [0,\infty)$ such that if $\gamma$ is a $\lambda$--hierarchy path with endpoints on $Y$, then $\gamma \subseteq N_{R(\lambda)}(Y)$. The first condition of hierarchical quasiconvexity now follows from the existence of hierarchy paths (Theorem \ref{thm:monotone_hierarchy_paths}), the coarse Lipschitzness of the projection maps (Axiom \ref{item:dfs_curve_complexes}), and the hyperbolicity of the $CU$'s. For the second condition,  observe that the hypothesis implies there exists a bound on the Hausdorff distance between $Y$ and $\Path^n_\lambda(Y)$ depending only on $R$, $n$, and $\lambda$. Thus by Theorem \ref{thm: hierarchy paths to hulls}, for each $\theta \geq \overline{\theta}$, there exists $D_\theta$ such that $d_{Haus}(H_\theta(Y),Y) \leq D_\theta$. Let $\kappa > 0$ and $x\in\mc{X}$ such that $d_U(x,Y) \leq \kappa$ for all $U \in \mf{S}$. Thus $x \in H_\theta(Y)$ for each $\theta \geq \kappa+\overline{\theta}$. Let $k(\kappa) = D_{\overline{\theta}+\kappa}$, then $d_\mc{X}(x,Y) \leq k(\kappa)$ and $Y$ is hierarchically quasiconvex.
\end{proof}

\begin{rem}
If $\mc{X}$ is a hyperbolic space, there exists many HHS structures on $\mc{X}$ (see \cite{Spriano2018}). In this case, Proposition \ref{prop: quasiconvex implies HQC} recovers \cite[Proposition 3.5]{Spriano2018} which states that a subset $Y\subseteq \mc{X}$ is quasiconvex if and only if $Y$ is hierarchically quasiconvex in any of the HHS structures on $\mc{X}$.
\end{rem}

\subsection{Hulls and coarse medians}\label{subsec:hulls_and_coarse_medians}

We now take a small detour from the main thrust  of the paper to highlight an application of Theorem \ref{thm: hierarchy paths to hulls}  and discuss the relation of our work in this section to the hulls in coarse median spaces constructed in \cite{BowditchConvexity}.  

In \cite{BowditchCoarseMedian}, Bowditch axiomatized the notion of a coarse center of three points in a metric space and defined \emph{coarse median spaces} as metric spaces where every triple of points has such a coarse center. Bowditch observed that all hierarchically hyperbolic spaces are coarse median spaces;  see also \cite[Theorem 7.3]{BHS17b}. The salient property of the coarse median structure of an HHS is the following fact.

\begin{lem}[See proof of {\cite[Theorem 7.3]{BHS17b}}]\label{lem:definition coarse median map}
Let \((\X, \mf{S})\) be a hierarchically hyperbolic space. There exist $\mu>0$ and a map \(\mf{m}\colon \X \times \X \times \X \to \X\)  with the property that for every  \((x,y,z) \in \X^3\) and \(U\in \mf{S}\), the projection \(\pi_U(\mf{m}(x,y,z))\) is within $\mu$ of all three sides of any $CU$ triangle with vertices \(\pi_U(x), \pi_U(y), \pi_U(z)\).
\end{lem}

We call the point $\mf{m}(x,y,z)$ the \emph{coarse center} of $x$, $y$, and $z$. There is a natural notion of convexity for coarse median spaces,  which we formulate in the hierarchically hyperbolic setting as follows.

\begin{defn}[Coarse median quasiconvexity]
Let \( (\X,\mf{S}) \) be an HHS. A subset \(Y\) of \(\X\) is said to be \emph{\(Q\)--median quasiconvex} if for every \(y,y' \in Y\) and \(x \in \X\) we have \(\mf{m}(y,y', x) \in N_Q (Y)\).
\end{defn}

Behrstock-Hagen-Sisto showed that a hierarchically quasiconvex subset is median quasiconvex in  \cite[Proposition 7.12]{BHS17b}. Using Theorem \ref{thm: hierarchy paths to hulls}, we establish the converse.

\begin{prop}\label{prop:median_convexity}
Let $(\mc{X},\mf{S})$ be an HHS and $Y \subseteq \mc{X}$. $Y$ is $k$--hierarchically quasiconvex if and only if $Y$ is $Q$--median quasiconvex where $k$ and $Q$ each determines the other. 
\end{prop}

\begin{proof}
Let $Y$ be a $Q$--median quasiconvex subset of the HHS $(\mc{X},\mf{S})$ and $\gamma$ be a $\lambda$--hierarchy path with endpoints $y_1,y_2 \in Y$.   If $x \in \gamma$, then $d_U\bigl(x, \mf{m}(y_1,y_2,x) \bigr)$ is uniformly bounded in terms of $\lambda$ and $\mf{S}$ for each $U \in \mf{S}$. By  the distance formula (Theorem \ref{thm:distance_formula}), \(d_{\mc{X}}(x, \mf{m}(y_1,y_2,x) \bigr)\) is also uniformly bounded. Since \(Y\) is  median quasiconvex, this implies that there exist $R(\lambda)$ such that $d_\mc{X}(x,Y) \leq R(\lambda)$. In particular, $\gamma \subseteq N_{R(\lambda)}(Y)$ and $Y$ is $k$--hierarchically quasiconvex, with $k$ determined by $Q$, by Proposition \ref{prop: quasiconvex implies HQC}.
\end{proof}

If $Y \subseteq \mc{X}$, let $M(Y)$  denote the \emph{coarse median hull} defined in \cite[Proposition 6.2]{BowditchConvexity}. Proposition \ref{prop:median_convexity} implies the following corollary that extends \cite[Lemma 7.3]{BowditchConvexity} in the special case of hierarchically hyperbolic spaces.

\begin{cor}
Let $(\mc{X},\mf{S})$ be an HHS and $Y \subseteq \mc{X}$.  For each $\theta \geq \theta_0$, there exists $D$ depending only on $\theta$ and $\mf{S}$ such that \[d_{Haus}\left(H_\theta(Y),{M}(Y) \right) \leq D. \]
\end{cor}

\begin{proof}
Let $Y \subseteq \mc{X}$ and $\theta \geq \theta_0$.  By Proposition \ref{prop:median_convexity}, $H_\theta(Y)$ is $Q_1$--median quasiconvex for some $Q_1$ depending on $\theta$ and $\mf{S}$. By \cite[Proposition 6.2]{BowditchConvexity} $M(Y)$ is $Q_2$--median quasiconvex, where $Q_2$ depends only on $\mf{S}$, and there exists $D_1$ depending on $\theta$ such that $M(Y) \subseteq N_{D_1}(H_\theta(Y))$.  By Proposition \ref{prop:median_convexity}, $M(Y)$ is $k$--hierarchically quasiconvex where $k$ depends only on $\mf{S}$. By the second condition in Definition \ref{defn:hierarchical_quasiconvexity}, there exists $D_2$ depending on $\theta$ and $\mf{S}$ such that $H_\theta(Y) \subseteq N_{D_2}(M(Y))$.
\end{proof}

\section{Characterization of {\sqc} subsets in HHSs}\label{sec:characterizing_quasiconvex_subsets}

We now turn our attention to the main objective of this paper, characterizing the {\sqc} subsets of hierarchically hyperbolic spaces. From now on we shall restrict our attention to HHSs with the \emph{bounded domain dichotomy}; a minor regularity condition satisfied by all HHGs as well as Teichm\"uller space with either the Weil-Petersson or Teichm\"uller metric and the fundamental groups of $3$--manifolds without Nil or Sol components.

\begin{defn}[Bounded domain dichotomy]\label{defn: bounded domain dichotomy}
A hierarchically hyperbolic space $(\mc{X},\mf{S})$ has the \emph{$B$--bounded domain
	dichotomy} if there exists $B>0$ such that for all $U \in\mf{S}$, if $\diam(CU)> B$, then $\diam(CU) = \infty$.
\end{defn}

The key to characterizing the {\sqc} subsets of hierarchically hyperbolic spaces is to determine what the projection of a {\sqc} subset to each of the associated hyperbolic spaces looks like. The property that characterizes the projection of {\sqc} subsets is the following orthogonal projection dichotomy.

\begin{defn}[Orthogonal projection dichotomy]\label{defn: orthogonal projection dichotomy}
For $B \geq 0$, a subset $Y$ of an HHS $(\mc{X},\mf{S})$ has the \emph{$B$--orthogonal projection dichotomy} if  for all $U,V\in\mf{S}$ with $U\perp V$, if $\diam(\pi_U(Y)) >B $ then $CV \subseteq N_B(\pi_V(Y))$.
\end{defn}

From now on, when we consider an HHS with the \(B_0\)--bounded domain dichotomy and a subspace with the \(B\)--orthogonal projection dichotomy, we will assume that \(B \geq B_0\).

We can now state our characterization of {\sqc} subsets of hierarchically hyperbolic spaces with the bounded domain dichotomy.

\begin{thm}[Characterization of {\sqcity}]\label{thm:classification_of_quasiconvex_subset}
Let $(\X,\mathfrak S)$ be a hierarchically hyperbolic space with the bounded domain dichotomy and $Y \subseteq \X$. Then the following are equivalent:
\begin{enumerate}
	\item \label{item:contracting} $Y$ is an $(A,D)$--contracting subset.
	\item \label{item:quadratic} The lower relative divergence of $\X$ with respect to $Y$ is at least quadratic.
	\item \label{item:superlinear} The lower relative divergence of $\X$ with respect to $Y$ is completely superlinear.
	\item \label{item:quasiconvex} $Y$ is $Q$--{\sqc}.
	\item \label{item:projections} $Y$ is $k$--hierarchically quasiconvex and has the $B$--orthogonal projection dichotomy.
\end{enumerate}
Moreover, the pair $(A,D)$ in part (\ref{item:contracting}), the convexity gauge $Q$ in part (\ref{item:quasiconvex}), and the pair $(k,B)$ in part (\ref{item:projections}) each determines the other two.
\end{thm} 

The work in Section \ref{sec: contracting and relative divergence}, showed that the implications \[(\ref{item:contracting}) \implies (\ref{item:quadratic}) \implies ( \ref{item:superlinear}) \implies (\ref{item:quasiconvex})\] hold in any quasi-geodesic space and that the pair $(A,D)$ determines $Q$. Further Proposition \ref{prop: quasiconvex implies HQC} showed that every $Q$--{\sqc} subset of a hierarchically hyperbolic space is $k$--hierarchically quasiconvex with $Q$ determining $k$. Thus in the next two subsections, we only need to prove the following:
\begin{itemize}
\item If $Y$ is $Q$--{\sqc}, then there exists $B>0$ determined by $Q$ such that $Y$ has the $B$--orthogonal projection dichotomy (Section \ref{subsec: quasiconvex subsets have the orthogonal projection dichotomy}).
\item If $Y$ is $k$--hierarchically quasiconvex and has the $B$--orthogonal projection dichotomy, then $Y$ is $(A,D)$--contracting where $(A,D)$ is determined by $(k,B)$ (Section \ref{subsec: Contracting Subsets in HHSs}).
\end{itemize}

Before beginning the proof, we record of the following corollary to Theorem \ref{thm:classification_of_quasiconvex_subset} that  allows us to characterize stable embeddings.

\begin{cor}\label{cor:classification_of_stable_subsets}
Let $(\mc{X},\mf{S})$ be an HHS with the bounded domain dichotomy and let $i:Y \to \mc{X}$ be a quasi-isometric embedding from a uniform quasi-geodesic space $Y$ to $\mc{X}$. The following are equivalent:
\begin{enumerate}
	\item $i$ is a stable embedding.
	\item $Z=i(Y)$ is hierarchically quasiconvex and there exists a $B>0$ such that for all $U,V\in\mf{S}$ with $U\perp V$, if $\diam(\pi_U(Z))>B$, then $\diam(CV)<B$.
\end{enumerate}
\end{cor}

\begin{proof}
By \cite[Corollary 2.16]{BHS_HHS_Quasiflats}, an HHS $(\mc{Z},\mf{T})$ is hyperbolic if and only if there exists $B$ such that for all $U,V\in \mf{T}$ with $U\perp V$, either $\diam\bigl(\pi_U(\mc{Z})\bigr) <B$ or $\diam\bigl(\pi_V(\mc{Z})\bigr) <B.$ By Proposition \ref{defn:stable}, $i$ is a stable embedding if and only if the image $Z=i(Y)$ is {\sqc} in $\mc{X}$ and hyperbolic. The equivalence follows from these observations and the fact that hierarchically quasiconvex subsets inherit the hierarchy structure from the ambient space as described in \cite[Proposition 5.6]{BHS17b}.
\end{proof}

Corollary \ref{cor:classification_of_stable_subsets} should be compared with \cite[Corollary 6.2]{ABD}. If $(\mc{X},\mf{S})$ has extra assumption of unbounded products required in \cite[Corollary 6.2]{ABD}, then  Corollary \ref{cor:classification_of_stable_subsets} can be immediately improved to \cite[Corollary 6.2]{ABD}. However, Corollary \ref{cor:classification_of_stable_subsets} is a strict expansion of \cite[Corollary 6.2]{ABD} as many HHS structures do not have unbounded products. Naturally occurring HHS structures without unbounded products can be found in right angled Coxeter groups and the Weil-Petersson metric on Teichm\"uller space. We briefly describe these structures in Section \ref{sec: Examples}.


\subsection{{\SQC} subsets have the orthogonal projection dichotomy}\label{subsec: quasiconvex subsets have the orthogonal projection dichotomy}

In this subsection, we provide the implication (\ref{item:quasiconvex}) to (\ref{item:projections}) in Theorem \ref{thm:classification_of_quasiconvex_subset}.
Our focus will be on studying the following set of domains.
\begin{defn}\label{def:S_star}
Define $\mf{S}^*$ to be the set of domains $U \in \mf{S}$ such that $\diam(CU) = \infty$ and there exists $V \in \mf{S}_U^\perp$ such that $\diam(CV) = \infty$.
\end{defn}

For each $U \in \mf{S}^*$ we have that both factors of the product region $\P_U$ have infinite diameter. In particular, if $\mf{S}^* = \emptyset$ and $\mf{S}$ has the bounded domain dichotomy, then $(\mc{X},\mf{S})$ is hyperbolic by \cite[Corollary 2.16]{BHS_HHS_Quasiflats}. Thus the intuition for restricting our attention to these domains is that the domains in $\mf{S}^*$ are the source of non-hyperbolic behavior in $(\mc{X},\mf{S})$.

The crucial step to proving {\sqc} subsets have the orthogonal projection dichotomy is the following proposition that establishes a sort of orthogonal projection dichotomy for the product regions of domains in $\mf{S}^*$.

\begin{prop}\label{prop:ortho_proj_for_product_regions}
Let $(\mc{X},\mf{S})$ be an HHS with the bounded domain dichotomy and \(Y\subseteq \mc{X}\) be a \(Q\)--{\sqc} subset. There is a constant $B_0>0$ depending on $\mf{S}$ and \(Q\) such that for all $B\geq B_0$ and  \(U \in \mf{S}^*\) we have \[\diam(\pi_U(Y)) > B \implies \P_U \subseteq N_B (\gate_{\P_U}(Y)).\]

\end{prop}

Since \(U\) is in \( \mf{S}^\ast\), the product region \(\P_U\) coarsely coincides with the product of two infinite diameter metric spaces. The proof of Proposition \ref{prop:ortho_proj_for_product_regions} is therefore motivated by the situation described in Figure \ref{fig: Example of spiral path in R2}. Namely, if $Y$ is a subset of the product of two infinite-diameter metric spaces, then either \(Y\) coarsely coincide with the whole product or there exists a quasi-geodesic $\gamma$ with endpoints on $Y$ and fixed constants such that there are points of \(\gamma\) whose distance to \(Y\) is comparable to $\diam(Y)$. Thus if $Y$ is $Q$--{\sqc}, then either $Y$ has bounded diameter or it coarsely covers the entire product.
\begin{figure}[H] \begin{tikzpicture}[line cap=round,line join=round,>=triangle 45,x=1cm,y=1cm]
	\draw [line width=1.5pt,color=blue] (-2.5,-1)-- (2.5,-1);
	\draw [line width=1pt] (-1.5,-1)-- (-1.5,1.5);
	\draw [line width=1pt] (-1.5,1.5)-- (1.5,1.5);
	\draw [line width=1pt] (1.5,1.5)-- (1.5,-1);
	\draw [line width=1.5pt,dashed, color=blue] (-2.5,-1)-- (-3.4,-1);
	\draw [line width=1.5pt,dashed, color=blue] (2.5,-1)-- (3.4,-1);
	\begin{scriptsize}
		\draw [fill=black] (-1.5,-1) circle (1.5pt);
		\draw [fill=black] (1.5,-1) circle (1.5pt);
		\draw [color=blue] (0,-.7) node {$Y$};
		\draw (0,1.3) node{$\gamma$};
	\end{scriptsize}
\end{tikzpicture}\caption{\label{fig: Example of spiral path in R2}  In \(\field{R}^2\) (equipped with the \(\ell_1\)-metric)  consider \(Y\) to be the \(x\)-axis. Let \(\gamma\) be the  $(3,0)$--quasi-geodesic consisting of three sides of a square with the fourth side on \(Y\). While the quasi-geodesic constants do no change, increasing the distance between the endpoints of \(\gamma\) produces points of \(\gamma\) arbitrarily far away from \(Y\). }
\end{figure}
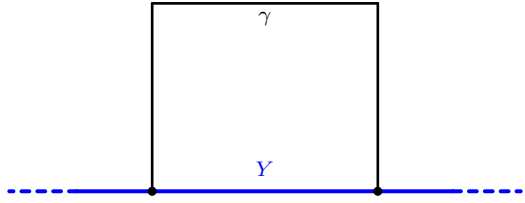

In Proposition \ref{prop: Hung's construction}, we prove that a similar situation holds for \(\P_U\). We show if  $\diam(\pi_U(Y))$ is sufficiently large and \(Y\) does not coarsely coincide with \(\P_U\), then we can find a uniform constant quasi-geodesic with endpoints on $\gate_{\P_U}(Y)$ that contains points relatively far from $\gate_{\P_U}(Y)$. To utilize this to prove Proposition \ref{prop:ortho_proj_for_product_regions}, we must promote this statement on \(\gate_{\P_U}(Y)\) to a statement on \(Y\). Specifically, we show that we can realize every quasi-geodesic of \(\P_U\) with endpoints on \(\gate_{\P_U}(Y)\) as a segment of a quasi-geodesic with endpoints on \(Y\), while maintaining uniform quasi-geodesic constants (Lemma \ref{lem: substituting quasi-geodesic}). This yields a quasi-geodesic with endpoints on \(Y\) that contains a point \(x\) of \(\P_U\) such that $d_\mc{X}(x,\gate_{\P_U}(Y))$ is comparable with \(\diam(\gate_{\P_U}(Y))\). If \(Y\) is {\sqc}, the bridge theorem (Theorem \ref{thm:parallel_gates}) implies that $d_\mc{X}(x,\gate_{\P_U}(Y))$ also provides a lower bound on the distance between \(x\) and \(Y\). However, since \(Y\) is {\sqc}, the distance between \(x\) and \(Y\) is uniformly bounded. Hence, if \(Y\) does not coarsely cover \(\P_U\), we obtain that \(\gate_{\P_U}(Y)\) must have bounded diameter which contradicts the assumption on $\diam(\pi_U(Y))$.

We  begin by describing a particularly nice class of paths in product spaces and show that they are quasi-geodesics (Lemma \ref{lem: spiral paths are quasi-geodesics}).

\begin{defn}[Spiral path]\label{defn: spiral path} 
Let \(X\) and \(Y\) be \((K,L)\)--quasi geodesic metric spaces, and let \(Z = X\times Y\) be equipped with the \(\ell_1\)--metric.
A \emph{spiral path} \(\gamma\) in \(Z\) is the concatenation \(\gamma = \gamma_1 \ast \cdots \ast \gamma_n\) of \((K,L)\)--quasi-geodesic of \(Z\) satisfying the following. 
\begin{itemize}
	\item Every \(\gamma_i\) is of the form \(\eta \times c_{y}\) or \(c_{x}\times \delta\) where \(\eta\) (resp. \(\delta\)) is a \((K,L)\)--quasi-geodesic of \(X\) (resp. \(Y\)) and \(c_{x_0}\) (resp. \(c_{y_0}\)) is the constant function with value \(x_0 \in X\) (resp. \(y_0 \in Y\)).
	\item For every \(i\), if  \(\gamma_i\) is constant on the $X$ (resp. $Y$) factor of $Z = X\times Y$, then  \(\gamma_{i+1}\) is constant on the $Y$ (resp. $X$) component of $Z = X \times Y$. 
\end{itemize}
A spiral path \(\gamma= \gamma_1 \ast \cdots \ast \gamma_n\) has \emph{slope \(N\)} if for every \(i\in \{1, \dots , n-2\}\) we have:
\[d(\gamma_{i+1}^+, \gamma_{i+1}^-) \geq N d(\gamma_i^+, \gamma_i^-),\] where \(\gamma_j^\pm\) are the endpoints of \(\gamma_j\). Note that the distance between the endpoints of \(\gamma_n\) can be arbitrary. 
\end{defn}

\begin{lem}[Spiral paths are quasi-geodesics]\label{lem: spiral paths are quasi-geodesics}
For each \(K\geq 1\), \(L\geq 0\) there are constants \(K', L'\) such that the following holds. 
Let \(X, Y\) be \((K,L)\)--quasi-geodesic metric space. If \(\gamma= \gamma_1 \ast \cdots \ast \gamma_n\) is a spiral path of slope \(N>4K^2\) in \(Z=X \times Y\), such that the endpoints of \(\gamma_1\) are at least \(3K^2L +1\) far apart, then \(\gamma\) is a \((K', L')\)--quasi-geodesic of \(X \times Y\).
\end{lem}

The following proof is essentially the same as showing the logarithmic spiral in $\mathbb{R}^2$ is a quasi-geodesic.  However, as we were not able to find a sufficient reference in the literature, we have included it in the interest of completeness. 

\begin{proof}
Let  \(\gamma = \gamma_1* \cdots * \gamma_n \colon [a_0, a_{n}] \rightarrow Z\) be spiral path of slope $N>4K^2$ and let \(a_1 < \cdots < a_{n}\) be points in \([a_0, a_{n}]\) such that \(\gamma_i = \gamma|_{[a_{i-1}, a_{i}]}\).

Let \(t_1, t_2 \in [a_0, a_n]\). We claim that \begin{align}\label{equation 0' spiral path} d\bigl(\gamma(t_1), \gamma(t_2)\bigr) \leq (K+1) |t_2-t_1| + 2L.\end{align} 
Since each \(\gamma_i\) is a \((K,L)\)--quasi-geodesic of \(Z\) for each $i$, we only need to consider the case where \(t_1 \in [a_k, a_{k+1}]\) and \(t_2 \in [a_j, a_{j+1}]\) with \(j-k \geq 1\).  By the choice on the distance between endpoints of $\gamma_1$ and the slope $N$ we have $d(\gamma(a_{i-1}),\gamma(a_i))> 3K^2L +1$, which implies  $|a_i-a_{i-1}|>L$. Therefore, \[|t_2-t_1|\geq |a_j-a_{k+1}|\geq (j-k-1)L.\] Since each $\gamma_i$ is $(K,L)$--quasi-geodesic we have
\begin{align*}d\bigl(\gamma(t_1), \gamma(t_2)\bigr) &\leq K |t_2-t_1| + (j-k+1)L\\&\leq (K+1) |t_2-t_1| + 2L.
\end{align*}

\noindent The remainder of the proof will show $|t_2-t_1| \preceq d\bigl(\gamma(t_1), \gamma(t_2)\bigr)$. 

For every \(i\), \(\gamma_i\ast \gamma_{i+1}\) is a \((K,2L)\)--quasi-geodesic of \(Z\), so  we only need to consider the case where \(t_1 \in [a_k, a_{k+1}]\) and \(t_2 \in [a_j, a_{j+1}]\) with \(j-k \geq 2\) as shown below. \vspace{2mm}

\begin{figure}[H]
	\begin{tikzpicture}[line cap=round,line join=round,>=triangle 45,x=1cm,y=1cm] 
		\draw [line width=.5mm] (-2,0)-- (-3,0);
		\draw [line width=.5mm] (-3,0)-- (-3,2.5);
		\draw [line width=.5mm] (-3,2.5)-- (5,2.5);
		\draw [line width=.5mm] (5,2.5)-- (5,-1.5);
		\draw [fill=black] (-2,0) circle (2.5pt);
		\draw[color=black] (-1.2,0) node {$\gamma(a_{j-3})$};
		\draw [fill=black] (-3,0) circle (2.5pt);
		\draw[color=black] (-3.74,0.04) node {$\gamma(a_{j-2})$};
		\draw [fill=black] (-3,2.5) circle (2.5pt);
		\draw[color=black] (-3.7,2.7) node {$\gamma(a_{j-1})$};
		\draw [fill=black] (5,2.5) circle (2.5pt);
		\draw[color=black] (5.7,2.7) node {$\gamma(a_j)$};
		\draw [color = blue, fill=blue] (-2.3,0) circle (2pt);
		\draw[color=blue] (-2.3,0.4) node {$\textcolor{blue}{\gamma(t_1)}$};
		\draw [fill=black] (5,-1.5) circle (2.5pt);
		\draw [color = blue, fill=blue] (5,1.6) circle (2pt);
		\draw[color=blue] (5.7,1.6) node {$\textcolor{blue}{\gamma(t_2)}$};
		\draw[color=black] (5.8,-1.4) node {$\gamma(a_{j+1})$};
		\draw[color=black] (-2.5,1.5) node {$\gamma_{j-1}$};
		\draw[color=black] (-2.5,-.4) node {$\gamma_{j-2}$};
		\draw[color=black] (1.14,2) node {$\gamma_j$};
		\draw[color=black] (4.6,.5) node {$\gamma_{j+1}$};
\end{tikzpicture}\end{figure}

\noindent We encourage the reader to refer to the above diagram as they follow the remainder of the proof.

By the triangle inequality we have
\begin{align}\label{equation 1' spiral path}
	d(\gamma(t_2), \gamma(t_1)) \geq d(\gamma(t_2), \gamma(a_{j-1})) - d(\gamma(a_{j-1}),\gamma(t_1)).
\end{align}
The remainder of the proof has two parts.  First we show that, $d(\gamma(t_2), \gamma(a_{j-1}))$  is much larger than $d(\gamma(a_{j-1}),\gamma(t_1))$ so that \[d(\gamma(t_2), \gamma(t_1)) \succeq d(\gamma(t_2), \gamma(a_{j-1})) \succeq |t_2-a_{j-1}|.\] We then finish by showing that $ |t_2 - a_{j-1}| \succeq |t_2-t_1|$.

To simplify notation let \(\ell(\gamma_i) = d(\gamma(a_{i-1}),\gamma(a_i))\). The slope condition then says \(\frac{1}{N}\ell(\gamma_{i}) >  \ell(\gamma_{i-1})\) for each $1\leq i \leq n-1$.
Since $N > 4K^2$,  we can iteratively apply the slope condition to get
\[
\sum_{i=1}^{j-1} \ell(\gamma_i) \leq \left(\frac{1}{N^{j-2}} + \cdots +\frac{1}{N} + 1 \right)\ell(\gamma_{j-1}) \leq 2 \ell(\gamma_{j-1})
\leq \frac{2}{N} \ell(\gamma_{j}). \tag{3}\label{geometric series spiral path'}
\]

\noindent From the triangle inequality and the fact $|a_{k+1}-t_1| \leq |a_{k+1}-a_k|$ we have
\begin{align*}
	d(\gamma(t_1), \gamma(a_{j-1}))&\leq d(\gamma(t_1), \gamma(a_{k+1}))+ \sum_{i=k+2}^{j-1} \ell(\gamma_i)\\
	&\leq K|a_{k+1}-a_k|+L + \sum_{i=k+2}^{j-1} \ell(\gamma_i)\\
	&\leq K\bigl(K\ell(\gamma_{k+1})+KL\bigr)+L+\sum_{i=k+2}^{j-1} \ell(\gamma_i)\\
	&\leq K^2\biggl(\sum_{i=k+1}^{j-1}\ell(\gamma_i)\biggr)+2K^2L.\\ 
\end{align*}

\noindent Then by applying Inequality (\ref{geometric series spiral path'}) we have
\[
d(\gamma(t_1), \gamma(a_{j-1}))\leq\biggl(\frac{2K^2}{N}\biggr) \ell(\gamma_{j}) +2K^2L
\leq\frac{1}{2}d(\gamma(t_2),\gamma(a_{j-1}))+2K^2L.\]
Substituting this into Inequality (\ref{equation 1' spiral path}) produces
\[
d(\gamma(t_2), \gamma(t_1)) \geq \frac{1}{2}d(\gamma(t_2), \gamma(a_{j-1}))-2K^2L.
\]
\noindent We can then use the fact that \(\gamma_j \ast \gamma_{j+1}\) is a \((K, 2L)\)--quasi-geodesic to obtain
\begin{align}
	d(\gamma(t_2), \gamma(t_1)) &\geq \frac{1}{2}d(\gamma(t_2), \gamma(a_{j-1}))-2K^2L  \nonumber\\
	&\geq \frac{1}{2}\left(\frac{1}{K}|t_2 - a_{j-1}| -2L \right)-2K^2L \nonumber\\
	&\geq \frac{1}{2K}|t_2 - a_{j-1}|-3K^2L. \nonumber \tag{4} \label{equation 2' spiral path}
\end{align}

We now show that \(|t_2 - a_{j-1}| \succeq |t_2-t_1|\), which completes the proof by Inequality (\ref{equation 2' spiral path}) . Since we required that \(\ell(\gamma_1) > 3K^2L+1\) and $N>4K^2$,  we have \(\frac{1}{K}|a_i - a_{i-1}| > 2L\) for each \(i\). This implies \[\ell(\gamma_i)\geq \frac{1}{K}|a_i- a_{i-1}| - L > \frac{1}{2K}|a_i - a_{i-1}|.\] In particular, using Inequality (\ref{geometric series spiral path'}) we obtain:

\begin{align*}
	\frac{2}{N}\left(K|a_{j}-a_{j-1}| + L\right) &\geq \frac{2}{N} \ell(\gamma_{j})\\
	&\geq \sum_{i=1}^{j-1} \ell(\gamma_i)  \\
	&\geq  \sum_{i=1}^{j-1} \frac{1}{2K}|a_{i}-a_{i-1}| \\
	&\geq  \frac{1}{2K}|a_{j-1}- t_1|.
\end{align*}

\noindent Hence we have
\[
|a_{j-1}-t_1| \leq \frac{4K^2}{N}|a_{j}-a_{j-1}|+\frac{4KL}{N} \leq |a_{j}-a_{j-1}| + L
\]
\noindent and we can conclude
\begin{align*}
	|t_2-t_1| &= |t_2-a_j| + |a_j-a_{j-1}| + |a_{j-1} -t_1|\\
	& \leq  |t_2-a_j| +2 |a_j-a_{j-1}| +L\\
	&\leq 3|t_2-a_{j-1}| + L.
\end{align*}

Combining this with Inequality (\ref{equation 0' spiral path}) and Inequality (\ref{equation 2' spiral path}), we obtain that there are constants \(K'\) and \(L'\) depending on $K$ and $L$ such that 
\[\frac{1}{K'} (t_2 - t_1) -L' \leq d(\gamma(t_2), \gamma(t_1)) \leq K' (t_2 - t_1)+L'. \]
\end{proof}

For the remainder of this section $(\mc{X},\mf{S})$ will be an HHS with the bounded domain dichotomy and \( \mf{S}^*\)  is as in Definition \ref{def:S_star}.
Recall, for each \(U \in \mf{S}\), the space $\PF_U\times \PE_U$ consists of tuples $a = (a_V)$, where $V \in \mf{S}_U \cup \mf{S}_U^{\perp}$, and that \(\P_U\) is defined as the image of  $\phi_U \colon \PF_U \times \PE_U \rightarrow \mc{X}$. By restricting to a choice of factor, we can endow $\PF_U$ and $\PE_U$ with the subspace metric of their images under $\phi_U$. While this relies on the choice of factor, the distance formula (Theorem \ref{thm:distance_formula}) says any two choices result in uniformly quasi-isometric metric spaces. Given $a,b \in \PF_U\times \PE_U $ we use $d_{V}(a,b)$ to denote $d_{V}(a_V,b_V)$, where $V \in \mf{S}_U \cup \mf{S}_U^{\perp}$.
If $U \in\mf{S}^*$, then both $\PF_U$ and $\PE_U$ are infinite diameter and so we can apply Lemma \ref{prop: Hung's construction in product} to build the desired quasi-geodesic in $\P_U$ based on $\gate_{\P_U}(Y)$. 
\begin{prop} \label{prop: Hung's construction in product}

Let $Y \subseteq \mc{X}$. There exist constants $L'$, $r_0$ and  functions $f,g,h:[r_0,\infty) \rightarrow [0,\infty)$, all depending only on $\mf{S}$,  
such that $f(r),g(r),h(r) \rightarrow \infty$ as $r \rightarrow \infty$  and the following holds: for each \(U \in \mf{S}^*\) and each $r\geq r_0$, if the $r$--neighborhood of $\phi_U^{-1} \left( \gate_{\P_U} (Y)\right)$ does not cover $\PF_U \times \PE_U$ and $\diam(\pi_U(Y)) > f(r)$, then there exists a $(L',L')$--quasi-geodesic $\eta$ with endpoints  $a,b \in \phi_U^{-1} \left( \gate_{\P_U} (Y)\right)$ such that $\eta$ is not contained in the $g(r)$--neighborhood of $\phi_U^{-1} \left( \gate_{\P_U} (Y)\right)$ and $d_{U}(a,b) >h(r)$. 
\end{prop}

\begin{proof} Our approach is to construct a spiral path of sufficient slope in $\PF_U \times \PE_U$ and then apply Lemma \ref{lem: spiral paths are quasi-geodesics} to conclude it is a quasi-geodesic.
Let $d(\cdot,\cdot)$ denote the $\ell_1$--distance in $\PF_U \times \PE_U$ and fix the following constants that depend only on $\mf{S}$:
\begin{itemize}
	\item $L$ such that \(\PF_U\) and $\PE_U$ are $(L,L)$--quasi geodesic spaces.
	\item $K$  such that $\pi_U$ is $(K,K)$--coarsely Lipschitz.
	\item $N =4L^2+1$ will be the slope of the spiral path we construct.
\end{itemize}
Let $r>10L^3+6$ and \(A= \phi_U^{-1} \left( \gate_{\P_U} (Y)\right)\). Suppose that the $r$--neighborhood of \(A\) does not cover $\PF_U \times \PE_U$. Thus there exists a point $z = (x_1,y_1) \in \PF_U \times \PE_U$ such that \(r \leq d(z, A) \leq r+2L\). Let \(a = (x_2,y_2)\) be a point of $A$ such that $d(z,a)-1 \leq d(z,A) $. 
We have \(\min\{d_{\PF_U}(x_1, x_2), d_{\PE_U}(y_1,y_2)\} \leq \frac{r+2L+1}{2}\). There are two cases depending on which of the two factors realizes the minimum.

\myparagraph{If \(d_{\PF_U}(x_1, x_2)\) realizes the minimum.}
In this case let \(z' = (x_2,y_1)\) and $D_r = \frac{r-2L-1}{2}$. Then \(d(z', A)\geq d(z, A) - d(z, z') \geq D_r\), which implies $d(z',a) > 3L^3+1$ because $r>10L^3+6$.

There exists \(B_r>r\) such that for any pair of points \(u, v\) of \(\PF_U\) if \(d_{U} (u,v) \geq B_r\), then  
\[d_{\PF_U}(u,v) \geq 2(r+2L+1)N .\] We shall assume \(\diam(\pi_U(Y)) > 2B_r\), so there is a point \(a' = (x_3,y_3)\) of $A$ such that \(d_{U}(x_2, x_3) \geq B_r\) and $d_{\PF_U}(x_2,x_3) > d_{\PE_U}(y_2,y_1) N$. We can now form a spiral path $\eta$ of slope $N=4L^2+1$ by connecting each sequential pair of points in the sequence 
\[a = (x_2,y_2) - (x_2,y_1) - (x_3,y_1) - (x_3,y_3) = a'\]
by $(L,L)$--quasi-geodesics. Since $d_{\PE_U}(y_2,y_1) > 3L^3+1$ , $\eta$ satisfies the hypothesis of Lemma \ref{lem: spiral paths are quasi-geodesics} and is therefore an $(L',L')$--quasi-geodesic for some $L'$ determined by $L$.

\begin{figure}[H]
	\centering
	
	\centering
	\def\svgscale{.7}
	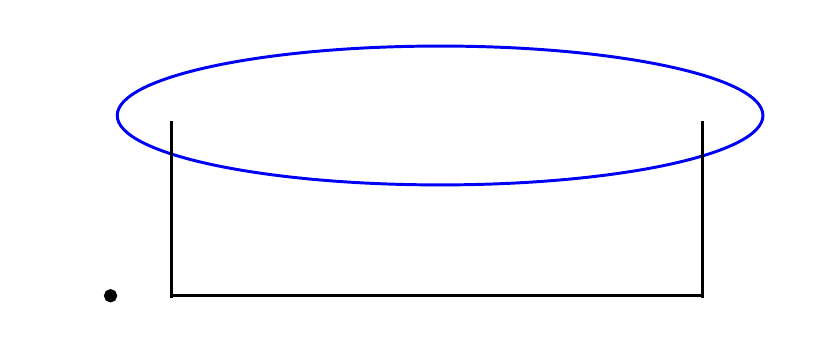
	\caption{Spiral path constructed when \(d_{\PF_U}(x_1, x_2)\leq \frac{r+2L+1}{2}\).}
	
\end{figure}

Since $z' = (x_2,y_1)$ is at least $D_r$ far from $A$, \(\eta\) has endpoints in \(A\) and is not contained in the $D_r$-neighborhood of $A$. Moreover, $d_{U}(a,a') \geq B_r$ and we get the claim with $f(r)=2B_r$, $g(r) = D_r$, and $h(r) = B_r$

\myparagraph{If \(d_{\PE_U}(y_1, y_2)\) realizes the minimum.}
Let \(z' = (x_1, y_2)\). As before we have that \(d(z', A) \geq D_r= \frac{r-2L-1}{2}\), which implies $d(z',a) > 3L^3+1$ .
Let \(y_3\) be a point of \(\PE_U\) such that \[(r+2L+1)N\leq d_{\PE_U}(y_2, y_3) \leq 2(r+2L+1)N.\]

\noindent There exists \(C_r>r\) such that for any pair of points \(u, v\) of \(\PF_U\) if \(d_{U} (u,v) \geq C_r\), then  
\[d_{\PF_U}(u,v) \geq 2(r+2L+1)N^2 .\] We shall assume \(\diam(\pi_U(Y)) > 2C_r\), so there exists \(a'= (x_4, y_4) \in A\) such that \(d_{U} (x_1, x_4) > C_r\). This implies $d_{\PF_U}(x_1,x_4) > 2(r+2L+1)N^2$ and we can now form a spiral path $\eta$ of slope $N=4L^2+1$ by connecting each sequential pair of points in the sequence 
\[a = (x_2,y_2) - (x_1,y_2) - (x_1,y_3) - (x_4,y_3) - (x_4,y_4) = a'\]
by an $(L,L)$--quasi-geodesics. 

\begin{figure}[H]
	\centering
	\def\svgscale{.7}
	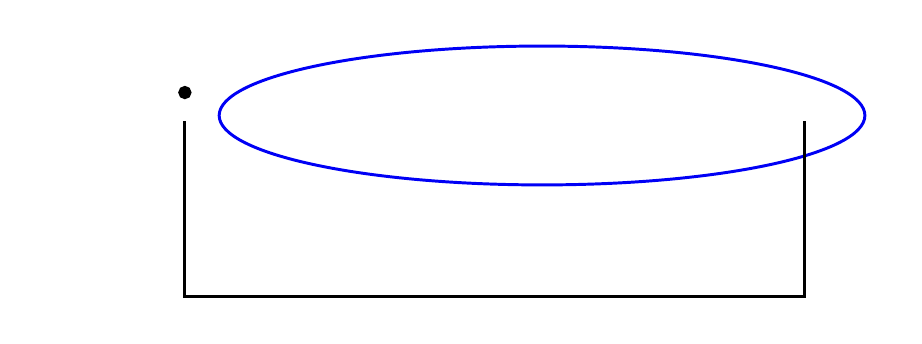
	\caption{Spiral path constructed when \(d_{\PE_U}(y_1, y_2)\leq \frac{r+2L+1}{2}\).}
\end{figure}

As before $\eta$ satisfies the hypothesis of Lemma \ref{lem: spiral paths are quasi-geodesics} and is therefore an $(L',L')$--quasi-geodes for some $L'$ determined by $L$.  The remaining claims follow as in the preceding case.
\end{proof}

The distance formula makes the map $\phi_U\colon \PF_U \times \PE_U \rightarrow \mc{X}$ a uniform quasi-isometric embedding.
Thus $\gate_{\P_U}(Y)$ coarsely covers \(\P_U\)  if and only if $\phi_U^{-1} \left( \gate_{\P_U} (Y)\right)$ coarsely covers $\PF_U \times \PE_U$, Proposition \ref{prop: Hung's construction in product} therefore allows us to immediately deduce the following result for $\P_U \subseteq \mc{X}$.

\begin{prop}\label{prop: Hung's construction}
Let $Y \subseteq \mc{X}$. There exist constants $L'$, $r_0$ and  functions $f,g,h:[r_0,\infty) \rightarrow [0,\infty)$, all depending only on $\mf{S}$,  
such that $f(r),g(r),h(r) \rightarrow \infty$ as $r \rightarrow \infty$  and the following holds: For each  \(U \in \mf{S}^*\)and each $r \geq r_0$, if the $r$--neighborhood of $ \gate_{\P_U}(Y)$ does not cover $\P_U$ and $\diam(\pi_U(Y)) > f(r)$, then there exists an $(L',L')$--quasi-geodesic $\eta$ with endpoints  $a,b \in \gate_{\P_U}(Y)$  such that:
\begin{enumerate}
	\item $ \eta \subseteq \P_U$,
	\item $\eta$ is not contained in the $g(r)$--neighborhood of $\gate_{\P_U}(Y)$,
	\item $d_{U}(a,b) > h(r)$.
\end{enumerate}

\end{prop}

Proposition \ref{prop: Hung's construction} furnishes a quasi-geodesic $\eta$ with endpoints on $\gate_{\P_U}(Y)$ that can be made as far from $\gate_{\P_U}(Y)$ as desired by increasing $\diam(\pi_U(Y))$. However, to exploit the fact that $Y$ is a {\sqc} subset, we need  the next lemma, which ``promotes" $\eta$ to a quasi-geodesic with endpoints on $Y$.

\begin{lem}\label{lem: substituting quasi-geodesic}
There exists $D>0$ such that if $x,y\in \mc{X}$ and $U\in \mf{S}$, with $d_U(x,y)>D$ and $\eta$ is a $(k,c)$--quasi-geodesic contained in $\P_U$ with endpoints $\gate_{\P_U}(x)$ and $\gate_{\P_U}(y)$, then there exists a $(k',c')$--quasi-geodesic containing $\eta$ and with endpoints $x$ and $y$, where $k'$ and $c'$ depend only on $\lambda$ and $\epsilon$.
\end{lem}
\begin{proof}
Let $D$ and $\lambda$ be as in Proposition \ref{prop: active subpath}. We further assume $\lambda$ is large enough that every pair of points in $ \mc{X}$ can be joined by an $\lambda$--hierarchy path (Theorem \ref{thm:monotone_hierarchy_paths}). 

Assume $d_U(x,y) > D$ and let $\tilde{\gamma}$ be the $\lambda$--hierarchy path connecting $x$ and $y$ provided by Proposition \ref{prop: active subpath}. Let $\alpha$ be the active subpath of $\tilde{\gamma}$ corresponding to $U$. Define $x'$ (resp. $y'$) be the endpoint of $\alpha$ closest to $x$ (resp. $y$) and $x'' = \mf{g}_{\P_U}(x)$ (resp. $y''=\mf{g}_{\P_U}(y)$). If $\eta \colon [b,c] \rightarrow \P_U$ is any $(k,c)$--quasi-geodesic in $\P_U$ connecting $x''$ and $y''$, let $\gamma$ be the concatenation of $\tilde{\gamma} - \alpha$, any  $\lambda$--hierarchy path from $x'$ to $x''$, $\eta$, and any $\lambda$--hierarchy path from $y'$ to $y''$. We will show that this path $\gamma$ is a $(k',c')$--quasi-geodesic where the constants depend only on $k$ and $c$.

The distances $d_\mc{X}(x', \P_U)$ and $d_\mc{X}(y',\P_U)$ are uniformly bounded by Proposition \ref{prop: active subpath}. By Lemma \ref{lem: closest_point_proj}, the distances $d_\mc{X}(x',\gate_{\P_U}(x'))$ and $d_\mc{X}(y',\gate_{\P_U}(y'))$  are uniformly bounded as well. Again by Proposition \ref{prop: active subpath}, $\gate_{\P_U}(x)$   coarsely coincides with $\gate_{\P_U}(x')$ and $\gate_{\P_U}(y)$   coarsely coincides with $\gate_{\P_U}(y')$. Thus there exists $\mu$ depending only on $\mf{S}$ such that $d_\mc{X}(x',x''), d_\mc{X}(y',y'') \leq \mu$.

Now, let $\gamma_x$ (resp. $\gamma_y$) be the subset of $\gamma$ from $x$ to $x''$ (resp. $y$ to $y''$). Since $d_\mc{X}(x',x'')$ and $d_\mc{X}(y',y'')$ are uniformly bounded by $\mu$,  $\gamma_x$ and $\gamma_y$ are both uniform quasi-geodesics.  By Lemma \ref{lem: closest_point_proj} and Proposition \ref{prop: active subpath}, there exists $K\geq 1$ depending on $k$, $c$, and $\mf{S}$ such that the following hold:
\begin{itemize}
	\item $d_\mc{X}(x',x'')$, $d_\mc{X}(y',y'') \leq K$.
	\item $\diam(\gate_{\P_U}(\gamma_x)),\diam(\gate_{\P_U}(\gamma_y)) \leq K$.
	\item $\gamma_x, \gamma_y, \eta$ are all $(K,K)$--quasi-geodesics.
	\item For all $p \in \P_U$ and $q \in \mc{X}$, $d_\mc{X}(q,\gate_{\P_U}(q)) \leq K d_\mc{X}(p,q) +K $.
\end{itemize}

Let $\gamma = \gamma_x * \eta * \gamma_y \colon [a,d] \rightarrow \mc{X}$ and $a<b<c<d$ such that $\gamma|_{[a,b]} = \gamma_x$, $\gamma|_{[b,c]} = \eta$ and $\gamma|_{[c,d]} = \gamma_y$. For $t,s \in [a,d]$, let $u=\gamma(t)$, $v= \gamma(s)$. We want to show $|t-s| \asymp d_\mc{X}(u,v)$ for some constants depending only on $K$. The only interesting cases are when $u$ and $v$ are in different components of $\gamma = \gamma_x * \eta * \gamma_y$, so without loss of generality, we have the following two cases.
\begin{enumerate}[{Case} 1:]
	\item Assume $t \in [a,b]$ and $s \in [b,c]$. Thus $u \in \gamma_x$ and $v \in \eta$ and we have:
	\begin{align*}
		d_\mc{X}(u,v) \leq& d_\mc{X}(u,x'') + d_\mc{X}(x'',v) \\
		\leq& K|t-b|+K|b-s| +2K \\
		\leq& K|t-s| +2K
	\end{align*}

	For the inequality $|t-s| \preceq d_{\mc{X}}(u,v)$, our choice of $K$ provides \[d_\mc{X}(u,x'')\leq d_\mc{X}(u,\gate_{\P_U}(u))+K \leq Kd_\mc{X}(u,v)+2K.\] By the triangle inequality $d_\mc{X}(v,x'') \leq d_\mc{X}(v,u)+d_\mc{X}(u,x'')$ and we derive the desired inequality as follows:
	\begin{align*}
		|t-s| =& |t-b|+|b-s|\\
		\leq& K d_\mc{X}(u,x'') +K d_\mc{X}(v,x'') +2K\\
		\leq& K^2 d_\mc{X}(u,v) +K\left(d_\mc{X}(u,v) + d_\mc{X}(u,x'')\right) +2K^2+2K\\
		\leq&  K^2 d_\mc{X}(u,v) +K d_\mc{X}(u,v) + K^2 d_\mc{X}(u,v) +4K^2+2K\\
		\leq& 3K^2 d_\mc{X}(u,v) + 6K^2
	\end{align*}

	\item Assume $t\in [a,b]$ and $s\in [c,d]$ so that $u\in \gamma_x$ and $v\in\gamma_y$. Further we can assume $u,v \in \tilde{\gamma}$, since otherwise the above proof holds by increasing the constants by $4K$. The inequality $d_\mc{X}(u,v) \preceq |t-s|$ can be established by a nearly identical argument to the previous case. For the inequality $|t-s| \preceq d_{\mc{X}}(u,v)$ we need to utilize the fact that $\tilde{\gamma}$ is a $(\lambda_0,\lambda_0)$--quasi-geodesic. Thus, by increasing $K$, we can ensure that
	\begin{itemize}
		\item $d_\mc{X}(u,v) \stackrel{K,K}{\asymp} d_\mc{X}(u,x') + d_\mc{X}(x',y') + d_\mc{X}(y',v)$,
		\item $d_\mc{X}(x',y') \stackrel{1,2K}{\asymp} d_\mc{X}(x'',y'') \stackrel{K,K}{\asymp} |b-c|$,
		\item $d_\mc{X}(u,x') \stackrel{1,K}{\asymp} d_\mc{X}(u,x'') \stackrel{K,K}{\asymp} |t-b|$,
		\item $d_\mc{X}(v,y') \stackrel{1,K}{\asymp} d_\mc{X}(v,y'') \stackrel{K,K}{\asymp} |c-s|$.
	\end{itemize}
	We then have the following calculation
	\begin{align*}
		|t-s| =& |t-b| + |b-c| + |c-s|\\
		\leq& K d_\mc{X}(u,x'') + K d_\mc{X}(x'',y'') + K d_\mc{X}(y'',v) +3K \\
		\leq& K d_\mc{X} (u,x') + K d_\mc{X}(x',y') + K d_\mc{X}(y',v) +7K^2\\
		\leq& K^2 d_\mc{X}(u,v) +8K^2.
	\end{align*}
\end{enumerate}
\end{proof}

We can now provide the proof of Proposition \ref{prop:ortho_proj_for_product_regions}.

\begin{proof}[Proof of Proposition \ref{prop:ortho_proj_for_product_regions}]
Let $Y\subseteq \mc{X}$ be $Q$--{\sqc} and $U \in \mf{S}$ such that $\diam(CU) = \infty$ and there exists $V \in \mf{S}_U^{\perp}$ with $\diam(CV) = \infty$. Recall our goal is to show that there exists $B$ depending on $\mf{S}$ and $Q$ such that if $\diam(\pi_U(Y)) >B$, then $\P_U \subseteq N_B(\gate_{\P_U}(Y))$.
Begin by fixing the following constants that all depend only on $\mf{S}$ and $Q$:
\begin{itemize}
	\item $\mu$ such that for all $x \in \mc{X}$, $d_U(x,\gate_{\P_U}(x)) < \mu$
	\item $D$, the constant from Proposition \ref{lem: substituting quasi-geodesic} 
	\item $L'$, the quasi-geodesic constant from Proposition \ref{prop: Hung's construction}
	\item  $k'$, the quasi-geodesic constant  obtained by applying Lemma \ref{lem: substituting quasi-geodesic} to a $(L',L')$--quasi-geodesic
	\item $K$, the constant from the bridge theorem (Theorem \ref{thm:parallel_gates}) for $Y$ and $\P_U$ (recall $Y$ is hierarchically quasiconvex by Proposition \ref{prop: quasiconvex implies HQC})
	
\end{itemize}

Let $f,g,h$ be as in Proposition \ref{prop: Hung's construction} and fix $r$ be large enough that \[g(r) > 2KQ(k',k')+K^2 +K  \text{ and } h(r)>D+ 2\mu.\] If $\P_U \subseteq N_r(\gate_{\P_U}(Y))$, then we are done. So for the purposes of contradiction, suppose that $\P_U \not\subseteq N_r(\gate_{\P_U}(Y))$ and that $\diam(\pi_U(Y)) > f(r)$.  Let $\eta$ be the $(L',L')$--quasi-geodesic provided by Proposition \ref{prop: Hung's construction} and let $a_1,b_1 \in \gate_{\P_U}(Y)$ be the endpoints of $\eta$.  Let $a_0,b_0 \in Y$ such that $\gate_{\P_U}(a_0) = a_1$ and $\gate_{\P_U}(b_0) = b_1$. Since \[d_U(a_0,b_0) > d_U(a_1,b_1) -2\mu>h(r) -2\mu >D,\] Lemma \ref{lem: substituting quasi-geodesic} produces a $(k',k')$--quasi-geodesic $\gamma$ with endpoints $a_0$ and $b_0$ and containing $\eta$ where $k'$ depending ultimately only on $\mf{S}$.  Since $Y$ is $Q$--{\sqc}, $\gamma \subseteq N_{Q(k',k')}(Y)$.  By Proposition \ref{prop: Hung's construction}, there exists $x \in \eta$ such that $d_\mc{X}(x,\gate_{\P_U}(Y)) > g(r)$. Let $y \in Y$ be such that $d_\mc{X}(x,y)-1 \leq d_\mc{X}(x,Y)$, then by the bridge theorem (Theorem \ref{thm:parallel_gates}) we have the following contradiction:
\[Q(k',k') \geq d_\mc{X}(x,y)-1 \geq \frac{1}{K}d_\mc{X}(x,\gate_{\P_U}(Y)) -K - 1 >  2Q(k',k').\]
\end{proof}

The following proposition uses Proposition \ref{prop:ortho_proj_for_product_regions} to finish the proof of the implication from (\ref{item:quasiconvex}) to (\ref{item:projections}) in Theorem~\ref{thm:classification_of_quasiconvex_subset}.

\begin{prop}\label{prop: quasiconvex => (1)+(2)}
If $(\mc{X},\mf{S})$ is an HHS with the bounded domain dichotomy and $Y$ is a $Q$--{\sqc} subset of $\mc{X}$, then there exists $B>0$ depending only on $Q$ and $\mf{S}$ such that $Y$ has the $B$--orthogonal projection dichotomy.
\end{prop}

\begin{proof}
Let $Y \subseteq \mc{X}$ be $Q$--{\sqc} and $B'>0$ be larger than the bounded domain dichotomy constant for $\mf{S}$ and the constant $B_0$ from Proposition \ref{prop:ortho_proj_for_product_regions}. Let $U \in \mf{S}$. If $U \not \in \mf{S}^*$, then by the bounded domain dichotomy, either $\diam(CU) < B'$ or for all $V \in \mf{S}^\perp_U$, $\diam(CV) <B'$.  In either case, the $B'$--orthogonal projection dichotomy is satisfied for $U$. Thus we can assume that $U \in\mf{S}^*$, so $\diam(CU) = \infty$ and there exists $V \in\mf{S}_U^{\perp}$ with $\diam(CV) = \infty$. Suppose $\diam(\pi_U(Y))>B'$. By Proposition \ref{prop:ortho_proj_for_product_regions}, $\P_U \subseteq N_{B'}(\gate_{\P_U}(Y))$. For all $V \in \mf{S}_U^\perp$, $\pi_V(\P_U)$ uniformly coarsely covers $CV$, thus there exists $B\geq B'$ depending only on $Q$ and $\mf{S}$ such that $CV \subseteq N_B(\pi_V(Y))$.
\end{proof}

\subsection{Contracting subsets in HHSs}\label{subsec: Contracting Subsets in HHSs}

We now finish the proof of Theorem \ref{thm:classification_of_quasiconvex_subset} by showing that for hierarchically quasiconvex subsets, the orthogonal projection dichotomy implies that the gate map $\gate_{Y}$ is contracting.

\begin{prop}
\label{prop: HQC + Dichotomy implies contracting}
Let $(\mc{X},\mathfrak S)$ be a hierarchically hyperbolic space with the bounded domain dichotomy and $Y \subseteq \mc{X}$ be  $k$--hierarchically quasiconvex. If $Y$ has the $B$--orthogonal projection dichotomy, then the gate map $\gate_Y:\mc{X} \rightarrow Y$ is $(A,D)$--contracting, where $A$ and $D$ depend only on $k$, $B$, and $\mf{S}$.
\end{prop}

\begin{proof}
The gate map satisfies the first two condition in the definition of a contracting map by Lemma \ref{lem:gate}. It only remains to prove the following:
\begin{center} 
	\textit{There exist some $0<A < 1$ and $D\geq 1$ depending only on $k$, $B$, and $\mf{S}$, such that for all $x \in\mc{X}$, $\diam(\gate_{Y}(B_R(x))\leq D$ where $R = Ad(x,Y)$.}
\end{center}
Fix a point $x_0 \in \mc{X}$ with $d_\mc{X}(x_0,Y)\geq C_0$ and let $x \in \mc{X}$ be any point with $d_\mc{X}(x_0,x)<C_1 d_\mc{X}(x_0,Y)$ for constants $C_0$ and $C_1$ to be determined below. We will prove that for each domain $U\in \mathfrak S$ the distance $d_U\bigl(\gate_Y(x_0),\gate_Y(x)\bigr)$ is uniformly bounded, then the above will follow from the distance formula (Theorem \ref{thm:distance_formula}).

We choose a ``large'' number $L$ (we will clarify how large $L$ is later). Let $K\geq 1$ be the coarse equality constant from the distance formula 
with thresholds $L$ and $2L$. Take $C_0>(2K+1)K$ sufficiently large so there is $W\in \mathfrak S$ such that $d_W\bigl(x_0,\gate_Y(x_0)\bigr)>2L$. Choose $C_1<1/(2K^2+1)$, ensuring that $d_\mc{X}\bigl(x_0,\gate_Y(x_0)\bigr)>(2K^2+1)d_\mc{X}(x_0,x)$. If $d_\mc{X}(x_0,x)\leq C_0$, then by the coarse Lipschitzness of the projections $d_U\bigl(\gate_Y(x_0),\gate_Y(x)\bigr)$ is uniformly bounded by a number depending on $C_0$ for each $U \in\mf{S}$. Therefore, we can assume that $d_\mc{X}(x_0,x)>C_0$.
We claim that there is $V\in \mathfrak S$ such that $d_V\bigl(x_0,\gate_Y(x_0)\bigr)>d_V(x_0,x)+L$. 

Assume for the purposes of contradiction that $d_W\bigl(x_0,\gate_Y(x_0)\bigr)\leq d_W(x_0,x)+L$ for all $W\in \mathfrak S$. Therefore, we have $d_W\bigl(x_0,\gate_Y(x_0)\bigr)\geq 2L \implies d_W(x_0,x)\geq L$ and this implies \[\ignore{d_W\bigl(x_0,\gate_Y(x_0)\bigr)}{2L}\leq 2\ignore{d_W\bigl(x_0,x\bigr)}{L}\] for all $W\in \mathfrak S$. Thus,  
\begin{align*}
	d_\mc{X}\bigl(x_0,\gate_Y(x_0)\bigr)
	&\leq K \sum\limits_{W \in \mf{S}} \ignore{d_W\bigl(x_0,\gate_Y(x_0)\bigr)}{2L}+K\\
	&\leq 2K \sum\limits_{W \in \mf{S}} \ignore{d_W\bigl(x_0,x\bigr)}{L}+K\\
	&\leq 2K\bigl(Kd_\mc{X}(x_0,x)+K\bigr)+K\\
	&\leq 2K^2 d_\mc{X}(x_0,x)+(2K+1)K\\
	&\leq 2K^2 d_\mc{X}(x_0,x)+C_0\\
	&\leq (2K^2+1)d_\mc{X}(x_0,x)
\end{align*}
which contradicts $C_1<1/(2K^2+1)$. Therefore, we can fix $V\in \mathfrak S$ such that \[d_V\bigl(x_0,\gate_Y(x_0)\bigr)>d_V(x_0,x)+L.\]
The construction of the gate map and the hyperbolicity of $CV$ ensure that, after enlarging $L$ and shrinking $C_1$ if necessary, $d_V\bigl(\gate_Y(x_0),\gate_Y(x)\bigr)<r$ where $r$ depends only on $k$ and $\mf{S}$. The triangle inequality then gives us \[d_V\bigl(x,\gate_Y(x_0)\bigr)>L \text{ and } d_V\bigl(x,\gate_Y(x)\bigr)>L-r.\]

Now let $U \in \mf{S}$. If $\diam\bigl(\pi_U(Y)\bigr)\leq B$, then $d_U\bigl(\gate_Y(x_0),\gate_Y(x)\bigr)\leq B$ and we are done. Thus we can assume that $\diam\bigl(\pi_U(Y)\bigr)> B$.
If $U=V$, then the distance $d_U\bigl(\gate_Y(x_0),\gate_Y(x)\bigr)$ is uniformly bounded above by the number $r$ and we are done.  We now consider the other possible cases depending on the relation between $U$ and $V$:

\textbf{Case 1:}  $V\nest U$. If we choose $L$ greater than $E+r$, then $$d_V\bigl(x_0,\gate_Y(x_0)\bigr)>E \text{ and } d_V\bigl(x,\gate_Y(x)\bigr)>E.$$ Thus by the bounded geodesic image axiom (Axiom \ref{item:dfs:bounded_geodesic_image}), the $CU$ geodesics from $\pi_U(x_0)$ to $\pi_U(\gate_Y(x_0))$ and from $\pi_U(x)$ to $\pi_U(\gate_Y(x))$  must intersect $N_E(\rho_U^V)$. Therefore, the distance $d_U\bigl(\gate_Y(x_0),\gate_Y(x)\bigr)$ is uniformly bounded due to the hyperbolicity of $CU$ and the properties of the gate map (Lemma \ref{lem:gate}). 

\textbf{Case 2:}  $U\nest V$. If some $CV$ geodesic from $\pi_V(\gate_Y(x_0))$ to $\pi_V(\gate_Y(x))$ stays $E$--far from $\rho_V^U$, then by the bounded geodesic image axiom (Axiom \ref{item:dfs:bounded_geodesic_image}),  $d_U\bigl(\gate_Y(x_0),\gate_Y(x)\bigr) \leq E$  and we are done. Therefore, we assume that all $CV$ geodesics from $\pi_V(\gate_Y(x_0))$ to $\pi_V(\gate_Y(x))$ intersect ${N}_E(\rho_V^U)$. Since $d_V\bigl(x_0,\gate_Y(x_0)\bigr)>d_V(x_0,x)+L$, if there was also a $CV$ geodesic from $\pi_V(x_0)$ to $\pi_V(x)$ that intersected $N_E(\rho_V^U)$ we would have
\begin{align*}
	d_V\bigl(\gate_Y(x_0), \rho_V^U\bigr)&\geq d_V\bigl(\gate_Y(x_0),x_0\bigr)-d_V(x_0,\rho_V^U)\\
	&> d_V\bigl(\gate_Y(x_0),x_0\bigr)-d_V(x_0,x)-2E\\
	&\geq L-2E.
\end{align*}
However, $d_V\bigl(\gate_Y(x_0),\gate_Y(x)\bigr) \leq r$ which implies $\pi_V(\gate_Y(x_0))$ lies in $N_{E+r}(\rho_V^U)$. Therefore, by assuming $L > 4E+r$ we can ensure that no $CV$ geodesic from $\pi_V(x_0)$ to $\pi_V(x)$ intersects $N_E(\rho_V^U)$.  Thus $d_U(x_0,x) < E$ by the bounded geodesic image axiom and it immediately follows that $d_U(\gate_{Y}(x_0),\gate_{Y}(x))$ is bounded by a constant depending on $k$ and $\mf{S}$.

\textbf{Case 3:}  $U\not\nest V$ and $V \not\nest U$. Recall that we can assume $\diam(\pi_U(Y)) >B$. Thus if $U\perp V$, we have $CV \subseteq N_B(\pi_V(Y))$ by the orthogonal projection dichotomy. However $d_V\bigl(x_0,\gate_Y(x_0)\bigr)>L$, so by Lemma \ref{lem: closest_point_proj} we can choose $L$ large enough so that $\pi_V(x_0)$ does not lie in the $B$--neighborhood of $\pi_V(Y)$. Thus $U$ and $V$ cannot be orthogonal and hence $U \trans V$. 

Now assume $L>2\kappa_0+3r+2E+1$. Then if $d_V\bigl(\gate_Y(x_0),\rho_V^U\bigr)\leq \kappa_0+r+E$ we have 

\begin{align*}d_V(x_0,\rho_V^U)\geq& d_V\bigl(x_0,\gate_Y(x_0)\bigr)-d_V\bigl(\gate_Y(x_0),\rho_V^U\bigr)-E\\
	\geq& L-(\kappa_0+r+E)-E\\
	>&\kappa_0\end{align*}
and
\begin{align*} 
	d_V(x,\rho_V^U)\geq& d_V\bigl(x,\gate_Y(x_0)\bigr)-d_V\bigl(\gate_Y(x_0),\rho_V^U\bigr)-E\\
	>&L-(\kappa_0+r+E)-E\\
	>&\kappa_0. 
\end{align*}
Therefore, $d_U(x_0, \rho_U^V)<\kappa_0$ and $d_U(x, \rho_U^V)<\kappa_0$ by consistency (Axiom \ref{item:dfs_transversal}). This implies that $d_U(x_0,x)\leq 2\kappa_0+E$ and thus $d_U(\gate_Y(x_0),\gate_Y(x))$ is bounded by a constant depending on $k$ and $\mf{S}$.

If instead $d_V\bigl(\gate_Y(x_0),\rho_V^U\bigr)>\kappa_0+r+E$, then  $d_V\bigl(\gate_Y(x),\rho_V^U\bigr)>\kappa_0$ since  $d_V\bigl(\gate_Y(x_0),\gate_Y(x)\bigr)<r$. By consistency $d_U\bigl(\gate_Y(x_0), \rho_U^V\bigr)<\kappa_0$ and $d_U\bigl(g(x), \rho_U^V\bigr)<\kappa_0$,  which implies that \[d_U\bigl(\gate_Y(x_0),\gate_Y(x)\bigr)\leq2\kappa_0+E.\]\end{proof}

\begin{rem}\label{rem: we need all the hypotheses}
Both hypotheses on the subspace in Proposition \ref{prop: HQC + Dichotomy implies contracting} are in fact required. In the standard HHG structure of $\mathbb{Z}^2$, the subgroup $\langle (1,0) \rangle$ is hierarchically quasiconvex, but does not satisfy the orthogonal projection dichotomy. On the other hand, the subgroup $\langle (1,1) \rangle$ has the orthogonal projection dichotomy, but is not hierarchically quasiconvex. Neither of these subsets are  {\sqc} and thus neither are contracting. Both of the above examples can even be made to be (non-strongly) quasiconvex by choosing \{(1,0), (1,1), (0,1)\} to be the generating set for $\mathbb{Z}^2$.
\end{rem}

\subsection{A generalization of the bounded geodesic image property}\label{subsec:bounded_geodesic_image}

As a first application of our characterization of {\sqc} subsets (Theorem \ref{thm:classification_of_quasiconvex_subset}) we show that {\sqc} subspaces of HHSs satisfy a version of the bounded geodesic image property. First recall the bounded geodesic image property for {quasiconvex} subsets of hyperbolic spaces (not to be confused with the bounded geodesic image axiom of an HHS).

\begin{prop}[Bounded geodesic image property for hyperbolic spaces]
Let $Y$ be a $K$--quasiconvex subset of a geodesic \(\delta\)--hyperbolic space $X$. Then there exists $r>0$ (depending on \(\delta\) and $K$) such that if $d(\cpproj_Y(x),\cpproj_Y(y))>r$, then every geodesic connecting $x$ and $y$ must intersect the $r$--neighborhood of $Y$. 
\end{prop}

In the case of {\sqc} subsets of hierarchically hyperbolic space, we replace the closest point projection with the gate map and geodesics with hierarchy paths. Theorem \ref{intro:contracting_features} from the introduction will follow as a result of the following proposition, which is a version of the active subpath theorem (Proposition \ref{prop: active subpath}) for {\sqc} subsets.

\begin{prop}\label{prop: active subpaths for quasiconvex}
Let $(\mc{X},\mf{S})$ be an HHS with the bounded domain dichotomy and  $Y\subseteq \mc{X}$ be a $Q$--{\sqc}. For all $\lambda\geq 1$, there exist constants \(\nu, D\), depending on $\lambda$ and $Q$, so that the following holds for all \(x, y\in \X\). If \(d_\X(\gate_Y (x), \gate_Y(y))>D\) and  $\gamma:[a,b]\to \mc{X}$  is a \(\lambda\)--hierarchy path joining \(x\) and \(y\), then there is a subpath \(\alpha=\gamma_{|[a_1,b_1]}\) of \(\gamma\) with the following properties:
\begin{enumerate}
	\item \(\alpha \subseteq N_\nu(Y)\).
	\item The diameters of \(\gate_Y\bigl(\gamma([a,a_1])\bigr)\) and \(\gate_Y\bigl(\gamma([b_1,b])\bigr)\) are both bounded by $\nu$.
\end{enumerate}
\end{prop}

\begin{proof}
By Theorem \ref{thm:classification_of_quasiconvex_subset}, $Y$ is hierarchically quasiconvex and has the orthogonal domain dichotomy. In particular, $\pi_U(Y)$ is uniformly quasiconvex in $CU$ for all $U\in \mf{S}$. Let $x,y \in \mc{X}$ and $\gamma$ be a $\lambda$--hierarchy path connecting $x$ and $y$. Since $\gamma$ is a $(\lambda,\lambda)$--quasi-geodesic we can choose \[x=x_0,x_1,x_2,\cdots,x_n=y\] on $\gamma$ such that the distances between $x_i$ and $x_{i+1}$ are all bounded by $2\lambda$. We will show that there exist $0\leq i_0 \leq j_0 \leq n$ such that:
\begin{itemize}
	\item For $i=i_0$ or $i=j_0$,  $d_\mc{X}(x_i,\gate_Y(x_i))$ is bounded by a constant depending only on $Q$, $\lambda$, and $\mf{S}$.
	\item  If $s<t<i_0$ or $j_0<s<t$, then $d_\mc{X}(\gate_Y(x_s),\gate_Y(x_t))$ is bounded by a constant depending only on $Q$, $\lambda$, and $\mf{S}$.
\end{itemize}

\noindent Since $Y$ is {\sqc}, once we have shown the above, the proposition will follow with $\alpha$ as the subsegment of $\gamma$ between $x_{i_0}$ and $x_{j_0}$.

For each $U\in \mf{S}$, the projection $\pi_U$  is uniformly coarsely Lipschitz, thus there is $\lambda'$ depending on $(\mc{X},\mf{S})$ and $\lambda$ such that the distances $d_U(x_i,x_{i+1})$ are all bounded above by $\lambda'$. 

By the hyperbolicity of each $CU$ and the properties of gate map (Lemma \ref{lem:gate}), there are constants $B$ and $\mu$ depending only on $\mf{S}$, $Q$, and $\lambda$ such that for each $V\in \mf{S}$ satisfying $d_V(\gate_Y(x),\gate_Y(y))>B$ there are $0\leq I_V<J_V\leq n$ with the following properties: 

\begin{enumerate}
	\item $d_{V}\bigl(x_i,\gate_Y(x_i)\bigr) \leq \mu$ for $I_V\leq i \leq J_V$ 
	\item If $s<t<I_V$ or $J_V<s<t$, then $d_{V}(\gate_Y(x_s),\gate_Y(x_t)) < \mu$.
	\item $d_V(x_{I_V},x_{J_V})\geq 10D$ where $D= 3(E+\mu+\kappa_0+\lambda')$
\end{enumerate}
For future convenience, we can and shall assume $B$ is large enough that, $B>E$, $(\mc{X},\mf{S})$ has the $B$--bounded domain dichotomy, and $Y$ has the $B$--orthogonal projection dichotomy. By the uniqueness axiom (Axiom \ref{item:dfs_uniqueness}), there is a constant $K$ depending on $B$ and $(\mc{X},\mf{S})$ such that if $d_{\mc{X}}(\gate_Y(x),\gate_Y(y))>K$, then the set $\mathcal{R}= \relevant_B(\gate_Y(x),\gate_Y(y))$ is non-empty.   Since for each $V\in\mathcal{R}$ we have $d_V(x_{I_V},x_{J_V})\geq 10D$ and each distance $d_V(x_i,x_{i+1})$ is bounded above by $\lambda' <D$, there are $I_{V}<i_{V}<j_{V}<J_{V}$ such that 
\begin{equation*}\label{eq:active_subpaths_middle_piece}
	D\leq d_V(x_{i_V}, x_{I_V})\leq 2D \text{ and } D\leq d_V(x_{j_V}, x_{J_V})\leq 2D.\tag{$*$}
\end{equation*} 
Let $i_0=\min\limits_{V\in\mc{R}}i_V$ and $j_0=\max\limits_{V\in\mathcal{R}}j_V$.\\

We first prove that for each $s$, $t$ that are both less than $i_0$ or both greater than $j_0$ the distance $d_{\mc{X}}(\gate_Y(x_s),\gate_Y(x_t))$ is uniformly bounded by some constant depending only on $\mf{S}$, $Q$, and $\lambda$. We will provide the proof for the case $s$, $t$ are both less than $i_0$ and the proof for the other case is essentially identical. Let $V\in\mf{S}$. If $V \not \in \mc{R}$, then $d_V(\gate_Y(x),\gate_Y(y))\leq B$ which implies $\diam\bigl(\pi_V(\gate_Y(\gamma))\bigr)$ is bounded by a constant that depends only on $B$, $\lambda$, $Q$ and $\mf{S}$. In particular, $d_V(\gate_Y(x_s),\gate_Y(x_t))$ is also uniformly bounded by this constant. When $V \in \mc{R}$, then $s$ and $t$ are both less than $i_V$. Therefore by item (2) above and (\ref{eq:active_subpaths_middle_piece}) we have that $d_V(\gate_Y(x_s),\gate_Y(x_t))$ is bounded by a constant depending only on $\mf{S}$, $Q$, and $\lambda$. By the distance formula (Theorem \ref{thm:distance_formula}) the distance $d_{\mc{X}}(\gate_Y(x_s),\gate_Y(x_t))$ is therefore bounded by a constant that ultimately depends only on  $\mf{S}$, $Q$, and $\lambda$.

We now prove that there exists $\nu'$ depending on $\mf{S}$, $Q$, and $\lambda$ such that for $i=i_0$ or $i=j_0$ \begin{equation*} \label{eq:active_subpaths_conclusion_1}d_{\mc{X}}(x_i,\gate_Y(x_i) )\leq \nu'. \tag{$**$}\end{equation*} Again we only give the proof for the case of $i=i_0$ and the argument for the case $i=j_0$ is almost identical. By the distance formula, it is sufficient to check that  we can uniformly bound \(d_{U}(x_{i}, \gate_Y(x_{i}))\) for each \(U \in \mf{S}\).

Fix a domain $V\in \mathcal{R}$ such that $i=i_0=i_V$. We shall show \(d_{U}(x_{i}, \gate_Y(x_{i}))\) for all $U \in \mf{S}$ by examining the four cases for how $U$ can be related to $V$. 

\textbf{Case 1:} \(V \bot U\). Since $Y$ has the $B$--orthogonal domain dichotomy, \[V \in \mc{R} \implies CU \subseteq N_B(\pi_U(Y)).\] Therefore by the properties of the gate map (Lemma \ref{lem:gate}), we have that $d_\mc{X}(x_i,\gate_Y(x_i))$ is uniformly bounded. 

\textbf{Case 2:} \(V \trans U\). Suppose \(d_V(x_{i},\rho^U_V) > {\kappa_0 + \mu+E} \), then \[d_V(\gate_Y(x_{i}), \rho_V^U)>{\kappa_0 }\] and by the consistency axiom (Axiom \ref{item:dfs_transversal}) and triangle inequality \[d_{U}(x_{i}, \gate_{Y}(x_{i}))\leq 2\kappa_0 + E.\]
Now assume that  \(d_V(x_{i},\rho^U_V) < {\kappa_0 + \mu+E} \). Since $D > \mu+E+\kappa_0$, $d_V(x_i,x_{I_V})\geq D$, and $d_V(x_i,x_{J_V})\geq D$, we have that $x_{I_V}$, $\gate_Y(x_{I_V})$, $x_{J_V}$, and $\gate_Y(x_{J_V})$ all project at least $\kappa_0$ far from $\rho_V^U$ in $CV$. Therefore, by the consistency axiom and triangle inequality \[d_{U}(x_{I_V}, \gate_{Y}(x_{I_V}))\leq 2\kappa_0 + E \text{ and } d_{U}(x_{J_V}, \gate_{Z}(x_{J_V}))\leq 2\kappa_0 + E.\] Thus, by the quasiconvexity of $\pi_U(Y)$ in $CU$ and the properties of the gate map, the distance \(d_{U}(x_{i}, \gate_Y(x_{i}))\) is bounded by a uniform constant determined by $\mf{S}$, $Q$ and $\lambda$.

\textbf{Case 3:} \(U \sqsubseteq V\). Consider geodesics in \(CV\) connecting the projections of the pairs of points \((x_{I_V}, \gate_Y(x_{I_V}))\), \((x_{i}, \gate_Y(x_{i}))\) and \((x_{J_V}, \gate_Y(x_{J_V}))\). By the assumptions on \(I_V\), \( i \) and \(J_V\),  at most one of these geodesics intersects \(N_E(\rho^U_V)\). If such a geodesic is not the one connecting \(\pi_V(x_{i})\) and \(\pi_V(\gate_Y(x_{i}))\), then we are done by bounded geodesic image axiom (Axiom \ref{item:dfs:bounded_geodesic_image}). Otherwise, the bounded geodesic image axioms requires that \(\pi_V(x_{I_V})\) and \(\pi_V(x_{J_V})\) are contained in the $3E$--neighborhood of \(\pi_U(Y)\) in \(CU\). By the quasiconvexity of $\pi_U(Y)$ in $CU$ and the properties of the gate map, the distance \(d_{U}(x_{i}, \gate_Y(x_{i}))\) is thus bounded by a uniform constant determined by $\mf{S}$, $Q$ and $\lambda$.

\textbf{Case 4:} \(V \sqsubseteq U\). Recall that $\pi_U(\gamma)$ is a unparameterized quasi-geodesic in $CU$, and let $\gamma_0$ be the subsegment of $\pi_U(\gamma)$ from $x_{I_V}$ to $x_i$ and $\gamma_1$ be the subsegment from $x_i$ to $x_{J_V}$. By the bounded geodesic image axiom and (\ref{eq:active_subpaths_middle_piece}),  there exists $E'\geq E$ determined by $\mf{S}$, such that both $\gamma_0$ and $\gamma_1$ intersect the $E'$--neighborhood of $\rho_U^V$. Since $\pi_U(\gamma)$ is a unparameterized $(\lambda,\lambda)$--quasi-geodesic, there exists $R$ depending on $E'$ and $\lambda$ such that $d_U(x_i,\rho_U^V) \leq R$. If $\alpha$ is some $CU$ geodesic connecting $\gate_Y(x)$ and $\gate_Y(y)$, then $\alpha$ also intersects the $E$--neighborhood of $\rho_U^V$ by the bounded geodesic image axiom.   Therefore, by the quasiconvexity of $\pi_U(Y)$ in $CU$ and the properties of the gate map, the distance \(d_{U}(x_{i}, \gate_Y(x_{i}))\) is bounded by a uniform constant determined by $\mf{S}$, $Q$ and $\lambda$. 
\end{proof}

\begin{rem}
The hypotheses of Proposition \ref{prop: active subpaths for quasiconvex} cannot be relaxed by taking $Y$  to be hierarchically quasiconvex instead of {\sqc}.  As a counterexample, one can consider $\mathbb{Z}^2$ with the standard HHG structure and let $Y$ be the $x$--axis.  As any horizontal line  in $\mathbb{Z}^2$ is a hierarchy path, for any $D>0$, there exists a hierarchy path $\gamma$ where both $d_\mc{X}(\gamma,Y)>D$ and $\diam(\gate_Y(\gamma))>D$. 
\end{rem}

\section{{\SQC} subsets in familiar examples} \label{sec: Examples}

In this section, we utilize Theorem \ref{thm:classification_of_quasiconvex_subset} to give descriptions of the {\sqc} subsets in well studied examples of hierarchically hyperbolic spaces. We will begin by briefly discussing the  HHS structure for the mapping class group, Teichm\"uller space, right-angled Artin and Coxeter groups, and graph manifolds. The descriptions are not complete as we only describe the parts of the HHS structure that we shall need to be able to apply the results from the general case. 
We direct the reader to the  references provided along side each example for complete details.
\\

\noindent\textbf{The Mapping Class Group and Teichm\"uller Space}\\
\noindent Mapping class group \cite{MMII, BHS17b}, Teichm\"uller metric \cite{Durham,Rafi,EMR}, Weil-Petersson metric \cite{Brock}.\\

Let $S$ be an oriented, connected, finite type surface with genus $g$ and $p$ punctures. The \emph{complexity of $S$} is $\xi(S) = 3g-3+p$. Assume $\xi(S) \geq 1$ and let $\mc{X}$ be the marking complex of $S$.

\begin{itemize}
\item \textit{Index Set:} $\mf{S}$ will be the collection of isotopy classes of (non-necessarily connected) essential subsurfaces of $S$ excluding $3$--punctured sphere, but including annuli. 

\item \textit{Hyperbolic Spaces:} For each $U \in \mf{S}$, $CU$ will be the curve graph of $U$. The space $CU$ will be infinite diameter if and only if $U$ is connected. The projection maps, $\pi_U$, are the well studied subsurface projections of Masur-Minsky.

\item \textit{Relations:} $U \perp V$ if $U$ and $V$ are disjoint and $U \nest V$ if $U$ is nested into $V$. If $U\nest V$, then $\rho_V^U$ will be the subset of curves in $CV$ corresponding to $\partial U$.

\end{itemize}

\noindent The HHS structure for Teichm\"uller space with either metric is identical except for the annular domains of $\mf{S}$. For the Teichm\"uller metric, modify the curve graphs of the annular domains by attaching a horoball. For the Weil-Petersson metric, the index set \(\mf{S}\) simply excludes annuli. This difference in the treatment of annular domains accounts for all of the differences in the coarse geometry of the these three spaces.\\

\noindent\textbf{RAAGs and RACGs} \cite{BHS}

Let $\Gamma$ be a finite simplicial graph and $G_\Gamma$ be the associated right-angled Artin or right-angled Coxeter group equipped with a word metric from a finite generating set. For an induced subgraph $\Lambda \subseteq \Gamma$, $link(\Lambda)$ is  the subgraph of $\Gamma - \Lambda$ induced by the vertices adjacent to every vertex in $\Lambda$ and $star(\Lambda)$ be the induced subgraph of $link(\Lambda) \cup \Lambda$.   If $\Lambda$ is an induced subgraph of $\Gamma$, then $G_\Lambda$ is a subgroup of $G_\Gamma$. We call subgroups of this form the \emph{special subgroups} of $G_\Gamma$. The following is an HHG structure on $G_\Gamma$.

\begin{itemize}
\item \textit{Index Set:} For $g,h \in G_{\Gamma}$ and $\Lambda$ a non-empty, induced subgraph of $\Gamma$, define the  equivalence relation $gG_{\Lambda}\sim hG_{\Lambda}$ if $g^{-1}h \in G_{star(\Lambda)}$. Let  $\mf{S}$ be defined as $ \{gG_{\Lambda}\}/\sim$. 
\item \textit{Hyperbolic Spaces:} $C[gG_{\Lambda}]$ can be obtained by starting with the  coset $gG_\Lambda$ and coning off each left coset of the special subgroups contained in $gG_\Lambda$.   $C[gG_{\Lambda}]$ is infinite diameter if and only if $G_\Lambda$ is infinite and $\Lambda$ does not split as a join.
\item  \textit{Relations:} $[gG_{\Lambda'}]\nest [gG_{\Lambda}]$ if $\Lambda' \subseteq \Lambda$  and $[gG_{\Lambda'}]\perp [gG_{\Lambda}]$ if  $\Lambda \subseteq link(\Lambda')$ (and hence $\Lambda' \subseteq link(\Lambda)$). If $[gG_{\Lambda'}]\nest [gG_{\Lambda}]$, then $\rho_{[gG_{\Lambda}]}^{[gG_{\Lambda'}]}$ will be the subset $gG_{\Lambda'}$ in $C[gG_{\Lambda}]$.
\end{itemize}

\noindent \textbf{Graph Manifolds} \cite{BHS17b}

Let $M$ be a non-geometric graph manifold and  $\mc{X}$ be the universal cover of $M$. Since the fundamental group of every graph manifold is quasi-isometric to the fundamental group of a flip graph manifold, we will assume $M$ is flip.  Let $T$ be Bass-Serre tree for $M$ and $X_v$ be the subspace of $\mc{X}$ corresponding to a vertex $v \in T$. Each $X_v$ is bi-Lipschitz to the product $R_v \times H_v$ where $R_v$ is a copy of the real line and $H_v$ is the universal cover of a hyperbolic surface with totally geodesic boundary. If $v,w$ are adjacent vertices in $T$, then let $\partial_w H_v$ and $\partial_v H_w$ denote the boundary components of $H_v$ and $H_w$ such that $R_v \times \partial_w H_v$ is identified with $R_w \times \partial_v H_w$ in $\mc{X}$. Since $M$ is flip, $R_v$ is identified with $\partial_v H_w$. For each $v \in T$, let $\widehat{H}_v$ denote the spaced obtained from $H_v$ after coning off each copy of $ \partial_w H_v$ for each vertex $w$ adjacent to $v$. The following is a HHS structure on $\mc{X}$.

\begin{itemize}
\item \textit{Index Set:} For adjacent vertices $v,w \in T$, define $R_v\sim\partial_v H_w$ then let $\mf{S} = \{T,R_v,\partial_v H_w, \widehat{H}_w\}/\sim$.
\item \textit{Hyperbolic Spaces:} Every element of $\mf{S}$ is a hyperbolic space, so we have $CU = U$ for all $U \in \mf{S}$. The diameter of $CU$ is infinite for all $U\in \mf{S}$.
\item \textit{Relations:} $T$ is the $\nest$--maximal domain and $[\partial_w H_v] \nest \widehat{H}_v$ for all $w,v$ adjacent in $T$.  For adjacent vertices $v,w \in T$, $\rho_T^{[R_v]} = \rho_T^{[\partial_w H_v]} = \{v,w\} \subset T$ and $\rho_{\widehat{H}_v}^{[\partial_w H_v]}$ is the cone point for $\partial_w H_v$ in $\widehat{H}_w$. For $v,w$ adjacent in $T$, $[R_v] \perp \widehat{H}_v$  and $[R_v] \perp [R_w]$ (recall $[\partial_w H_v] = [R_w]$).
\end{itemize}

\begin{rem}
When the manifold $M$ is flip, the above describes an HHG structure on $\pi_1(M)$.  However, if $M$ is not flip, then the quasi-isometry from $\pi_1(M)$ to the fundamental group of the flip graph manifold need not be equivariant and the above will be an HHS, but not an HHG structure on $\pi_1(M)$. See \cite[Remark 10.2]{BHS17b} for a discussion of the existence of HHG structures on $3$--manifold groups.
\end{rem}

In the case of right-angled Artin groups with connected defining graphs, Tran and Genevois independently showed that {\sqc} subgroups are either finite index or hyperbolic(and are actually free when they are hyperbolic) \cite{Tran2017,Genevois_Hyp_in_CAT(0)}. The same result is shown for the mapping class group in \cite{Kim2017} and for certain $\mathcal{CFS}$ right-angled Coxeter groups in \cite{HH2017}. Based on these examples, one may conjecture that the {\sqc} subsets of any non-relatively hyperbolic, hierarchically hyperbolic space are either hyperbolic or coarsely cover the entire space. While \cite{Tran2017} provides a counterexample to this conjecture in right-angled Coxeter groups, it nevertheless holds in many of the examples described above.  In Proposition \ref{prop:one_orthogonal_component}, we give sufficient conditions for every {\sqc} subset of an HHS to be either hyperbolic or coarsely covering. We then unite and expand the work of Tran, Genevois, Kim, and Nguyen-Tran by applying Proposition \ref{prop:one_orthogonal_component} to the mapping class group, Teichm\"uller space, right-angled Artin and Coxeter groups, and graph manifolds in Corollary \ref{cor:MCG_Teich_RAAG}.

\begin{prop}\label{prop:one_orthogonal_component}
Let $(\mc{X},\mf{S})$ be an HHS with the bounded domain dichotomy and let $\mf{S}^*$ be as defined in Definition \ref{def:S_star}. Assume the following two conditions hold:
\begin{enumerate}[(1)]
	\item \label{item:cover} For all $W \in \mf{S}-\mf{S}^*$ either $CW$ has bounded diameter or the set  \[\{\rho_W^V : V \in \mf{S}^* \text{ with } V\trans W \text{ or } V\nest W \}\] uniformly coarsely covers $CW$.
	
	\item \label{item:connected}  For every $U,V \in\mf{S}^*$ there exists a sequence $U=U_1,\cdots, U_n=V$ of domains in $\mf{S}^*$ with $U_i \perp U_{i+1}$ for all $1\leq i\leq n-1$.
\end{enumerate}
\noindent Then, if $Y\subseteq\mc{X}$ is {\sqc}, either $Y$ is hyperbolic or some finite neighborhood of $Y$ covers all of $\mc{X}$.
\end{prop}

\begin{proof}
Let $Y\subseteq \mc{X}$ be $Q$--{\sqc}. By Theorem \ref{thm:classification_of_quasiconvex_subset} there exists $B$, depending only on $Q$ and $\mf{S}$, such that $Y$ has the $B$--orthogonal projection dichotomy. Further, we can assume $B$ is large enough such that $(\mc{X},\mf{S})$ satisfies the $B$--bounded domain dichotomy. We will show that if $Y$ is not hyperbolic, then for all $W \in \mf{S}$ we have that $CW$ is uniformly coarsely covered by $\pi_W(Y)$.  Thus for all $x \in \mc{X}$ we will have that $d_W(x,\gate_{Y}(x))$ is uniformly bounded and therefore $Y$ will coarsely cover $\mc{X}$ by the distance formula (Theorem \ref{thm:distance_formula}).

Suppose that $Y$ is not hyperbolic.  By Proposition \ref{defn:stable}, the inclusion map $i:Y\hookrightarrow X$  cannot be a stable embedding. Therefore by Corollary \ref{cor:classification_of_stable_subsets}, there exists a domain $U \in \mf{S}^{*}$ such that $\diam(\pi_U(Y)) >B$. First we will show that for any domain $W \in \mf{S}^*$, $CW \subseteq {N}_B(\pi_W(Y))$.

Let $W \in \mf{S}^*$. By hypothesis, there exists a sequence $U=U_1,\cdots, U_n=W$ of domains in $\mf{S}^*$ with $U_i \perp U_{i+1}$ for all $1\leq i\leq n-1$.  
Since $Y$ has the $B$--orthogonal projection dichotomy and $\diam(CU_i) = \infty$ for each $1\leq i \leq n$, we have 
$CU_{i} \subseteq {N}_B(\pi_{U_{i}}(Y))$ for all $1\leq i \leq n$.   In particular, $CW \subseteq N_B(\pi_{W}(Y))$.

Now let $W\in \mf{S} - \mf{S}^*$ such that $ \diam(CW) = \infty$.  We will show that $\pi_W(Y)$ uniformly coarsely covers $CW$ by showing that for all $V \in\mf{S}^*$ such that $\rho_W^V$ is defined there exists $y \in Y$ such that $\pi_W(y)$ is uniformly close to $\rho_W^V$. First suppose $V \in \mf{S}^*$ with $V\nest W$. By the preceding paragraph, there exist $x,x' \in Y$ such that $d_V(x,x') > 100E$.  If $\gamma$ is a hierarchy path connecting $x$ and $x'$, then  $\pi_W(\gamma)$ is uniformly close to $\rho_W^V$ by the bounded geodesic image axiom (Axiom \ref{item:dfs:bounded_geodesic_image}). Further, since $Y$ is {\sqc} there exists $y \in Y$ such that $d_W(\rho_W^V,\pi_W(y))<B'$ where $B'$ depends only on $Q$ and $\mf{S}$. If instead $V \in \mf{S}^*$ and $V\trans W$, then there exists $y \in Y$ such that $d_V(y,\rho_V^W) > \kappa_0$.  Thus $d_W(y,\rho_W^V)\leq\kappa_0$ by the consistency axiom (Axiom \ref{item:dfs_transversal}). Since the set $\{\rho_W^V : V \in \mf{S}^* \text{ with } V\trans W \text{ or } V\nest W \}$ uniformly coarsely covers $CW$ by hypothesis, we have that $\pi_W(Y)$ uniformly coarsely covers all of $CW$ as well. 

Hence we have that for all $W \in \mf{S}$, $CW$ is uniformly coarsely covered by $\pi_W(Y)$ and so $Y$ coarsely covers $\mc{X}$ by the distance formula.
\end{proof}

Before continuing, we will take a brief detour to define a property of graphs that will be relevant to our study of right-angled Coxeter groups. Given a graph $\Gamma$, define $\Gamma^4$ as the graph whose vertices are induced 4--cycles of $\Gamma$. Two vertices in $\Gamma^4$ are adjacent if and only if the corresponding induced 4-cycles in $\Gamma$ have two non-adjacent vertices in common. Given graphs $\Lambda_1, \Lambda_2$ recall that the join $\Lambda_1\ast\Lambda_2$, is the graph obtained  from $\Lambda_1\sqcup \Lambda_2$ by adding an edge between each vertex of $\Lambda_1$ and each vertex of $\Lambda_2$.

\begin{defn}[Constructed from squares]\label{defn:CFS}
A graph $\Gamma$ is \emph{$\mathcal{CFS}$} if $\Gamma=\Omega*K$, where $K$ is a (possibly empty) clique and $\Omega$ is a non-empty subgraph such that $\Omega^4$ has a connected component $T$ where every vertex of $\Omega$ is contained in a $4$--cycle that is a vertex of $T$. If $\Gamma$ is $\mc{CFS}$ and $\Omega^4$ is connected, then we say $\Gamma$ is \emph{strongly $\mc{CFS}$}. If $\Gamma$ is (strongly) $\mc{CFS}$, then by abuse of language we will say that the right-angled Coxeter group $G_\Gamma$ is (strongly) $\mc{CFS}$. See Figure \ref{aone} below for examples of $\mc{CFS}$ and strongly $\mc{CFS}$ graphs.
\end{defn}

\begin{figure}[H]
\begin{tikzpicture}[scale=.4]

	\draw (0-10,0.5-20) node[circle,fill,inner sep=.75pt, color=black](1){} -- (0.5-10,0-20) node[circle,fill,inner sep=.75pt, color=black](1){} -- (0-10,-0.5-20) node[circle,fill,inner sep=.75pt, color=black](1){} -- (-0.5-10,0-20) node[circle,fill,inner sep=.75pt, color=black](1){} -- (0-10,0.5-20);
	
	\draw (-0.5-10,0-20) node[circle,fill,inner sep=.75pt, color=black](1){} -- (0-10,1-20) node[circle,fill,inner sep=.75pt, color=black](1){} -- (0.5-10,0-20)  node[circle,fill,inner sep=.75pt, color=black](1){} -- (0-10,-1-20) node[circle,fill,inner sep=.75pt, color=black](1){} -- (-0.5-10,0-20);
	
	\draw (0-10,1-20) node[circle,fill,inner sep=.75pt, color=black](1){}  -- (1-10,0-20) node[circle,fill,inner sep=.75pt, color=black](1){}  -- (0-10,-1-20) node[circle,fill,inner sep=.75pt, color=black](1){} -- (-1-10,0-20) node[circle,fill,inner sep=.75pt, color=black](1){} -- (0-10,1-20);
	
	\draw (-1-10,0-20) node[circle,fill,inner sep=.75pt, color=black](1){} -- (0-10,2-20) node[circle,fill,inner sep=.75pt, color=black](1){} -- (1-10,0-20) node[circle,fill,inner sep=.75pt, color=black](1){}  -- (0-10,-2-20)node[circle,fill,inner sep=.75pt, color=black](1){}  -- (-1-10,0-20); 
	
	\draw (0-10,2-20) node[circle,fill,inner sep=.75pt, color=black](1){}  -- (2-10,0-20)node[circle,fill,inner sep=.75pt, color=black](1){}  -- (0-10,-2-20) node[circle,fill,inner sep=.75pt, color=black](1){} -- (-2-10,0-20)node[circle,fill,inner sep=.75pt, color=black](1){}  -- (0-10,2-20);
	
	\draw (-2-10,0-20) node[circle,fill,inner sep=.75pt, color=black](1){}  -- (0-10,4-20) node[circle,fill,inner sep=.75pt, color=black](1){}  -- (2-10,0-20) node[circle,fill,inner sep=.75pt, color=black](1){} -- (0-10,-4-20) node[circle,fill,inner sep=.75pt, color=black](1){} -- (-2-10,0-20);
	
	\draw (0-10,4-20) node[circle,fill,inner sep=.75pt, color=black](1){}  -- (4-10,0-20) node[circle,fill,inner sep=.75pt, color=black](1){}  -- (0-10,-4-20) node[circle,fill,inner sep=.75pt, color=black](1){} -- (-4-10,0-20) node[circle,fill,inner sep=.75pt, color=black](1){} -- (0-10,4-20);
	
	\draw (-4-10,0-20)node[circle,fill,inner sep=.75pt, color=black](1){}  -- (0-10,6-20) node[circle,fill,inner sep=.75pt, color=black](1){} -- (4-10,0-20) node[circle,fill,inner sep=.75pt, color=black](1){} -- (0-10,-6-20)node[circle,fill,inner sep=.75pt, color=black](1){}  -- (-4-10,0-20);
	
	\draw (0-10,6-20) node[circle,fill,inner sep=.75pt, color=black](1){} -- (6-10,0-20) node[circle,fill,inner sep=.75pt, color=black](1){} -- (0-10,-6-20) node[circle,fill,inner sep=.75pt, color=black](1){}  -- (-6-10,0-20) node[circle,fill,inner sep=.75pt, color=black](1){} -- (0-10,6-20);
	
	\node at (0-10,11-23) {$\Gamma_1$};

	\draw (0+5,0.5-20) node[circle,fill,inner sep=.75pt, color=black](1){} -- (0.5+5,0-20) node[circle,fill,inner sep=.75pt, color=black](1){} -- (0+5,-0.5-20) node[circle,fill,inner sep=.75pt, color=black](1){} -- (-0.5+5,0-20) node[circle,fill,inner sep=.75pt, color=black](1){} -- (0+5,0.5-20);
	
	\draw (-0.5+5,0-20) node[circle,fill,inner sep=.75pt, color=black](1){} -- (0+5,1-20) node[circle,fill,inner sep=.75pt, color=black](1){} -- (0.5+5,0-20)  node[circle,fill,inner sep=.75pt, color=black](1){} -- (0+5,-1-20) node[circle,fill,inner sep=.75pt, color=black](1){} -- (-0.5+5,0-20);
	
	\draw (0+5,1-20) node[circle,fill,inner sep=.75pt, color=black](1){}  -- (1+5,0-20) node[circle,fill,inner sep=.75pt, color=black](1){}  -- (0+5,-1-20) node[circle,fill,inner sep=.75pt, color=black](1){} -- (-1+5,0-20) node[circle,fill,inner sep=.75pt, color=black](1){} -- (0+5,1-20);
	
	\draw (-1+5,0-20) node[circle,fill,inner sep=.75pt, color=black](1){} -- (0+5,2-20) node[circle,fill,inner sep=.75pt, color=black](1){} -- (1+5,0-20) node[circle,fill,inner sep=.75pt, color=black](1){}  -- (0+5,-2-20)node[circle,fill,inner sep=.75pt, color=black](1){}  -- (-1+5,0-20); 
	
	\draw (0+5,2-20) node[circle,fill,inner sep=.75pt, color=black](1){}  -- (2+5,0-20)node[circle,fill,inner sep=.75pt, color=black](1){}  -- (0+5,-2-20) node[circle,fill,inner sep=.75pt, color=black](1){} -- (-2+5,0-20)node[circle,fill,inner sep=.75pt, color=black](1){}  -- (0+5,2-20);
	
	\draw (-2+5,0-20) node[circle,fill,inner sep=.75pt, color=black](1){}  -- (0+5,4-20) node[circle,fill,inner sep=.75pt, color=black](1){}  -- (2+5,0-20) node[circle,fill,inner sep=.75pt, color=black](1){} -- (0+5,-4-20) node[circle,fill,inner sep=.75pt, color=black](1){} -- (-2+5,0-20);
	
	\draw (0+5,4-20) node[circle,fill,inner sep=.75pt, color=black](1){}  -- (4+5,0-20) node[circle,fill,inner sep=.75pt, color=black](1){}  -- (0+5,-4-20) node[circle,fill,inner sep=.75pt, color=black](1){} -- (-4+5,0-20) node[circle,fill,inner sep=.75pt, color=black](1){} -- (0+5,4-20);
	
	\draw (-4+5,0-20)node[circle,fill,inner sep=.75pt, color=black](1){}  -- (0+5,6-20) node[circle,fill,inner sep=.75pt, color=black](1){} -- (4+5,0-20) node[circle,fill,inner sep=.75pt, color=black](1){} -- (0+5,-6-20)node[circle,fill,inner sep=.75pt, color=black](1){}  -- (-4+5,0-20);
	
	\draw (0+5,6-20) node[circle,fill,inner sep=.75pt, color=black](1){} -- (6+5,0-20) node[circle,fill,inner sep=.75pt, color=black](1){} -- (0+5,-6-20) node[circle,fill,inner sep=.75pt, color=black](1){}  -- (-6+5,0-20) node[circle,fill,inner sep=.75pt, color=black](1){} -- (0+5,6-20);

	\draw (0+5,0.5-20)[color=red] node[circle,fill,inner sep=.75pt, color=black](1){} -- (1+5,0-20) node[circle,fill,inner sep=.75pt, color=black](1){}-- (0+5,4-20) node[circle,fill,inner sep=.75pt, color=black](1){} -- (-6+5,0-20) node[circle,fill,inner sep=.75pt, color=black](1){} -- (0+5,0.5-20) node[circle,fill,inner sep=.75pt, color=black](1){};
	
	\node at (0+5,11-23) {$\Gamma_2$};

\end{tikzpicture}

\caption{Two graphs $\Gamma_1$ and $\Gamma_2$ are both $\mathcal{CFS}$. However, graph $\Gamma_1$ is graph strongly $\mathcal{CFS}$ but $\Gamma_2$ is not since the red induced 4-cycle in $\Gamma_2$ is not ``connected'' to any other induced 4-cycle in the graph.}
\label{aone}
\end{figure}
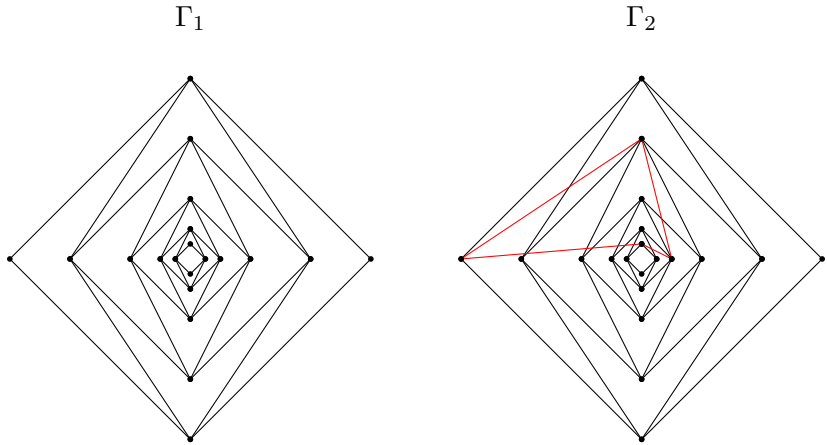 

\begin{cor}\label{cor:MCG_Teich_RAAG}
The following HHSs have the property that every {\sqc} subset is either hyperbolic or coarsely covers the entire space:
\begin{enumerate}[(a)]

	\item The Teichm\"uller space of a finite type, oriented surface with the Teichm\"uller metric
	\item The Teichm\"uller space of a finite type, oriented surface of complexity at least $6$ with the Weil-Petersson metric
	\item The mapping class group of a finite type, oriented surface
	\item A right-angled Artin group with connected defining graph
	\item A right-angled Coxeter group with strongly $\mathcal{CFS}$ defining graph
	\item The fundamental group of a non-geometric graph manifold
\end{enumerate}

\noindent In particular, if  $H$ is a {\sqc} subgroup in any of the groups (c)-(f),  then $H$ is either stable or finite index.
\end{cor}

\begin{proof}
All of the above examples have the bounded domain dichotomy.  We shall show they satisfy the two hypotheses of Proposition \ref{prop:one_orthogonal_component}.

\textit{Mapping class group/Teichm\"uller metric:} If $\xi(S)\leq 1$, then the mapping class group and Teichm\"uller space will both be hyperbolic; thus we can assume $\xi(S)\geq 2$. In this case, $\mf{S}^*$ is the set of all connected proper subsurfaces. Thus Hypothesis (\ref{item:cover}) follows from the fact that every curve on the surface corresponds to the boundary curve of some connected subsurface. Given two subsurfaces $U$, $V$ a sequence satisfying Hypothesis (\ref{item:connected}) is found by taking a path in $CS$ connecting $\partial U$ and $\partial V$.

\textit{Weil-Petersson:} $\mf{S}^*$ is the collection of all connected proper subsurfaces whose complement contains a subsurface of complexity at least $1$. In particular, since the complexity is at least $6$, $\mf{S}^*$ contains every subsurface of complexity $1$.  For every connected subsurface $W \not\in \mf{S}^*$, every curve on $W$ corresponds to the boundary curve of some complexity $1$ subsurface providing Hypothesis (\ref{item:cover}). Hypothesis (\ref{item:connected}) follows from the observations that if $U \subset S$ is a subsurface of complexity $1$ and $\alpha$ is a curve disjoint from $U$, then there exists $V \subseteq S$, a subsurface of complexity $1$, such that $\alpha \subseteq \partial V $ and $U$ is disjoint from $V$. Thus any path in $CS$ can be promoted to a sequence of sequentially disjoint subsurfaces in $\mf{S}^*$.

\textit{RAAGs:} $\mf{S}^*$ is the collection of $[gG_\Lambda]$ such that there exists \(\Delta\subseteq link(\Lambda)\) where \(\Lambda\) and \(\Delta\) are both non-empty and not joins. In particular, since $\Gamma$ is connected, $\mf{S}^*$ contains all of the $[gG_\Lambda]$ where $\Lambda$ is a single vertex. 
Hypothesis (\ref{item:cover}) follows from the fact that $G_\Lambda$ acts cocompactly on its Cayley graph and the construction of $C[gG_\Lambda]$. For Hypothesis (\ref{item:connected}), let $[g_1G_{\Lambda_1}],[g_2G_{\Lambda_2}] \in \mf{S}^*$ and $m=|g_1^{-1}g_2|$. 
We shall proceed by induction on $m$.
If $m=0$, then $g_1=g_2=g$ and since $\Gamma$ is connected, there is a sequence of vertices $v_1, v_2, \cdots, v_n$ such that $v_i$ and $v_{i+1}$ are adjacent for all $1\leq i\leq n-1$ and $v_1 \in link(\Lambda_1)$, $v_n \in link(\Lambda_2)$. Thus $[gG_{\Lambda_1}],[gG_{v_1}],\hdots,[gG_{v_n}], [gG_{\Lambda_2}]$ is the required sequence. 

If $m>0$, then there 
exists $g_3 \in G_\Gamma$ such that $|g_1^{-1}g_3| = m-1$ and $|g_3^{-1}g_2|=1$. Let $v$ be the vertex of $\Gamma$ such 
that $g_3^{-1}g_2$ is either $v$ or $v^{-1}$. By induction, there exist two sequences of elements of $\mf{S}^*$,  \[ [g_1G_{\Lambda_1}] = U_1,U_2, \hdots,U_n = [g_3G_{v}] \] and \[[g_2G_{v}] = V_1,V_2,\hdots,V_k = [g_2G_{\Lambda_2}]\] such that $U_i \perp U_{i+1}$ for $1\leq i \leq n-1$ and $V_i\perp V_{i+1}$ for all $1 \leq i \leq k-1$. Since $[g_3G_{v}] = [g_2G_v]$,  \[[g_1G_{\Lambda_1}] = U_1,U_2, \hdots,U_n, V_2,\hdots, V_n = [g_2G_{\Lambda_2}]\] is the required sequence.

\textit{RACGs:} Since $\Gamma$ is strongly $\mathcal{CFS}$, we can write $\Gamma=\Omega*K$ where $K$ is a clique (possibly empty) and $\Omega$ is a non-empty graph such that $\Omega^4$ is connected and every vertex of $\Omega$ is contained a $4$--cycle that is a vertex of $\Omega^4$. Since $G_\Omega$ is a finite index subgroup of $G_\Gamma$, it suffices to prove that every {\sqc} subset of $G_\Omega$ is either hyperbolic or coarsely covers $G_\Omega$.  We now prove that the standard HHG structure, $\mf{S}$, on $G_\Omega$ satisfies satisfy the two hypotheses of Proposition \ref{prop:one_orthogonal_component}. The argument will be similar to the case of right-angled Artin groups above. 

We first observe that $\mf{S}^*$ is the collection of $[gG_\Lambda]$ such that there exists \(\Delta\subseteq link(\Lambda)\) where \(\Lambda\) and \(\Delta\) both have at least 2 points and they are not joins. In particular, $\mf{S}^*$ contains all domains $[gG_{\{a,b\}}]$ where $a$ and $b$ are two non-adjacent vertices of an induced 4-cycle. Hypothesis (\ref{item:cover}) follows from the fact that $G_\Lambda$ acts cocompactly on its Cayley graph and the construction of $C[gG_\Lambda]$.

For Hypothesis (\ref{item:connected}), let $[g_1G_{\Lambda_1}],[g_2G_{\Lambda_2}] \in \mf{S}^*$ and $m=|g_1^{-1}g_2|$. We shall proceed by induction on $m$. We first assume that $m=0$. Therefore, $g_1=g_2=g$. We note that for $i=0 \text{ or }1$ there exists \(\Delta_i\subseteq link(\Lambda_i)\) where \(\Lambda_i\) and \(\Delta_i\) both contain at least 2 vertices and are not joins. Therefore, $link(\Lambda_i)$ contains a pair $(u_i,v_i)$ of two non-adjacent vertices of some induced $4$ cycle. Since $\Omega^4$ is connected, there is a sequence of pairs of non-adjacent vertices $(u_1,v_1)=(a_1,b_1), (a_2,b_2), \cdots, (a_n,b_n)=(u_2,v_2)$ such that $a_i$ and $b_i$ are both adjacent to $a_{i+1}$ and $b_{i+1}$ for all $1\leq i\leq n-1$. Thus $[gG_{\Lambda_1}],[gG_{\{a_1,b_1\}}],\hdots,[gG_{\{a_n,b_n\}}], [gG_{\Lambda_2}]$ is the required sequence. 

If $m>0$, then there exists $g_3 \in G_\Omega$ such that $|g_1^{-1}g_3| = m-1$ and $|g_3^{-1}g_2|=1$. Let $v$ be the vertex of $\Omega$ such 
that $g_3^{-1}g_2=v$. Since every vertex of $\Omega$ is contained in a four cycle that is a vertex of $\Omega^4$, there is a vertex $w$ such that $v$ and $w$ are two non-adjacent vertices of an induced 4-cycle. By induction, there exist two sequences of elements of $\mf{S}^*$, \[[g_1G_{\Lambda_1}] = U_1,U_2, \hdots,U_n = [g_3G_{\{v,w\}}]\]  and \[[g_2G_{\{v,w\}}] = V_1,V_2,\hdots,V_k = [g_2G_{\Lambda_2}]\] such that $U_i \perp U_{i+1}$ for $1\leq i \leq n-1$ and $V_i\perp V_{i+1}$ for all $1 \leq i \leq k-1$. Since $[g_3G_{\{v,w\}}] = [g_2G_{\{v,w\}}]$,  \[[g_1G_{\Lambda_1}] = U_1,U_2, \hdots,U_n, V_2,\hdots, V_n = [g_2G_{\Lambda_2}]\] is the required sequence.

\textit{Graph Manifolds:} In this case, $\mf{S}^* = \mf{S} -\{T\}$ and Hypothesis (\ref{item:cover}) is immediate from the facts that for every vertex $v \in T$ is an element of $\rho_T^{[R_v]}$ and every point in $H_v$ is uniformly close to some boundary component $\partial_w H_v$. For Hypothesis (\ref{item:connected}), consider $U,W \in \mf{S}^*$. If $U=[R_u]$ and $W=[R_w]$, let $v_1,\hdots,v_n$ be a sequence of adjacent vertices in $T$ such that $v_1$ is adjacent to $u$ and $v_n$ is adjacent to $w$. In this case the sequence $[R_u], [R_{v_1}],\hdots,[R_{v_n}], [R_w]$ satisfies the hypothesis. If $U = [\widehat{H}_u]$ or $W = [\widehat{H}_w]$, the hypothesis is satisfied by adding $[\widehat{H}_u]$ before $[R_u]$ or $[\widehat{H}_w]$ after $[R_w]$ to $[R_u], [R_{v_1}],\hdots,[R_{v_n}], [R_w]$ as needed.
\end{proof}

In the setting of $2$--dimensional right-angled Coxeter groups, Tran provided a characterization of the special {\sqc} subgroups \cite{Tran2017}.  This characterization was expanded by Genevois to include all right-angled Coxeter groups in \cite{Genevois_Hyp_in_CAT(0)}.  We provide a new proof of this characterization using Theorem \ref{thm:classification_of_quasiconvex_subset}.

\begin{thm}[\cite{Tran2017}, \cite{Genevois_Hyp_in_CAT(0)}]
\label{thm:quasiconvex_subgroups_of_RACG}
Let $\Gamma$ be a simplicial graph and $\Delta$ an induced subgraph of $\Gamma$. If $G_\Gamma$ is the right-angled Coxeter group corresponding to $\Gamma$ and $G_\Delta$ is the subgroup generated by the vertices of $\Delta$, then the following conditions are equivalent:
\begin{enumerate}
	\item The subgroup $G_\Delta$ is {\sqc} in $G_\Gamma$.
	\item If $\Delta$ contains two non-adjacent vertices of an induced 4--cycle $\sigma$, then $\Delta$ contains all vertices of $\sigma$.
\end{enumerate}
\end{thm}

\begin{proof}
Before we begin, we document a few additional facts we will need about the HHG structure on a right-angled Coxeter group. For any induced subgraph $\Lambda$, $\P_{[G_\Lambda]}$  is coarsely equal to the subgroup $G_\Lambda \times G_{link(\Lambda)}$ and $G_\Lambda$ can be coarsely identified with $\PF_{[G_\Lambda]}$.  In particular, $G_\Lambda$ is hierarchically quasiconvex, $\pi_U(G_\Lambda)$  uniformly coarsely covers $CU$ for $U \nest [G_\Lambda]$, and $\pi_V(G_\Lambda)$ is  uniformly bounded for all $V \not \nest [G_\Lambda]$.  See \cite{BHS} for full details on the HHG structure on right-angled Coxeter groups.

$(1) \implies (2):$ Assume for a contradiction that $G_\Delta$ is {\sqc}, but there is a 4--cycle $\sigma$ with two pairs of non-adjacent vertices $\{a_1,a_2\}$ and $\{b_1,b_2\}$ such that $\{a_1,a_2\}$ is a subset of $\Delta$ and $\{b_1,b_2\}$ is not. We know that $U=[G_{\{a_1,a_2\}}]$ and $[G_{\{b_1,b_2\}}]=V$ are orthogonal domains. However, $\pi_U(G_\Delta)$ coarsely covers $CU$, but $\pi_V(G_\Delta)$ has uniformly bounded diameter which contradicts Theorem \ref{thm:classification_of_quasiconvex_subset}.

$(2) \implies (1):$ As $G_\Delta$ is hierarchically quasiconvex, we only need to demonstrate that $G_\Delta$ satisfies the orthogonal projection dichotomy. Let $B$ be a positive number such that $(G_\Gamma,\mf{S})$ has the $B$--bounded domain dichotomy, $CW \subseteq N_B(\pi_W(G_\Delta))$ for all $W \nest [G_\Delta]$, and $\diam(\pi_W(G_\Delta))<B$ for all $W \not \nest [G_\Delta]$.  If $\diam(\pi_U(G_\Delta)) > B$, then it must be the case that $U = [G_\Lambda]$ where  $\Lambda \subseteq \Delta$ and $\Lambda$ contains two non-adjacent vertices $s$ and $t$. If $V\in\mf{S}_U^\perp$, then  $V = [G_{\Lambda'}]$ where $\Lambda' \subseteq link(\Lambda)$ and $\Lambda \subseteq link(\Lambda')$. If  $\Lambda'$ is a join or $\Lambda' = \{v\}$, then $\diam(CV) \leq B$ and $CV \subseteq N_{2B}(\pi_V(G_\Delta))$. In the other case, we will show $\Lambda' \subseteq \Delta$. 

If $\Lambda'$ is not a join and contains at least two vertices, then for each vertex $v \in \Lambda'$ there exists a vertex $w \in \Lambda'$ that is not adjacent to $v$. Since $\Lambda \subseteq link(\Lambda')$, the vertices $v,s,w,t$ form a $4$--cycle. However, (2) then requires $ v,w \in \Delta$. Hence, $\Lambda' \subseteq \Delta$ and $V = [G_{\Lambda'}] \nest [G_\Delta]$ implying $CV \subseteq N_B(\pi_V(G_\Delta))$. Thus $G_\Delta$ has the $2B$--orthogonal projection dichotomy and we are finished by Theorem \ref{thm:classification_of_quasiconvex_subset}.  
\end{proof}

\subsection{CFS right-angled Coxeter groups}

Recently, Behrstock proposed the program of classifying all $\mc{CFS}$ right-angled Coxeter groups up to quasi-isometry and 
commensurability. This was motivated by the genericity of $\mc{CFS}$ right-angled Coxeter groups among random right-angled 
Coxeter groups as well as the fact that being $\mc{CFS}$ is a necessary (but not sufficient) condition for a right-angled Coxeter group 
to be quasi-isometric to a right-angled Artin group (see \cite{B}).

In \cite{B}, Behrstock presented the first example of a $\mc{CFS}$ right-angled Coxeter group that contains a one-ended stable subgroup answering outstanding questions about stable subgroups and quasi-isometries between right-angled Artin groups and right-angled Coxeter groups. Using Theorem \ref{thm:quasiconvex_subgroups_of_RACG}, we can expand Behrstock's construction to  produce $\mc{CFS}$ right-angled Coxeter groups that contain any other right-angled Coxeter group as a {\sqc} subgroup. This shows that there is incredible diversity among the quasi-isometry types of $\mc{CFS}$ right-angled Coxeter groups.

\begin{prop}
\label{prop:quasiconvex_subgroups_of_CFS_coxeter_groups}
Any right-angled Coxeter group (resp. hyperbolic right-angled Coxeter group) is an infinite index {\sqc} subgroup (resp. stable subgroup) of a $\mathcal{CFS}$ right-angled Coxeter group.
\end{prop}

\begin{proof}
To prove the proposition we shall utilize a construction of certain $\mathcal{CFS}$ graphs described in \cite{B}. Let $\Omega_n$ be a graph with $2n$ vertices built in the following inductive way.  Let $\Omega_1$ be a pair of vertices $a_1$, $b_1$ with no edge between them. Given the graph $\Omega_{n-1}$, we obtain the graph $\Omega_n$ by adding a new pair of vertices $a_n$, $b_n$ to the graph $\Omega_{n-1}$ and adding four new edges, one connecting each of $\{a_{n-1}, b_{n-1}\}$ to each of $\{a_{n}, b_{n}\}$. In Figure~\ref{afourth} graph $\Gamma_1$ is exactly $\Omega_{13}$. For each integer $m\geq 2$ there is a sufficiently large $n$ such that the graph $\Omega_n$ contains $m$ vertices whose pairwise distances are at least $3$. 

Let $G_\Gamma$ be an arbitrary right-angled Coxeter group. We will construct a $\mathcal{CFS}$ right-angled Coxeter group $G_\Omega$ that contains $G_\Gamma$ as a {\sqc} subgroup. Let $m$ be a number of vertices of $\Gamma$. Choose a positive integer $n$ sufficient large so the graph $\Omega_n$ contains a set $S$ of $m$ vertices whose pairwise distance is at least $3$. We glue the graphs $\Gamma$ and $\Omega_n$ by identifying the vertex set of $\Gamma$ to $S$. Let $\Omega$ be the resulting graph. In Figure~\ref{afourth} graph $\Gamma_2$ is an example of graph $\Omega$ when $\Gamma$ is the 5-cycle graph and graph $\Gamma_3$ is another example of graph $\Omega$ when $\Gamma$ is the 4-cycle graph.

The graphs $\Omega$ and $\Omega_n$ have the same vertex set and \(\Omega_n^4 \subset \Omega^{4}\). Thus $\Omega$ is a $\mc{CFS}$ graph as $\Omega_n$ is a $\mathcal{CFS}$ graph.  Since the distance in $\Omega_n$ between any distinct vertices of $S$ is at least $3$, $\Gamma$ is an induced subgraph of $\Omega$ with the property that if $\Gamma$ contains two non-adjacent vertices of an induced 4--cycle $\sigma$, then $\Gamma$ contains all vertices of $\sigma$. Therefore, $G_\Gamma$ is a {\sqc} subgroup of $G_\Omega$ by Theorem \ref{thm:quasiconvex_subgroups_of_RACG}. If $G_\Gamma$ is a hyperbolic group, then it is a stable subgroup of $G_\Omega$. 
\end{proof}

In light of Proposition \ref{prop:quasiconvex_subgroups_of_CFS_coxeter_groups}, we believe that {\sqc} subgroups will play an important role in understanding the quasi-isometry classification of $\mc{CFS}$ right-angled Coxeter groups. In particular, it suggests that the  quasi-isometry classification of $\mc{CFS}$ right-angled Coxeter groups maybe no simpler than the quasi-isometry classification of all right-angled Coxeter groups.

We finish this section by illustrating the results of this section with  three $\mc{CFS}$ right-angled Coxeter groups whose quasi-isometry types can be distinguished utilizing their {\sqc} subsets.

\begin{exmp}\label{ex:qi_classifification}
Let $\Gamma_1$, $\Gamma_2$, and $\Gamma_3$ be the graphs in Figure~\ref{afourth}. All of the right-angled Coxeter groups $G_{\Gamma_1}$, $G_{\Gamma_2}$, and $G_{\Gamma_3}$ are $\mc{CFS}$, but no pair of them are quasi-isometric. By \cite{HH2017}, $G_{\Gamma_1}$ is quasi-isometric to a right-angled Artin group with connected defining graph. Thus, all of $G_{\Gamma_1}$'s non-coarsely covering {\sqc} subsets are quasi-trees. However, $G_{\Gamma_2}$ contains a one-ended hyperbolic {\sqc} subgroup (induced by the blue 5--cycle) and $G_{\Gamma_3}$ contain a virtually $\mathbb{Z}^2$ {\sqc} subgroup (induced by the red 4--cycle). The table below summarizes some of the differences between $G_{\Gamma_1}$, $G_{\Gamma_2}$, and $G_{\Gamma_3}$.
\end{exmp}

\begin{center}
\begin{tabular}{|m{3.4cm}|m{3.4cm}|m{3.4cm}|m{3.4cm}|}\hline
	& {$G_{\Gamma_1}$}& {\centering $G_{\Gamma_2}$}& {\centering $G_{\Gamma_3}$} \\ \hline
	Strongly $\mc{CFS}$& Yes & Yes & No\\ \hline
	Non-coarsely \newline covering \newline {\sqc} \newline subsets  & All quasi-trees & All hyperbolic. \newline Contains a one-ended  stable \newline subgroup. &    Contains a \newline {\sqc} \newline virtually $\mathbb{Z}^2$ \newline subgroup. \\ \hline
	Morse \newline boundary & Totally \newline disconnected & Contains a circle & Connectivity \newline unknown$^4$\\ \hline
	QI to a RAAG & Yes  & No & No\\ \hline
\end{tabular}
\end{center}
{\footnotesize{$^4$ Karrer has since shown that the Morse boundary of $G_{\Gamma_3}$ is totally disconnected \cite{Karrer_Totally_disconnected}.}}

\begin{figure}[H]
\begin{tikzpicture}[scale=.4]
	
	\draw (0,0.5) node[circle,fill,inner sep=.75pt, color=black](1){} -- (0.5,0) node[circle,fill,inner sep=.75pt, color=black](1){} -- (0,-0.5) node[circle,fill,inner sep=.75pt, color=black](1){} -- (-0.5,0) node[circle,fill,inner sep=.75pt, color=black](1){} -- (0,0.5);
	
	\draw (-0.5,0) node[circle,fill,inner sep=.75pt, color=black](1){} -- (0,1) node[circle,fill,inner sep=.75pt, color=black](1){} -- (0.5,0)  node[circle,fill,inner sep=.75pt, color=black](1){} -- (0,-1) node[circle,fill,inner sep=.75pt, color=black](1){} -- (-0.5,0);
	
	\draw (0,1) node[circle,fill,inner sep=.75pt, color=black](1){}  -- (1,0) node[circle,fill,inner sep=.75pt, color=black](1){}  -- (0,-1) node[circle,fill,inner sep=.75pt, color=black](1){} -- (-1,0) node[circle,fill,inner sep=.75pt, color=black](1){} -- (0,1);
	
	\draw (-1,0) node[circle,fill,inner sep=.75pt, color=black](1){} -- (0,2) node[circle,fill,inner sep=.75pt, color=black](1){} -- (1,0)node[circle,fill,inner sep=.75pt, color=black](1){}  -- (0,-2)node[circle,fill,inner sep=.75pt, color=black](1){}  -- (-1,0); 
	
	\draw (0,2) node[circle,fill,inner sep=.75pt, color=black](1){}  -- (2,0)node[circle,fill,inner sep=.75pt, color=black](1){}  -- (0,-2) node[circle,fill,inner sep=.75pt, color=black](1){} -- (-2,0)node[circle,fill,inner sep=.75pt, color=black](1){}  -- (0,2);
	
	\draw (-2,0) node[circle,fill,inner sep=.75pt, color=black](1){}  -- (0,4) node[circle,fill,inner sep=.75pt, color=black](1){}  -- (2,0) node[circle,fill,inner sep=.75pt, color=black](1){} -- (0,-4) node[circle,fill,inner sep=.75pt, color=black](1){} -- (-2,0);
	
	\draw (0,4) node[circle,fill,inner sep=.75pt, color=black](1){}  -- (4,0) node[circle,fill,inner sep=.75pt, color=black](1){}  -- (0,-4) node[circle,fill,inner sep=.75pt, color=black](1){} -- (-4,0) node[circle,fill,inner sep=.75pt, color=black](1){} -- (0,4);
	
	\draw (-4,0)node[circle,fill,inner sep=.75pt, color=black](1){}  -- (0,6) node[circle,fill,inner sep=.75pt, color=black](1){} -- (4,0) node[circle,fill,inner sep=.75pt, color=black](1){} -- (0,-6)node[circle,fill,inner sep=.75pt, color=black](1){}  -- (-4,0);
	
	\draw (0,6) node[circle,fill,inner sep=.75pt, color=black](1){} -- (6,0) node[circle,fill,inner sep=.75pt, color=black](1){} -- (0,-6) node[circle,fill,inner sep=.75pt, color=black](1){}  -- (-6,0) node[circle,fill,inner sep=.75pt, color=black](1){} -- (0,6);
	
	\draw (-6,0) node[circle,fill,inner sep=.75pt, color=black](1){} -- (0,8) node[circle,fill,inner sep=.75pt, color=black](1){} -- (6,0) node[circle,fill,inner sep=.75pt, color=black](1){} -- (0,-8) node[circle,fill,inner sep=.75pt, color=black](1){} -- (-6,0);
	
	\draw (0,8) -- (8,0) -- (0,-8) node[circle,fill,inner sep=.75pt, color=black](1){} -- (-8,0) node[circle,fill,inner sep=.75pt, color=black](1){} -- (0,8) node[circle,fill,inner sep=.75pt, color=black](1){};
	
	\draw (-8,0) node[circle,fill,inner sep=.75pt, color=black](1){} -- (0,10) node[circle,fill,inner sep=.75pt, color=black](1){}-- (8,0) node[circle,fill,inner sep=.75pt, color=black](1){} -- (0,-10) node[circle,fill,inner sep=.75pt, color=black](1){} -- (-8,0) node[circle,fill,inner sep=.75pt, color=black](1){};
	
	\node at (0,11) {$\Gamma_1$};

	\draw (0-10,0.5-20) node[circle,fill,inner sep=.75pt, color=black](1){} -- (0.5-10,0-20) node[circle,fill,inner sep=.75pt, color=black](1){} -- (0-10,-0.5-20) node[circle,fill,inner sep=.75pt, color=black](1){} -- (-0.5-10,0-20) node[circle,fill,inner sep=.75pt, color=black](1){} -- (0-10,0.5-20);
	
	\draw (-0.5-10,0-20) node[circle,fill,inner sep=.75pt, color=black](1){} -- (0-10,1-20) node[circle,fill,inner sep=.75pt, color=black](1){} -- (0.5-10,0-20)  node[circle,fill,inner sep=.75pt, color=black](1){} -- (0-10,-1-20) node[circle,fill,inner sep=.75pt, color=black](1){} -- (-0.5-10,0-20);
	
	\draw (0-10,1-20) node[circle,fill,inner sep=.75pt, color=black](1){}  -- (1-10,0-20) node[circle,fill,inner sep=.75pt, color=black](1){}  -- (0-10,-1-20) node[circle,fill,inner sep=.75pt, color=black](1){} -- (-1-10,0-20) node[circle,fill,inner sep=.75pt, color=black](1){} -- (0-10,1-20);
	
	\draw (-1-10,0-20) node[circle,fill,inner sep=.75pt, color=black](1){} -- (0-10,2-20) node[circle,fill,inner sep=.75pt, color=black](1){} -- (1-10,0-20) node[circle,fill,inner sep=.75pt, color=black](1){}  -- (0-10,-2-20)node[circle,fill,inner sep=.75pt, color=black](1){}  -- (-1-10,0-20); 
	
	\draw (0-10,2-20) node[circle,fill,inner sep=.75pt, color=black](1){}  -- (2-10,0-20)node[circle,fill,inner sep=.75pt, color=black](1){}  -- (0-10,-2-20) node[circle,fill,inner sep=.75pt, color=black](1){} -- (-2-10,0-20)node[circle,fill,inner sep=.75pt, color=black](1){}  -- (0-10,2-20);
	
	\draw (-2-10,0-20) node[circle,fill,inner sep=.75pt, color=black](1){}  -- (0-10,4-20) node[circle,fill,inner sep=.75pt, color=black](1){}  -- (2-10,0-20) node[circle,fill,inner sep=.75pt, color=black](1){} -- (0-10,-4-20) node[circle,fill,inner sep=.75pt, color=black](1){} -- (-2-10,0-20);
	
	\draw (0-10,4-20) node[circle,fill,inner sep=.75pt, color=black](1){}  -- (4-10,0-20) node[circle,fill,inner sep=.75pt, color=black](1){}  -- (0-10,-4-20) node[circle,fill,inner sep=.75pt, color=black](1){} -- (-4-10,0-20) node[circle,fill,inner sep=.75pt, color=black](1){} -- (0-10,4-20);
	
	\draw (-4-10,0-20)node[circle,fill,inner sep=.75pt, color=black](1){}  -- (0-10,6-20) node[circle,fill,inner sep=.75pt, color=black](1){} -- (4-10,0-20) node[circle,fill,inner sep=.75pt, color=black](1){} -- (0-10,-6-20)node[circle,fill,inner sep=.75pt, color=black](1){}  -- (-4-10,0-20);
	
	\draw (0-10,6-20) node[circle,fill,inner sep=.75pt, color=black](1){} -- (6-10,0-20) node[circle,fill,inner sep=.75pt, color=black](1){} -- (0-10,-6-20) node[circle,fill,inner sep=.75pt, color=black](1){}  -- (-6-10,0-20) node[circle,fill,inner sep=.75pt, color=black](1){} -- (0-10,6-20);
	
	\draw (-6-10,0-20) node[circle,fill,inner sep=.75pt, color=black](1){} -- (0-10,8-20) node[circle,fill,inner sep=.75pt, color=black](1){} -- (6-10,0-20) node[circle,fill,inner sep=.75pt, color=black](1){} -- (0-10,-8-20) node[circle,fill,inner sep=.75pt, color=black](1){} -- (-6-10,0-20);
	
	\draw (0-10,8-20) node[circle,fill,inner sep=.75pt, color=black](1){} -- (8-10,0-20)node[circle,fill,inner sep=.75pt, color=black](1){} -- (0-10,-8-20) node[circle,fill,inner sep=.75pt, color=black](1){} -- (-8-10,0-20) node[circle,fill,inner sep=.75pt, color=black](1){} -- (0-10,8-20) node[circle,fill,inner sep=.75pt, color=black](1){};
	
	\draw (-8-10,0-20) node[circle,fill,inner sep=.75pt, color=black](1){} -- (0-10,10-20) node[circle,fill,inner sep=.75pt, color=black](1){}-- (8-10,0-20) node[circle,fill,inner sep=.75pt, color=black](1){} -- (0-10,-10-20) node[circle,fill,inner sep=.75pt, color=black](1){} -- (-8-10,0-20) node[circle,fill,inner sep=.75pt, color=black](1){};
	
	\draw (0-10,0.5-20)[color=blue] node[circle,fill,inner sep=.75pt, color=black](1){} -- (1-10,0-20) node[circle,fill,inner sep=.75pt, color=black](1){}-- (0-10,4-20) node[circle,fill,inner sep=.75pt, color=black](1){} -- (-6-10,0-20) node[circle,fill,inner sep=.75pt, color=black](1){} -- (0-10,-10-20); 
	\draw (0-10,-10-20)[color=blue]  .. controls (2-10,-7-20) and (1-10, -2-20) ..  (0-10,0.5-20);
	
	\node at (0-10,11-20) {$\Gamma_2$};

	\draw (0+10,0.5-20) node[circle,fill,inner sep=.75pt, color=black](1){} -- (0.5+10,0-20) node[circle,fill,inner sep=.75pt, color=black](1){} -- (0+10,-0.5-20) node[circle,fill,inner sep=.75pt, color=black](1){} -- (-0.5+10,0-20) node[circle,fill,inner sep=.75pt, color=black](1){} -- (0+10,0.5-20);
	
	\draw (-0.5+10,0-20) node[circle,fill,inner sep=.75pt, color=black](1){} -- (0+10,1-20) node[circle,fill,inner sep=.75pt, color=black](1){} -- (0.5+10,0-20)  node[circle,fill,inner sep=.75pt, color=black](1){} -- (0+10,-1-20) node[circle,fill,inner sep=.75pt, color=black](1){} -- (-0.5+10,0-20);
	
	\draw (0+10,1-20) node[circle,fill,inner sep=.75pt, color=black](1){}  -- (1+10,0-20) node[circle,fill,inner sep=.75pt, color=black](1){}  -- (0+10,-1-20) node[circle,fill,inner sep=.75pt, color=black](1){} -- (-1+10,0-20) node[circle,fill,inner sep=.75pt, color=black](1){} -- (0+10,1-20);
	
	\draw (-1+10,0-20) node[circle,fill,inner sep=.75pt, color=black](1){} -- (0+10,2-20) node[circle,fill,inner sep=.75pt, color=black](1){} -- (1+10,0-20) node[circle,fill,inner sep=.75pt, color=black](1){}  -- (0+10,-2-20)node[circle,fill,inner sep=.75pt, color=black](1){}  -- (-1+10,0-20); 
	
	\draw (0+10,2-20) node[circle,fill,inner sep=.75pt, color=black](1){}  -- (2+10,0-20)node[circle,fill,inner sep=.75pt, color=black](1){}  -- (0+10,-2-20) node[circle,fill,inner sep=.75pt, color=black](1){} -- (-2+10,0-20)node[circle,fill,inner sep=.75pt, color=black](1){}  -- (0+10,2-20);
	
	\draw (-2+10,0-20) node[circle,fill,inner sep=.75pt, color=black](1){}  -- (0+10,4-20) node[circle,fill,inner sep=.75pt, color=black](1){}  -- (2+10,0-20) node[circle,fill,inner sep=.75pt, color=black](1){} -- (0+10,-4-20) node[circle,fill,inner sep=.75pt, color=black](1){} -- (-2+10,0-20);
	
	\draw (0+10,4-20) node[circle,fill,inner sep=.75pt, color=black](1){}  -- (4+10,0-20) node[circle,fill,inner sep=.75pt, color=black](1){}  -- (0+10,-4-20) node[circle,fill,inner sep=.75pt, color=black](1){} -- (-4+10,0-20) node[circle,fill,inner sep=.75pt, color=black](1){} -- (0+10,4-20);
	
	\draw (-4+10,0-20)node[circle,fill,inner sep=.75pt, color=black](1){}  -- (0+10,6-20) node[circle,fill,inner sep=.75pt, color=black](1){} -- (4+10,0-20) node[circle,fill,inner sep=.75pt, color=black](1){} -- (0+10,-6-20)node[circle,fill,inner sep=.75pt, color=black](1){}  -- (-4+10,0-20);
	
	\draw (0+10,6-20) node[circle,fill,inner sep=.75pt, color=black](1){} -- (6+10,0-20) node[circle,fill,inner sep=.75pt, color=black](1){} -- (0+10,-6-20) node[circle,fill,inner sep=.75pt, color=black](1){}  -- (-6+10,0-20) node[circle,fill,inner sep=.75pt, color=black](1){} -- (0+10,6-20);
	
	\draw (-6+10,0-20) node[circle,fill,inner sep=.75pt, color=black](1){} -- (0+10,8-20) node[circle,fill,inner sep=.75pt, color=black](1){} -- (6+10,0-20) node[circle,fill,inner sep=.75pt, color=black](1){} -- (0+10,-8-20) node[circle,fill,inner sep=.75pt, color=black](1){} -- (-6+10,0-20);
	
	\draw (0+10,8-20) node[circle,fill,inner sep=.75pt, color=black](1){} -- (8+10,0-20)node[circle,fill,inner sep=.75pt, color=black](1){} -- (0+10,-8-20) node[circle,fill,inner sep=.75pt, color=black](1){} -- (-8+10,0-20) node[circle,fill,inner sep=.75pt, color=black](1){} -- (0+10,8-20) node[circle,fill,inner sep=.75pt, color=black](1){};
	
	\draw (-8+10,0-20) node[circle,fill,inner sep=.75pt, color=black](1){} -- (0+10,10-20) node[circle,fill,inner sep=.75pt, color=black](1){}-- (8+10,0-20) node[circle,fill,inner sep=.75pt, color=black](1){} -- (0+10,-10-20) node[circle,fill,inner sep=.75pt, color=black](1){} -- (-8+10,0-20) node[circle,fill,inner sep=.75pt, color=black](1){};
	
	\draw (0+10,0.5-20)[color=red] node[circle,fill,inner sep=.75pt, color=black](1){} -- (1+10,0-20) node[circle,fill,inner sep=.75pt, color=black](1){}-- (0+10,4-20) node[circle,fill,inner sep=.75pt, color=black](1){} -- (-6+10,0-20) node[circle,fill,inner sep=.75pt, color=black](1){} -- (0+10,0.5-20) node[circle,fill,inner sep=.75pt, color=black](1){};

	\node at (0+10,11-20) {$\Gamma_3$};

\end{tikzpicture}

\caption{Three graphs $\Gamma_1$, $\Gamma_2$, and $\Gamma_3$ are all $\mathcal{CFS}$, but no pair of them are quasi-isometric.}
\label{afourth}
\end{figure}
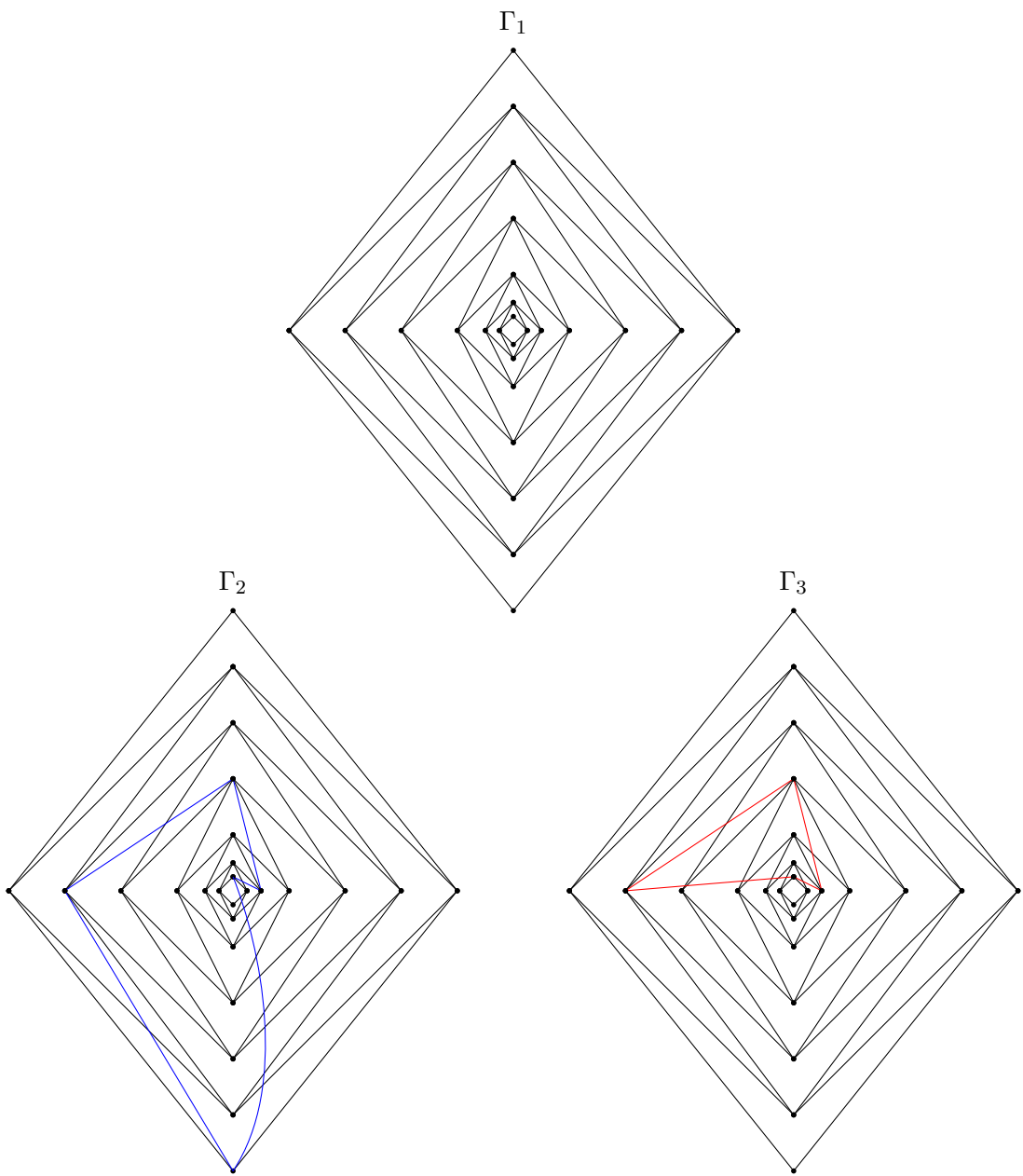

\section{Hyperbolically embedded subgroups of HHGs}\label{sec: Hyperbolically embedded subgroups of HHGs}

In this section, we utilize Theorem \ref{thm:classification_of_quasiconvex_subset} to prove the following classification of hyperbolically embedded subgroups of hierarchically hyperbolic groups. As our proof does not directly utilize the definition of hyperbolically embedded we shall omit the definition here and direct the curious reader to \cite{DGO}.

\begin{thm}\label{thm: classification hyp. emb in HHG}
Let \(G\) be a hierarchically hyperbolic group and let \(\{H_i\}\) be a finite collection of subgroups. Then the following are equivalent:
\begin{enumerate}
	\item The collection \(\{H_i\}\) is hyperbolically embedded in \(G\).
	\item The collection \(\{H_i\}\) is  almost malnormal and each $H_i$ is {\sqc}.
	
\end{enumerate}
\end{thm}

Combining work of Dahmani-Guirardel-Osin \cite{DGO} and Sisto \cite{Sisto_quasiconveity_of_hyperbolically_embedded}, the implication \((1) \implies (2)\) holds for all finitely generated groups. To see that the converse does not hold in general, consider a non-virtually cyclic lacunary hyperbolic group $G$ where every proper subgroup is infinite cyclic and {\sqc} (the existence of such a group is shown in \cite[Theorem1.12]{OOS_Lacunary}). If $I$ is a proper subgroup of $G$, then by \cite[Theorem 1.2]{Tran2017}, $I$ has finite index in its commensurator $H$. Thus $H$ is a proper, infinite, almost malnormal, {\sqc} subgroup of $G$. However, $H$ cannot be hyperbolically embedded  as $G$ does not contain any non-abelian free subgroups and thus fails to be acylindrically hyperbolic (see \cite{Osin_acyl_hyp, DGO}).

Despite this failure in general, Genevois showed that in the setting of CAT(0) cubical groups, $(2)$ does imply $(1)$ (\cite[Theorem 6.31]{Genevois_Hyp_in_CAT(0)}). Genevois employees a combination of the Bestvina-Bromberg-Fujiwara construction (\cite[Theorems A, B]{BBF15}) with some work of Sisto (\cite[Theorems 6.3, 6.4]{SistoOnMetricRelative}) that is summarized in the following sufficient conditions for a collection of subgroups to be hyperbolically embedded.
\begin{thm}[\cite{BBF15, SistoOnMetricRelative}]\label{thm:short_BBF_construction}
Let $G$ be a finitely generated group and $\mc{Z}$ be the collection of all (left) cosets of a finite collection of finitely generated subgroups $\{H_i\}$ in $G$. Fix a finite  generating set \(S\) for \(G\) such that \(H_i = \langle H_i \cap S \rangle\) for all \(i\).  Suppose for every \(Z_1 \neq Z_2 \in \mc{Z}\) we are given a subset \(\proj_{Z_1}(Z_2) \subseteq Z_1\) and for \(Z_1, Z_2, Z_3 \in \mc{Z}\) define \(d^\proj_{Z_3}(Z_1,Z_2) = \diam_{Z_3} \left(\proj_{Z_3}(Z_1) \cup \proj_{Z_3}(Z_2)\right)\). The collection $\{H_i\}$ is hyperbolically embedded in $G$ if there exists \(C> 0\) such that the following hold:
\begin{enumerate}
	\item[(P0)] For all \(Z_1 \neq Z_2\), \(\diam(\proj_{Z_1}(Z_2)) \leq C\).
	\item[(P1)] For any triple \(Z_1,Z_2,Z_3 \in \mc{Z}\) of distinct elements, at most one of the three numbers \(d^\proj_{Z_1}({Z_2},{Z_3}), d^\proj_{Z_2}({Z_1},{Z_3}), d^\proj_{Z_3}({Z_1},{Z_2})\) is greater than \(C\).
	\item[(P2)] For any \(Z_1,Z_2\in \mc{Z}\), the set
	\[\{Z \in \mc{Z} \mid d^{\proj}_{Z}(Z_1,Z_2)>C\}\] is finite.
	\item[(P3)] For all $g \in G$, \(d^\proj_{gZ_1} (gZ_2, gZ_3)= d^\proj_{Z_1}(Z_2,Z_3)\) for any \(Z_1,Z_2,Z_3 \in \mc{Z}\).
\end{enumerate}
\end{thm}

As Genevois does in the cubical case, we shall show that an almost malnormal collection of {\sqc} subgroups of an HHG satisfies (P0) - (P3) of Theorem \ref{thm:short_BBF_construction}. The bulk of that work is in Proposition \ref{prop:BBF_axioms_for_malnormal_quasiconvex_subgroups}, which we will state and prove after collecting a few preliminary lemmas.

\begin{lem}\label{lem:malnormal_finite_intersection}
Let $\{H_1,\cdots, H_n\}$ be an almost malnormal collection of subgroups of a finitely generated group $G$ and $B\geq 0$. For all $g_1,g_2\in G$, if $g_1H_i \neq g_2H_j$, then $\diam({N}_B(g_1H_i) \cap {N}_B(g_2H_j))$ is finite.
\end{lem}

\begin{proof}
The conclusion follows directly from \cite[Proposition 9.4]{Hruska10} and the definition of almost malnormal.
\end{proof}

The next two lemmas tell us that a hierarchically quasiconvex subset coarsely intersects a {\sqc} subset whenever the image under the gate map is large. Further, the diameter of this coarse intersection is proportional to the diameter of the gate. In addition to being key components in our proof of Theorem \ref{thm: classification hyp. emb in HHG}, these lemmas can also be interpreted as additional generalizations of the bounded geodesic image property of {\sqc} subsets of hyperbolic spaces.
\begin{lem}\label{lem:contracting_dichotomy}
Let $(\mc{X},\mf{S})$ be an HHS with the bounded domain dichotomy, $A\subseteq \mc{X}$ be $k$--hierarchically quasiconvex subset, and $Y\subseteq \mc{X}$ be $Q$--{\sqc}. There exists $r>1$ depending on $Q$ and $k$ such that if $\diam_\mc{X}(\gate_Y(A)) > r$, then $d_\mc{X}(a,\gate_{Y}(a)) < r$ for each $a \in \gate_A(Y)$. 
\end{lem}

\begin{proof}
By Proposition \ref{prop: quasiconvex implies HQC}, there exists $k'$ such that both $A$ and $Y$ are $k'$--hierarchically quasiconvex. Recall that for each point \(x \in \X\) and \(U \in \mf{S}\), the distance in \(CU\) between \(\gate_Y(x)\) and the closest point projection of \(\pi_U(x)\) onto \(\pi_U(Y)\) is uniformly bounded by some \(\epsilon >1\). Let $K\geq \epsilon$ be such that $Y$ has the $K$--orthogonal projection dichotomy and that $K$ is larger than the constant from the bridge theorem (Theorem \ref{thm:parallel_gates}) determined by $k'$. Define $\mc{H} = \{U \in \mf{S} : \diam\bigl(\pi_U(\gate_{Y}(A))\bigr) > 2K\}$. By the uniqueness axiom (Axiom \ref{item:dfs_uniqueness}), there exists $C$ such that if $\diam(\gate_Y(A))>C$, then $\mc{H} \neq \emptyset$. Assume $\diam(\gate_Y(A)) > C$ and let $a \in \gate_A(Y)$. By item (\ref{item:df_bridge}) of the bridge theorem, $\relevant_{2K}(a,\gate_{Y}(a)) \subseteq \mc{H}^{\perp}$. Suppose for the purposes of contradiction that $V\in \relevant_{2K}(a,\gate_{Y}(a))$. Thus, there must exist $H \in \mc{H}$ with $V\perp H$. By Theorem \ref{thm:classification_of_quasiconvex_subset}, $CH \subseteq N_{K}(\pi_H(Y))$ and $CV \subseteq N_K(\pi_V(Y))$ which implies that $d_V(a,\gate_{Y}(a))<K + \epsilon < 2K$. However, this contradicts $V \in \relevant_{2K}(a,\gate_{Y}(a))$. Hence, $\relevant_{2K}(a,\gate_{Y}(a)) = \emptyset$, and by the distance formula (Theorem \ref{thm:distance_formula}), there exists $K'$ depending only on $K$ (and thus only on $Q$ and $\kappa_1$) such that $d_{\mc{X}}(a,\gate_{Y}(a))<K'$.  The conclusion follows by choosing $r = \max\{K',C\}$.
\end{proof}

\begin{lem}\label{lem:gate_is_coarse_intersection}
Let $(\mc{X},\mf{S})$ be an HHS with the bounded domain dichotomy, $A\subseteq \mc{X}$ be a  $k$--hierarchically quasiconvex subset and $Y\subseteq \mc{X}$ be $Q$--{\sqc}. There exists $r>1$ depending on $k$ and $Q$ such that for all $D\geq r$ if $\diam(\gate_Y(A))>r$, then there exists $K\geq 1$ depending on $k$, $D$, and $Q$ such that
\[ \diam(N_D(A)\cap N_D(Y)) \stackrel{1,K}{\asymp} \diam(\gate_Y(A)).\]
\end{lem}

\begin{proof}
Let $r$ be the constant given by Lemma \ref{lem:contracting_dichotomy} and suppose $ \diam(\gate_Y(A))>r$. 
Thus, for $D\geq r$, $\diam(N_D(A)\cap N_D(Y)) \neq \emptyset$. First consider $x,y \in N_D(A)\cap N_D(Y)$. Let $x',y'\in A$ be points such that $d_\mc{X}(x,x') \leq D$ and $d_{\mc{X}}(y,y') \leq D$. By Lemma \ref{lem: closest_point_proj} and the fact that $x,y \in N_D(Y)$, there exists $K'$ depending on $Q$ such that
\[d_\mc{X}(x,\gate_Y(x')) \leq 4DK'; \hspace{2cm} d_\mc{X}(y,\gate_Y(y')) \leq 4DK'.\]
Hence we have 
\[d_\mc{X}(x,y) \leq  d_\mc{X}(\gate_Y(x'),\gate_Y(y'))+ 8DK'\]
which shows
\[\diam(N_D(A)\cap N_D(Y)) \leq \diam(\gate_Y(A))+8DK'.\]

For the inequality $\diam(\gate_Y(A)) \preceq \diam(N_D(A)\cap N_D(Y))$, Lemma \ref{lem:contracting_dichotomy} provides $\gate_Y\bigl(\gate_A(Y)\bigr) \subseteq N_D(A)\cap N_D(Y)$ and the bridge theorem (Theorem  \ref{thm:parallel_gates}) says there exists $K''$ depending on $k$ and $Q$ such that $\gate_Y(A) \subseteq N_{K''}(\gate_Y(\gate_A(Y)))$. Thus we have 
\[\diam(\gate_Y(A)) \leq \diam(\gate_Y(\gate_A(Y)))+2K'' \leq \diam(N_D(A)\cap N_D(Y)) +2K''\]
and we are finished by choosing $K = \max\{2K'',6DK'+3K'\}.$
\end{proof}

We now prove that the cosets of a collection of almost malnormal, {\sqc} subgroups of an HHG  satisfy (P0)-(P2) of Theorem \ref{thm:short_BBF_construction} when $\proj_{Z_1}(Z_2)$ is defined by the gate map. This is the main tool for the proof of Theorem \ref{thm: classification hyp. emb in HHG}.

\begin{prop}\label{prop:BBF_axioms_for_malnormal_quasiconvex_subgroups}
Let $(G,\mf{S})$ be an HHG and  $d(\cdot,\cdot)$ denote the distance in the word metric on $G$ with respect to some fixed finite generating set. If $\{H_1,\cdots, H_n\}$ is a collection of $Q$--{\sqc}, almost malnormal subgroups of $G$ and $\mc{Z}$ is the collection of all left cosets of the $H_i$, then there exists $C>0$ such that for all distinct $Z_1,Z_2,Z_3 \in \mc{Z}$ we have:
\begin{enumerate}[(1)]
	\item $\diam(\gate_{Z_1}(Z_2)) \leq C$;
	\item  if $d(\gate_{Z_3}(Z_1), \gate_{Z_3}(Z_2)) >C$, then  $d(\gate_{Z_2}(Z_1), \gate_{Z_2}(Z_3)) <C$ and $d(\gate_{Z_1}(Z_2), \gate_{Z_1}(Z_3)) <C$;
	\item $\{Z\in\mc{Z} : d(\gate_{Z}(Z_1), \gate_{Z}(Z_2)) > C\}$ has only a finite number of elements.
\end{enumerate}
\end{prop}

\begin{proof}

We will prove each of the three assertions individually. Before beginning, we remind the reader that all hierarchically hyperbolic groups satisfy the bounded domain dichotomy and that every element of $\mc{Z}$ is $k$--hierarchically quasiconvex for some $k$ depending only on $Q$. \\

\noindent \textbf{Assertion (1):} there exists $C_1>0$ such that $\diam(\gate_{Z_1}(Z_2)) \leq C_1$ for all $Z_1,Z_2 \in \mc{Z}$.

\begin{proof}
	Let $r>1$ be the constant from Lemma \ref{lem:gate_is_coarse_intersection} for $Q$ and define \[F = \{gH_i \in \mc{Z} : gH_i \cap B_{r}(e) \neq \emptyset\}\] where $B_r(e)$ is the ball of radius $r$ around the identity in $G$. Since $F$ is a finite set, Lemma \ref{lem:malnormal_finite_intersection} provides a uniform number $D_1$ such that $\diam(N_r(gH_i)\cap N_r(H_j)) \leq D_1$ for any distinct $gH_i,H_j \in F$. By Lemma \ref{lem:gate_is_coarse_intersection}, there exists $D_2$ depending on $Q$ such that $\diam(\gate_{H_j}(gH_i))\leq D_2$ where $gH_i \neq H_j$ are elements in $F$.
	
	We now prove that there is a uniform constant $C_1$ such that for each pair of distinct cosets
	\(g_1H_i\) and \(g_2H_j\) we have \[\diam(\gate_{g_1H_i}(g_2H_j))\leq C_1.\] If \(\diam(\gate_{g_1H_i}(g_2H_j))\leq r\), then we are done. Otherwise, by Lemma \ref{lem:contracting_dichotomy}, there are elements \(h_i \in H_i\) and \(h_j\in H_j\) such that \(d_G(g_1h_i, g_2h_j) < r\). This implies that 
	$h_i^{-1}g_1^{-1}g_2H_j$ is an element in $F$ and $h_i^{-1}g_1^{-1}g_2H_j\neq H_i$. Therefore,  \(\diam(\gate_{H_i}(h_i^{-1}g_1^{-1}g_2H_j))\leq D_2\). Thus, by the coarse equivariance of the gate maps (Lemma \ref{lem:coarse_equiv_of_gate}), the diameter of $\gate_{g_1H_i}(g_2H_j)$ is bounded above by a uniform number $C_1$. 
\end{proof}

\noindent \textbf{Assertion (2):} there exists $C_2>0$ such that for all $Z_1,Z_2,Z_3 \in \mc{Z}$, if $d(\gate_{Z_3}(Z_1), \gate_{Z_3}(Z_2)) >C_2$, then \[d(\gate_{Z_2}(Z_1), \gate_{Z_2}(Z_3)) <C_2 \text{ and } d(\gate_{Z_1}(Z_2), \gate_{Z_1}(Z_3)) <C_2.\]
\begin{proof}

	Fix $\theta \geq \theta_0$. Let $Z_1,Z_2,Z_3 \in \mc{Z}$ and $B = H_\theta(\gate_{Z_2}(Z_1) \cup \gate_{Z_1}(Z_2))$. We remind the reader that they should view $B$ as a bridge between $Z_1$ and $Z_2$. 
	Our goal is to show that there exists $b \in B$ such that $d(b,\gate_{Z_3}(b))$ is uniformly bounded. From this our conclusion will follow from the coarse Lipschitzness of the gate map. 
	
	By Assertion (1), $\gate_{Z_3}(Z_1),\gate_{Z_3}(Z_2)$ are uniformly coarsely contained in $\gate_{Z_3}(B)$.  Since the gate map is coarsely Lipschitz we have
	\[\diam(\gate_{Z_3}(B)) \succeq d(\gate_{Z_3}(Z_1),\gate_{Z_3}(Z_2))\] 
	with constants depending only on $Q$. Let $r$ be the constant from Lemma \ref{lem:contracting_dichotomy} with $A=B$ and $Y = Z_3$ and suppose $d(\gate_{Z_3}(Z_1),\gate_{Z_3}(Z_2))$ is large enough that $\diam(\gate_{Z_3}(B)) > r$. By Lemma \ref{lem:contracting_dichotomy}, there exists $b \in B$ such that $d(b,Z_3) <r$.
	
	By Lemma \ref{lem:gate_of_bridge}, we have that $\gate_{Z_2}(Z_1)$ is uniformly coarsely equal to $\gate_{Z_2}(B)$ in particular $\gate_{Z_2}(b)$ is uniformly coarsely contained in $\gate_{Z_2}(Z_1)$. Since the gate maps are uniformly coarsely Lipschitz and $d(b,Z_3) <r$, we have that $d(\gate_{Z_2}(Z_3),\gate_{Z_2}(Z_1)) < C_2$.  By switching the roles of $Z_1$ and $Z_2$, we get $d(\gate_{Z_1}(Z_3),\gate_{Z_1}(Z_2)) < C_2$.
\end{proof}

\noindent  \textbf{Assertion (3):} there exists $C_3>0$ such that for all $Z_1,Z_2 \in \mc{Z}$, the set $\{Z\in\mc{Z} : d_\mc{X}(\gate_{Z}(Z_1), \gate_{Z}(Z_2)) > C_3\}$ has only a finite number of elements.

\begin{proof}
	Let $Z_1,Z_2 \in \mc{Z}$. Fix $\theta\geq \theta_0$ and let $B= H_\theta(\gate_{Z_2}(Z_1) \cup \gate_{Z_1}(Z_2))$. By the bridge theorem, we have that \(B\) is coarsely equals to the product of \(\gate_{Z_1}(Z_2) \times H_\theta (a,b)\), where \(a \in \gate_{Z_1}(Z_2)\) and \(b= \gate_{Z_2}(a)\). 
	By Assertion (1), the gate \(\gate_{Z_1}(Z_2)\) has uniformly bounded diameter.  By Proposition \ref{prop:hulls_finite_case}, there exists $\lambda\geq \lambda_0$, such that $H_\theta(a,b)$ is contained in $\Path_{\lambda}^1(a,b)$, the set of \(\lambda\)--hierarchy paths between \(a\) and \(b\). Since the distance between \(a\) and \(b\) is finite, so is the diameter of $\Path_{\lambda}^1(a,b)$. Therefore \(H_\theta (a,b)\) has bounded diameter and so does the set $B =H_\theta(\gate_{Z_2}(Z_1) \cup \gate_{Z_1}(Z_2))$. Since $G$ is locally finite, $B$ can contain only a finite number of elements of $G$.
	
	Let $r$ be as in Lemma \ref{lem:contracting_dichotomy}. Since $\gate_{Z_2}(Z_1),\gate_{Z_1}(Z_2) \subseteq B$, for any \(Z \in \mc{Z}\) with $d(\gate_Z(Z_1), \gate_Z(Z_2))$ larger than $r$ we have \(\diam(\gate_{Z}(B))>r\). Thus every such \(Z\) intersects the \(r\)--neighborhood of \(B\). 
	By locally finiteness of \(G\), we obtain that \(N_r(B)\) contains a finite number of element of $G$. Since the elements of \(\mc{Z}\) are cosets of finitely many subgroups, every point of \(N_r(B)\) can belong to uniformly finitely many elements of \(\mc{Z}\), which concludes the proof of Assertion (3).
\end{proof}

\noindent Proposition \ref{prop:BBF_axioms_for_malnormal_quasiconvex_subgroups} now holds by taking $C = \max\{C_1,C_2,C_3\}$.

\end{proof}

We now have all the ingredients needed to give the proof of Theorem \ref{thm: classification hyp. emb in HHG}.

\begin{proof}[Proof of Theorem {\ref{thm: classification hyp. emb in HHG}}]
Recall, we need to show that if \(G\) is a hierarchically hyperbolic group and \(\{H_i\}\) a finite almost malnormal  collection of {\sqc} subgroups, then \(\{H_i\} \) is hyperbolically embedded in \(G\). In particular, we shall show that the left cosets of the $H_i$'s satisfy the requirements of Theorem \ref{thm:short_BBF_construction}. Since each \(H_i\) is a {\sqc} subgroup of \(G\), by \cite[Theorem 1.2]{Tran2017} we have that they are all finitely generated. Let $S$ be a finite generating set for $G$ such that for each $i$, $H_i \cap S$ generates $H_i$. As before, let \(\mc{Z}\) be the set of all left cosets of \(\{H_i\}\). For every pair of distinct \(Z_1, Z_2 \in \mc{Z}\) we want to define a set \(\proj_{Z_1}(Z_2)\) that satisfies (P0) - (P3) of Theorem \ref{thm:short_BBF_construction}. If we define \(\proj_{Z_1}(Z_2)\)  as \(\gate_{Z_1}(Z_2)\), Proposition \ref{prop:BBF_axioms_for_malnormal_quasiconvex_subgroups} provides that (P0) - (P2) will be satisfied. However, since the gate maps are only coarsely equivariant, condition (P3) may not hold.

Thus, for \(Z_1 \neq Z_2\) define 
\[\proj_{Z_1}(Z_2) = \bigcup_{g \in G} g^{-1}\gate_{gZ_1}(gZ_2).\]
By construction we have that \(\proj_{gZ_1}(gZ_2) = g(\proj_{Z_1}(Z_2))\) and thus (P3) holds. Since  $\proj_{Z_1}(Z_2)$ and  \(\gate_{Z_1}(Z_2))\) uniformly coarsely coincide by the coarse equivariance of the gates maps (Lemma \ref{lem:coarse_equiv_of_gate}),  (P0) - (P2) are  satisfied as a corollary of Proposition \ref{prop:BBF_axioms_for_malnormal_quasiconvex_subgroups}. Hence, the collection $\{H_i\}$ is hyperbolically embedded in $G$ by Theorem \ref{thm:short_BBF_construction}.
\end{proof}
Our method of proof for Theorem \ref{thm: classification hyp. emb in HHG} relies in a fundamental way upon the coarse equivariance of the gate map.  If the group $G$ has an HHS structure, but not an HHG structure, then the gate map need not be coarsely equivariant.  In particular, Theorem \ref{thm: classification hyp. emb in HHG} does not (currently) apply to the fundamental groups of non-flip graph manifolds and thus we have the following interesting case of Question \ref{ques:hyperbolically_embedded}.

\begin{ques}
If $M$ is a non-flip graph manifold and $\{H_i\}$ is a finite, almost malnormal collection of {\sqc} subgroups of $\pi_1(M)$, is $\{H_i\}$ hyperbolically embedded in $\pi_1(M)$?
\end{ques}

\section*{Appendix A. Subsets with arbitrary reasonable lower relative divergence}

The  proposition in this appendix utilizes the notion of asymptotic equivalence between families of functions. We will present the definition in the specific case we need and direct the reader to \cite[Section 2]{Tran2015} for the more general case. 
\begin{defnappend}
Let $f$ and $g$ be two functions from $[0,\infty)$ to $[0,\infty)$. The function $f$ is \emph{dominated by the function $g$} 
if there are positive constants $A$, $B$, $C$ and $D$ such that $f(r)\leq Ag(Br)+Cr$ for all $r>D$. Two functions $f$ and $g$ are \emph{equivalent} if $f$ is dominated by $g$ and vice versa. 

Let $X$ be a geodesic metric space and $\{\sigma_\rho^n\} = div(X,Y)$ be the lower relative divergence of $X$ with respect to some subset $Y \subseteq X$. We say $div(X,Y)$ is \emph{equivalent} to a function  $f \colon [0,\infty) \rightarrow [0,\infty)$ if there exist $L\in (0,1]$ and positive integer $M$ such that $\sigma_{L\rho}^{Mn}$ is equivalent to $f$ for all $\rho\in (0,1]$ and $n\geq 2$.
\end{defnappend}

\begin{propappend}
\label{prop:lower_relative_div_example}
Let $f\colon [0, \infty)\to [0, \infty)$ be a non-decreasing function, and assume that there is a positive integer $r_0$ such that $f(r)\geq r$ for each $r>r_0$.  There is a geodesic space $X$ with a subspace $Y$ such that the lower relative divergence $div(X,Y)$ is equivalent to $f$. 
\end{propappend}

\begin{proof}
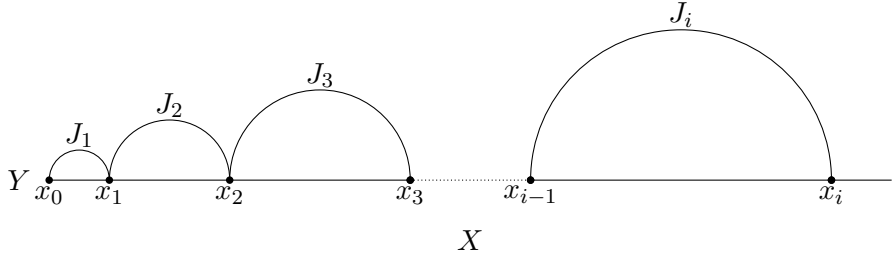
\begin{figure}[H]
	\begin{tikzpicture}[scale=.8]
		
		\draw (1,0) node[circle,fill,inner sep=1pt, color=black](1){} -- (2,0) node[circle,fill,inner sep=1pt, color=black](1){}-- (4,0) node[circle,fill,inner sep=1pt, color=black](1){}-- (7,0) node[circle,fill,inner sep=1pt, color=black](1){}; 
		
		\draw[densely dotted] (7,0) node[circle,fill,inner sep=1pt, color=black](1){} -- (9,0) node[circle,fill,inner sep=1pt, color=black](1){};
		
		\draw (9,0) node[circle,fill,inner sep=1pt, color=black](1){} -- (14,0) node[circle,fill,inner sep=1pt, color=black](1){};
		
		\draw (14,0) node[circle,fill,inner sep=1pt, color=black](1){} -- (15,0);

		\draw (2,0) arc (0:180:0.5);
		
		\draw (4,0) arc (0:180:1);
		
		\draw (7,0) arc (0:180:1.5);
		
		\draw (14,0) arc (0:180:2.5);

		\node at (1,-0.25) {$x_0$};
		
		\node at (2,-0.25) {$x_1$};
		
		\node at (4,-0.25) {$x_2$};
		
		\node at (7,-0.25) {$x_3$};
		
		\node at (9,-0.25) {$x_{i-1}$};
		
		\node at (14,-0.25) {$x_i$};

		\node at (1.5,0.75) {$J_1$};
		
		\node at (3,1.25) {$J_2$};
		
		\node at (5.5,1.75) {$J_3$};
		
		\node at (11.5,2.75) {$J_i$};

		\node at (0.5,0) {$Y$};
		
		\node at (8,-1) {$X$};

	\end{tikzpicture}
	
	\caption{By controlling the length of each arc $J_i$ we can get the desired lower relative divergence of the geodesic space $X$ with respect to the subspace $Y$.}
	\label{asecond}
\end{figure}
Let $Y$ be a ray with initial point $x_0$. Let $(x_i)$ be the sequence of points along $Y$ such that for each $i\geq 1$ the distance $d_Y(x_{i-1},x_i)=i$ and we connect each pair $(x_{i-1},x_i)$ by a segment $J_i$ of length $f(i)$ (see Figure~\ref{asecond}). Let $X$ be the resulting geodesic space and $div(X,Y) = \{\sigma^n_{\rho}\}$.  We shall show that $div(X,Y)$ is equivalent to $f$. 

We first prove that for all $n\geq 3$ and $\rho \in (0,1]$, $f$ dominates $\sigma^n_{\rho}$ by showing that $\sigma^n_{\rho}(r)\leq f((n+3)r)$ for each $r>r_0$. Let $i_0$ be a smallest integer that is greater or equal to $(n+2)r$. Let $x$ and $y$ be two points in the segment $J_{i_0}$ such that $d(x_{i_0-1}, x)=d(x_{i_0},y)=r$. Both $x$ and $y$ belong to $\partial N_r(Y)$. Moreover, the subpath $\alpha$ of $J_{i_0}$ connecting $x$ and $y$ lies outside the $r$--neighborhood of $Y$, and the length of $\alpha$ is exactly is $f(i)-2r$. Therefore, $d(x,y)=\min\{i_0+2r,f(i_0)-2r\}$. Hence $d(x,y)\geq nr$ as \[f(i_0)-2r\geq f((n+2)r)-2r\geq (n+2)r-2r=nr\] and \[i_0+2r\geq (n+4)r \geq nr.\] Since $\alpha$ is the unique path outside the $\rho r$--neighborhood of $Y$ connecting $x$ and $y$, we have 
\[\sigma^n_{\rho}(r)\leq d_{\rho r}(x,y)=f(i_0)-2r\leq f(i_0).\] 
Since $i_0\leq (n+2)r+1\leq (n+3)r$ and $f$ is non-decreasing, $f(i_0)\leq f((n+3)r)$. Thus,  $\sigma^n_{\rho}(r)\leq f((n+3)r)$ which implies that $\sigma^n_{\rho}$ is dominated by $f$.

Now we prove that for all $n\geq 3$ and $\rho\in (0,1]$, $\sigma^n_{\rho}$ dominates $f$ by showing that $\sigma^n_{\rho}(r)\geq f(r)-2r$ for each $r>r_0$. Let $u$ and $v$ be an arbitrary points in $\partial N_r(Y)$ such that $d(u,v)\geq nr$ and there is a path outside the $r$--neighborhood of $Y$ connecting $u$ and $v$. Therefore, $u$ and $v$ must lies in some segment $J_{i_1}$. We can assume that $d(u,x_{i_1-1})=d(v,x_{i_1})=r$. Therefore, \[i_1\geq d(x_{i_1-1},x_{i_1})\geq d(u,v)-2r\geq nr-2r\geq r.\]
This implies that $f(i_1)\geq f(r)$ since $f$ is non-decreasing. Since the subpath $\beta$ of $J_{i_1}$ connecting $u$ and $v$ is the unique path outside the $\rho r$--neighborhood of $Y$ connecting these points, we have
\[d_{\rho r}(u,v)=f(i_1)-2r\geq f(r)-2r.\]
Therefore, $\sigma^n_{\rho}(r)\geq f(r)-2r$ which implies that $\sigma^n_{\rho}$ dominates $f$. Thus, the lower relative divergence $div(X,Y)$ is equivalent to $f$.
\end{proof}

\bibliographystyle{alpha}
\bibliography{Tran}
\end{document}